\documentclass[a4paper,10pt,reqno]{amsart}
\usepackage{dsfont,marginnote,amsmath,amsfonts,amscd,amssymb,graphicx,mathrsfs,eufrak}
\usepackage[dvips,all,arc,curve,color,frame]{xy}
\usepackage[usenames]{color}

\usepackage[colorlinks]{hyperref}
\usepackage{tikz,mathrsfs}
\usetikzlibrary{arrows,decorations.pathmorphing,decorations.pathreplacing,positioning,shapes.geometric,shapes.misc,decorations.markings,decorations.fractals,calc,patterns}

\tikzset{>=stealth',
        cvertex/.style={circle,draw=black,inner sep=1pt,outer sep=3pt},
        vertex/.style={circle,fill=black,inner sep=1pt,outer sep=3pt},
        star/.style={circle,fill=yellow,inner sep=0.75pt,outer sep=0.75pt},
        tvertex/.style={inner sep=1pt,font=\scriptsize},
        gap/.style={inner sep=0.5pt,fill=white}}

\newtheorem{thm}{Theorem}[section]
\newtheorem{prop}[thm]{Proposition}
\newtheorem{lemma}[thm]{Lemma}
\newtheorem{defin}[thm]{Definition}
\newtheorem{cor}[thm]{Corollary}
\newtheorem{summary}[thm]{Summary}

\theoremstyle{definition} 
 
\newtheorem{example}[thm]{Example}

\newtheorem{remark}[thm]{Remark}
\newtheorem{question}[thm]{Question}
\newtheorem{conj}[thm]{Conjecture}

\numberwithin{equation}{section}

\newcommand{\C}[1]{\mathbb{C}^{#1}}
\newcommand{\m}{\mathfrak{m}}

\newcommand{\p}{\mathfrak{p}}

\newcommand{\s}[1]{\mathscr{#1}}

\renewcommand{\u}[1]{\underline{#1}}

\renewcommand{\t}[1]{\textnormal{#1}}

\newcommand{\looptop}[2]{\xy \SelectTips{cm}{10}
\POS(0,0) \endxy}

\def\RHom{\mathop{\rm {\bf R}Hom}\nolimits}

\def\op{\mathop{\rm op}\nolimits}
\def\GL{\mathop{\rm GL}\nolimits}
\def\SL{\mathop{\rm SL}\nolimits}
\def\CM{\mathop{\rm CM}\nolimits}

\def\depth{\mathop{\rm depth}\nolimits}
\def\fl{\mathop{\sf fl}\nolimits}
\def\hgt{\mathop{\rm ht}\nolimits}
\def\mod{\mathop{\rm mod}\nolimits}

\def\coh{\mathop{\rm coh}\nolimits}
\def\Mod{\mathop{\rm Mod}\nolimits}
\def\refl{\mathop{\rm ref}\nolimits}
\def\proj{\mathop{\rm proj}\nolimits}

\def\pd{\mathop{\rm proj.dim}\nolimits}
\def\id{\mathop{\rm inj.dim}\nolimits}
\def\Hom{\mathop{\rm Hom}\nolimits}
\def\End{\mathop{\rm End}\nolimits}
\def\Ext{\mathop{\rm Ext}\nolimits}
\def\Tor{\mathop{\rm Tor}\nolimits}
\def\Tr{\mathop{\rm Tr}\nolimits}
\def\add{\mathop{\rm add}\nolimits}
\def\Cok{\mathop{\rm Cok}\nolimits}
\def\Ker{\mathop{\rm Ker}\nolimits}
\def\ker{\mathop{\rm ker}\nolimits}
\def\Im{\mathop{\rm Im}\nolimits}
\def\Sing{\mathop{\rm Sing}\nolimits}
\def\Supp{\mathop{\rm Supp}\nolimits}
\def\Ass{\mathop{\rm Ass}\nolimits}

\def\Spec{\mathop{\rm Spec}\nolimits}
\def\Max{\mathop{\rm Max}\nolimits}
\def\gl{\mathop{\rm gl.dim}\nolimits}

\def\D{\mathop{\rm{D}^{}}\nolimits}
\def\Db{\mathop{\rm{D}^b}\nolimits}
\def\Kb{\mathop{\rm{K}^b}\nolimits}

\CompileMatrices
\MakeOutlines

\edef\marginnotetextwidth{\the\textwidth} 

\def\MU#1{\mathop{\mu^+_{#1}}}
\def\NU#1{\mathop{\mu^-_{#1}}}

\begin{document}
\title[\textsc{MM Modules and AR Duality} for Non-isolated Singularities.]{\textsc{Maximal Modifications and Auslander--Reiten Duality for Non-isolated Singularities.}}
\dedicatory{In memory of Kentaro Nagao}
\author{Osamu Iyama}
\address{Osamu Iyama\\ Graduate School of Mathematics\\ Nagoya University\\ Chikusa-ku, Nagoya, 464-8602, Japan}
\email{iyama@math.nagoya-u.ac.jp}
\author{Michael Wemyss}
\address{Michael Wemyss, The Maxwell Institute, School of Mathematics, James Clerk Maxwell Building, The King's Buildings, Mayfield Road, Edinburgh, EH9 3JZ, UK.}
\email{wemyss.m@googlemail.com}
\thanks{The first author was partially supported by JSPS Grant-in-Aid for Scientific Research 24340004, 23540045, 20244001 and 22224001.  The second author was partially supported by a JSPS Postdoctoral Fellowship and by the EPSRC}
\begin{abstract}
We first generalize classical Auslander--Reiten duality for isolated singularities to cover singularities with a one-dimensional singular locus.  We then define the notion of CT modules for non-isolated singularities and we show that these are intimately related to noncommutative crepant resolutions (NCCRs).  When $R$ has isolated singularities, CT modules recover the classical notion of cluster tilting modules but in general the two concepts differ.  Then, wanting to generalize the notion of NCCRs to cover partial resolutions of $\Spec R$, in the main body of this paper we introduce a theory of modifying and maximal modifying modules.  Under mild assumptions all the corresponding endomorphism algebras of the maximal modifying modules for three-dimensional Gorenstein rings are shown to be derived equivalent.  We then develop a theory of mutation for modifying modules which is similar but different to mutations arising in cluster tilting theory.  Our mutation works in arbitrary dimension, and in dimension three the behavior of our mutation
 strongly depends on whether a certain factor algebra is artinian.
\end{abstract}
\maketitle
\parindent 20pt
\parskip 0pt

\tableofcontents

\section{Introduction}

\subsection{Motivation and History}

One of the basic results in representation theory of commutative
algebras is Auslander--Reiten (=AR) duality \cite{Aus78, AusReiten,Y} for isolated singularities, which gives us many important consequences, e.g.\ the existence of almost split sequences and the Calabi-Yau property of the stable categories of Cohen--Macaulay (=CM) modules over Gorenstein isolated singularities. One of the aims of this paper is to establish a version of AR duality for singularities with one dimensional singular loci.  As an application, the stable categories of CM modules over
Gorenstein singularities with one dimensional singular loci
enjoy a generalized Calabi-Yau property. This is a starting
point of our research to apply the methods of cluster tilting
in representation theory to study singularities.

One of the highlights of representation theory of commutative
algebras is AR theory of simple surface singularities \cite{Auslander_rational}.
They have only finitely many indecomposable CM modules, and
the Auslander algebras (i.e.\ the endomorphism algebras of the
direct sums of all indecomposable CM modules) enjoy many nice
properties. 
If we consider singularities of dimension greater than two, then
there are very few representation-finite singularities, and
their Auslander algebras do not satisfy such nice properties.
The reason is that the categories of CM modules do not behave
nicely in the sense that the homological properties of simple
functors corresponding to free modules are very hard to control.
Motivated to obtain the correct category on which higher AR theory should be performed, in \cite{IyamaAR} the first author introduced the notion of a maximal $n$-orthogonal subcategory and maximal $n$-orthogonal module for the category $\mod\Lambda$, later renamed cluster tilting subcategories and cluster tilting modules respectively.  Just as classical AR theory on $\mod\Lambda$ was moved to AR theory on  $\CM \Lambda$ following the work of Auslander on the McKay correspondence for surfaces $\Lambda$ \cite{Auslander_rational}, this suggests that in the case of a higher dimensional CM singularity $R$ we should apply the definition of a maximal $n$-orthogonal subcategory/modules to $\CM R$ and hope that this provides a suitable framework for tackling higher-dimensional geometric problems.  Strong evidence for this is provided when $R$ is a three dimensional normal isolated Gorenstein singularity, since in this case it is known \cite[8.13]{IR} that such objects have an intimate relationship with Van den Bergh's noncommutative crepant resolutions (NCCRs) \cite{VdBNCCR}.   Requiring $R$ to be isolated is absolutely crucial to this relationship (by normality the singularities are automatically isolated in the surfaces case); from an algebraic perspective this perhaps should not be surprising since AR theory only works well for isolated singularities.   It turns out that the study of maximal $n$-orthogonal modules in $\CM R$ is not well-suited to non-isolated singularities since the Ext vanishing condition is far too strong;  the question arises as to what subcategories of $\CM R$ should play the role of the maximal $n$-orthogonal subcategories above.

Although in this paper we answer this question, in fact we say much more.  Philosophically, the point is that we are asking ourselves the wrong question. The restriction to studying maximal orthogonal modules is unsatisfactory since crepant resolutions need not exist (even for 3-folds) and so we develop a theory which can deal with singularities in the crepant partial resolutions.  Since the endomorphism rings of maximal orthogonal modules have finite global dimension, these will not do the job for us.  

We introduce the notion of maximal modifying modules (see \ref{MMintro} below) which intuitively we think of as corresponding to shadows of maximal crepant partial resolutions. Geometrically this level always exists, but only sometimes will it be smooth. With regards to this viewpoint maximal modifying modules are a more natural class of objects to work with compared to noncommutative crepant resolutions;  we should thus always work in this level of generality and simply view the case when the geometry is smooth as being a happy coincidence.  Pushing this philosophy further, everything that we are currently able to do with NCCRs we should be able to do with maximal modifying modules, and this motivates much of the work in this paper.

In fact in many regards restricting our attention to only studying maximal crepant partial resolutions misses much of the picture and so we should (and do) work even more generally.  When one wants to flop curves between varieties with canonical singularities which are not terminal this does not take place on the maximal level but we should still be able to understand this homologically.  This motivates our definition and the study of modifying modules (see \ref{MMintro} below).  Viewing our modifying modules $M$ as conjectural shadows of partial resolutions we should thus be able to track the birational transformations between the geometrical spaces by using some kind of homological transformation between the corresponding modifying modules.  This leads us to develop a theory of mutation for modifying modules, which we do in Section~\ref{mutations}.

We note that some parts of the theory of (maximal) modifying modules developed in this paper are analogues of cluster tilting theory \cite{GLS,IR}, especially when the ring has Krull dimension three.  One main difference is that we do not assume that the ring is an isolated singularity, so we need to introduce (maximal) modifying modules which are much more general than (maximal) rigid modules in cluster tilting theory. Some of the main properties in cluster tilting theory are still true in our setting, for example, mutation is an involution in dimension three (see \ref{mainintro} below) and gives a derived equivalence in any dimension (see \ref{ddimmut_intro1} below). On the other hand, new features also appear in our setting. For example, mutation sometimes does not change the given modifying modules (see \ref{mainintro} below).  This feature is necessary, since it exactly reflects the geometry of partial crepant resolutions.

Although in this paper we are principally interested in the geometrical and commutative algebraic statements, the proofs of our theorems require a slightly more general noncommutative setting.  For this reason throughout this paper we use the language of singular Calabi--Yau algebras:
\begin{defin}
Let $\Lambda$ be a module finite $R$-algebra, then for $d\in\mathbb{Z}$ we call $\Lambda$ $d$-Calabi--Yau (=$d$-CY) if there is a functorial isomorphism 
\[
\Hom_{\D(\Mod \Lambda)}(X,Y[d])\cong D_0\Hom_{\D(\Mod \Lambda)}(Y,X)
\]
for all $X\in\Db(\fl \Lambda)$, $Y\in\Db(\mod\Lambda)$, where $D_0$ is the Matlis dual (see \S 2.4 for more details).  Similarly we call $\Lambda$ singular $d$-Calabi--Yau (=$d$-sCY) if the above functorial isomorphism holds for all $X\in\Db(\fl\Lambda)$ and $Y\in\Kb(\proj \Lambda)$.
\end{defin}
Clearly $d$-sCY (respectively $d$-CY) algebras are closed under derived equivalence \cite[3.1(1)]{IR}.  When $\Lambda=R$, it is known (see \ref{3.2IR}) that $R$ is $d$-sCY if and only if $R$ is Gorenstein and equi-codimensional with $\dim R=d$.  Thus throughout this paper, we use the phrase `$R$ is $d$-sCY' as a convenient shorthand for this important property.

We remark that by passing to mildly noncommutative $d$-sCY algebras we increase the technical difficulty, but we emphasize that we are forced to do this since we are unable to prove the purely commutative statements without passing to the noncommutative setting.

We now describe our results rigorously, and in more detail.

\subsection{Auslander--Reiten Duality for Non-Isolated Singularities}\label{ARsectionintro} 
Throughout this subsection let $R$ be an equi-codimensional (see \ref{equicodim} below) CM ring of dimension $d$ with a canonical module $\omega_{R}$.  Recall that for a non-local CM ring $R$, a finitely generated $R$-module $\omega_{R}$ is called a \emph{canonical module} if $({\omega_{R}})_{\m}$ is a canonical $R_{\m}$-module for all $\m\in\Max R$ \cite[3.3.16]{BH}.  In this case $(\omega_{R})_{\p}$ is a canonical $R_{\p}$-module for all $\p\in\Spec R$ since canonical modules localize for local CM rings \cite[3.3.5]{BH}.

We denote $\CM R$ to be the category of CM $R$-modules (see \S\ref{DepthandCMmodules}), $\underline{\CM}R$ to be the \emph{stable category} and $\overline{\CM}R$ to be the \emph{costable category}.  The \emph{AR translation} is defined to be
\[
\tau:=\Hom_{R}(\Omega^{d}\Tr(-),\omega_{R}):\underline{\CM}R\to\overline{\CM}R.
\]
When $R$ is an isolated singularity one of the fundamental properties of the category $\CM R$ is the existence of Auslander--Reiten duality \cite[I.8.8]{Aus78} \cite[1.1(b)]{AusReiten}, namely
\[
\u{\Hom}_R(X,Y)\cong D_0 \Ext^1_R(Y,\tau X)
\]
for all $X,Y\in\CM R$ where $D_0$ is the Matlis dual (see \S\ref{AR}).   Denoting $D_1:=\Ext_R^{d-1}(-,\omega_R)$ to be the duality on the category of Cohen--Macaulay modules of dimension $1$, we show that AR duality generalizes to mildly non-isolated singularities as follows:
\begin{thm}[=\ref{ARduality}]\label{ARdualityIntro}
Let $R$ be a $d$-dimensional, equi-codimensional CM ring with a canonical module $\omega_{R}$ and singular locus of Krull dimension less than or equal to one.  Then there exist functorial isomorphisms
\begin{align*}
\fl \u{\Hom}_{R}(X,Y)&\cong D_0(\fl\Ext^1_{R}(Y,\tau X)),\\
\frac{\u{\Hom}_{R}(X,\Omega Y)}{\fl\u{\Hom}_{R}(X,\Omega Y)}&\cong
D_1\left(\frac{\Ext^1_{R}(Y,\tau X)}{\fl\Ext^1_{R}(Y,\tau X)}\right)
\end{align*}
for all $X,Y\in\CM R$, where for an $R$-module $M$ we denote $\fl M$ to be the largest finite length $R$-submodule of $M$.
\end{thm}
In fact we prove \ref{ARduality} in the setting of noncommutative $R$-orders (see \S\ref{AR} for precise details). In the above and throughout this paper, for many of the global-local arguments to work we often have to add the following mild technical assumption. 
\begin{defin}\label{equicodim}
A commutative ring $R$ is \emph{equi-codimensional} if all its maximal ideals have the same height.
\end{defin}
Although technical, such rings are very common; for example all domains finitely generated over a field are equi-codimensional \cite[13.4]{Eisenbud}.  Since our main applications are three-dimensional normal domains finitely generated over $\mathbb{C}$, in practice this adds no restrictions to what we want to do.  We will use the following well-known property \cite[17.3(i), 17.4(i)(ii)]{Mat}:
\begin{lemma}\label{heightcoheight}
Let $R$ be an equi-codimensional CM ring, and let $\p\in\Spec R$.  Then
\[
\hgt\p+\dim R/\p=\dim R.
\]
\end{lemma}

The above generalized Auslander--Reiten duality implies the following generalized ($d-1$)-Calabi-Yau property of the triangulated category $\underline{\CM} R$.

\begin{cor}[=\ref{CMsymmcomplete}]\label{2CYgen} 
Let $R$ be a $d$-sCY ring with $\dim\Sing R\leq 1$.  Then\\
\t{(1)} There exist functorial isomorphisms
\begin{align*}
\fl \u{\Hom}_{R}(X,Y)&\cong D_0(\fl\u{\Hom}_{R}(Y,X[d-1])),\\
\frac{\u{\Hom}_{R}(X,Y)}{\fl\u{\Hom}_{R}(X,Y)}&\cong
D_1\left(\frac{\u{\Hom}_{R}(Y,X[d-2])}{\fl\u{\Hom}_{R}(Y,X[d-2])}\right)
\end{align*}for all $X,Y\in\CM R$.\\
\t{(2)} (=\ref{HomsymmNOTnormal}) For all $X,Y\in\CM R$, $\Hom_R(X,Y)\in \CM R$ if and only if $\Hom_R(Y,X)\in \CM R$.
\end{cor}
Note that \ref{2CYgen}(2) also holds (with no assumptions on the singular locus) provided that $R$ is normal (see \ref{HomsymmDirect}).  This symmetry in the Hom groups gives us the technical tool we require to move successfully from the cluster tilting level to the maximal modification level below, and is entirely analogous to the symmetry given by \cite[Lemma 1]{BcB} as used in cluster theory (e.g.\ \cite{GLS}).

\subsection{Maximal Modifications and NCCRs}  Here we introduce our main definitions, namely modifying, maximal modifying and CT modules, and then survey our main results. 

Throughout, an $R$-algebra is called \emph{module finite} if it is a finitely generated $R$-module.  As usual, we denote $(-)_\p:=-\otimes_RR_\p$ to be the localization functor.  For an $R$-algebra $\Lambda$, clearly $\Lambda_\p$ is an $R_\p$-algebra and we have a functor $(-)_\p:\mod\Lambda\to\mod\Lambda_\p$. Recall \cite{Aus78, Auslanderisolated, CR90}:
\begin{defin}\label{nonsingorder}
Let $R$ be a CM ring and let $\Lambda$ be a module finite $R$-algebra.  We say\\
\t{(1)}  $\Lambda$ is an \emph{$R$-order} if $\Lambda\in\CM R$.\\ 
\t{(2)}  An $R$-order $\Lambda$ is \emph{non-singular} if $\gl \Lambda_\p=\dim R_\p$ for all primes $\p$ of $R$.\\
\t{(3)} An $R$-order $\Lambda$ has \emph{isolated singularities} if $\Lambda_{\p}$ is a non-singular $R_{\p}$-order for all non-maximal primes $\p$ of $R$.

\end{defin}
In the definition of non-singular $R$-order above, $\gl \Lambda_\p=\dim R_\p$ means that $\gl\Lambda_\p$ takes the smallest possible value.  In fact for an $R$-order $\Lambda$ we always have that $\gl\Lambda_\p\geq \dim R_\p:=t_{\p}$ for all primes $\p$ of $R$ since $\pd_{\Lambda_\p}\Lambda_\p/(x_1,...,x_{t_{\p}})\Lambda_\p= \dim R_{\p}$ for a $\Lambda_\p$-regular sequence $x_1,\hdots,x_{t_{\p}}$.  We also remark that since the localization functor is exact and dense, we always have $\gl\Lambda_\p\leq\gl\Lambda$ for all $\p\in\Spec R$.

Throughout this paper we denote 
\[
(-)^{*}:=\Hom_{R}(-,R):\mod R\to\mod R
\]
and we say that $X\in\mod R$ is \emph{reflexive} if the natural map $X\to X^{**}$ is an isomorphism.  We denote $\refl R$ to be the category of reflexive $R$-modules.  By using Serre's ($S_2$)-condition (see for example \cite[3.6]{EG}, \cite[1.4.1(b)]{BH}), when $R$ is a normal domain the category $\refl R$ is closed under both kernels and extensions.
\begin{defin}\label{NCCR}
Let $R$ be CM, then by a \emph{noncommutative crepant resolution} (NCCR) of $R$ we mean $\Gamma:=\End_R(M)$ where $M\in\refl R$ is non-zero such that $\Gamma$ is a non-singular $R$-order.
\end{defin}
We show in \ref{nonsingGoren} that under very mild assumptions the condition in \ref{nonsingorder}(2) can in fact be checked at only maximal ideals, and we show in \ref{NCCRdefequiv} that \ref{NCCR} is equivalent to the definition of NCCR due to Van den Bergh \cite{VdBNCCR} when $R$ is a Gorenstein normal domain.  Recall the following:
\begin{defin}
Let $A$ be a ring.  We say that an $A$-module $M$ is a \emph{generator} if $A\in\add M$.  A projective $A$-module $M$ which is a generator is called a {\em progenerator}.
\end{defin}
Motivated by wanting a characterization of the reflexive generators which give NCCRs, we define:
\begin{defin}
Let $R$ be a $d$-dimensional CM ring with a canonical module $\omega_{R}$.  We call $M\in\CM R$ a \emph{CT module} if
\[
\add M=\{ X\in\CM R : \Hom_R(M,X)\in\CM R \}=\{ X\in\CM R : \Hom_R(X,M)\in\CM R\}.
\]
\end{defin}
Clearly a CT module $M$ is always a generator and cogenerator (i.e.\ $\add M$ contains both $R$ and $\omega_{R}$). 
We show in \ref{rigidisolated} that this recovers the established notion of maximal ($d-2$)-orthogonal modules when $R$ is $d$-dimensional and has isolated singularities. The following result in the complete case is shown in \cite[2.5]{IyamaAR} under the assumption that $G$ is a small subgroup of $\GL(d,k)$, and $S^G$ is an isolated singularity.  We can drop all assumptions under our definition of CT modules:
\begin{thm}[=\ref{skewgp}]\label{notsmallCTintro}
Let $k$ be a field of characteristic zero, and
let $S$ be the polynomial ring $k[x_1,\hdots,x_d]$ (respectively formal power series ring $k[[x_1,\hdots,x_d]]$). For a finite subgroup $G$ of $\GL(d,k)$, let $R=S^G$.  Then $S$ is a CT $R$-module.
\end{thm}

One of our main results involving CT modules is the following, where part (2) answers a question of Van den Bergh posed in \cite[4.4]{VdBNCCR}.
\begin{thm}[=\ref{NCCRhasCM}]\label{NCCRhasCMintro}
Let $R$ be a normal 3-sCY ring.   Then \\
\t{(1)} CT modules are precisely those reflexive generators which give NCCRs.\\
\t{(2)} $R$ has a NCCR $\iff$ $R$ has a NCCR given by a CM generator $M$ $\iff$ $R$ has a CT module.
\end{thm}
However in many cases $R$ need not have a NCCR so we must weaken the notion of CT module and allow for our endomorphism rings to have infinite global dimension.  We do so as follows:
\begin{defin}\label{MMintro}
Let $R$ be a $d$-dimensional CM ring.  We call $N\in\refl R$ a \emph{modifying module} if $\End_R(N)\in\CM R$, whereas we call $N$ a \emph{maximal modifying (MM) module} if $N$ is modifying and furthermore it is maximal with respect to this property, that is to say if there exists $X\in\refl R$ with $N\oplus X$ modifying, necessarily $X\in\add N$.  Equivalently, we say $N$ is maximal modifying if
\[
\add N=\{ X\in\refl R: \Hom_{R}(N \oplus X,N\oplus X)\in\CM R  \}.
\]
If $N$ is an MM module (respectively modifying module), we call $\End_R(N)$ a \emph{maximal modification algebra (=MMA)} (respectively modification algebra). 
\end{defin}
In this paper we will mainly be interested in the theoretical aspects of MMAs, but there are many natural examples.  In fact NCCRs are always MMAs (see \ref{NCCRgiveMM}) and so this gives one rich source of examples.  However MMAs need not be NCCRs, and for examples of this type of behaviour, together with the links to the geometry, we refer the reader to \cite{IW5}.  

When $R$ is $d$-dimensional with isolated singularities we show in \ref{rigidisolated} that modifying modules recover the established notion of ($d-2$)-rigid modules, whereas MM modules recover the notion of maximal ($d-2$)-rigid modules.   However, other than pointing out this relationship, throughout we never assume that $R$ has isolated singularities. 

When an NCCR exists, we show that MMAs are exactly the same as NCCRs:
\begin{prop}[=\ref{NCCRCTisgMR}]\label{NCCRCTisgMRintro}
Let $R$ be a normal 3-sCY ring, and assume that $R$ has a NCCR (equivalently, by \ref{NCCRhasCMintro}, a CT module).  Then\\
\t{(1)} MM modules are precisely the reflexive modules which give NCCRs.\\
\t{(2)} MM modules which are CM (equivalently, by \ref{CMiffgen}, the MM generators) are precisely the CT modules.\\
\t{(3)} CT modules are precisely those CM modules which give NCCRs.
\end{prop}
The point is that $R$ need not have a NCCR, and our definition of maximal modification algebra is strictly more general.

\subsection{Derived Equivalences}  We now explain some of our results involving derived equivalences of modifying modules.  We say that two module finite $R$-algebras $A$ and $B$ are  \emph{derived equivalent} if $\D(\Mod A)\simeq\D(\Mod B)$ as triangulated categories, or equivalently $\Db(\mod A)\simeq\Db(\mod B)$ by adapting \cite[8.1, 8.2]{Rickard}.  First, we show that any algebra derived equivalent to a modifying algebra also has the form $\End_R(M)$.

\begin{thm}[=\ref{closed under derived equivalences}]\label{closed under derived equivalences intro}
Let $R$ be a normal $d$-sCY ring, then\\
\t{(1)} Modifying algebras of $R$ are closed under derived equivalences, i.e.\ any ring derived equivalent to a modifying algebra is isomorphic to a modifying algebra.\\
\t{(2)} NCCRs of $R$ are closed under derived equivalences.
\end{thm}

The corresponding statement for MMAs is slightly more subtle, but we show it is true in dimension three (\ref{MMAs db intro}), and also slightly more generally in \ref{MMAs dim3 closed}(2).

Throughout this paper we freely use the notion of a tilting module which we \emph{always} assume has projective dimension less than or equal to one: 
\begin{defin}\label{tiltingdef}
Let $\Lambda$ be a ring.  Then $T\in\mod \Lambda$ is called a \emph{partial tilting module} if $\pd_{\Lambda}T\leq 1$ and $\Ext^{1}_{\Lambda}(T,T)=0$.  If further there exists an exact sequence
\[
0\to \Lambda\to T_{0}\to T_{1}\to 0
\]
with each $T_{i}\in\add T$, we say that $T$ is a \emph{tilting module}.
\end{defin}

Our next result details the relationship between modifying and maximal modifying modules on the level of derived categories.

\begin{thm}[=\ref{allgMRderived}]\label{allgMRderivedintro} 
Let $R$ be a normal 3-sCY ring with MM module $M$.  Then \\
\t{(1)} If $N$ is any modifying module, then $T:=\Hom_R(M,N)$ is a partial tilting $\End_{R}(M)$-module that induces a recollement \cite[\S1.4]{BBD}
\[
{\SelectTips{cm}{10}
\xy
(-23,0)*+{\rm K}="0",
(0,0)*+{\D(\Mod\End_R(M))}="1",
(36,0)*+{\D(\Mod\End_R(N))}="2",
\ar"0";"1",
\ar@<-1.5ex>"1";"0",
\ar@<1.5ex>"1";"0",
\ar|{F}"1";"2",
\ar@<-1.5ex>"2";"1",
\ar@<1.5ex>"2";"1",
\endxy}
\]
where $F=\mathbf{R}{\rm Hom}(T,-)$ and ${\rm K}$ is a certain triangulated subcategory of $\D(\Mod\End_R(M))$.\\
\t{(2)} If further $N$ is maximal modifying then the above functor $F$ is an equivalence.
\end{thm}
Theorems~\ref{closed under derived equivalences intro} and \ref{allgMRderivedintro} now give the following, which we view as the noncommutative analogue of a result of Chen \cite[1.1]{Chen}.
\begin{cor}[=\ref{closed in dimension three}]\label{MMAs db intro}
Let $R$ be a normal 3-sCY ring with an MM module.  Then all MMAs are derived equivalent, and further any algebra derived equivalent to an MMA is also an MMA.
\end{cor}

In our study of modifying modules, reflexive modules over noncommutative $R$-algebras play a crucial role.

\begin{defin}\label{reflexivedefinition} Let $R$ be any commutative ring. If $A$ is any $R$-algebra then we say that $M\in\mod A$ is a \emph{reflexive $A$-module} if it is reflexive as an $R$-module.  
\end{defin}
Note that we do not require that the natural map $M\to \Hom_{A^{\op}}(\Hom_A(M,A),A)$ is an isomorphism.  However when $A$ is 3-sCY  and $M$ is a reflexive $A$-module in the sense of \ref{reflexivedefinition}, automatically it is.  Our main theorem regarding maximal modifying modules is the following remarkable relationship between modifying modules and tilting modules.  Note that (3) below says that $R$ has a maximal modification algebra if and only if it has a maximal modification algebra $\End_R(N)$ where $N$ is a CM generator, a generalization of \ref{NCCRhasCMintro}(2).
\begin{thm}[=\ref{gRgMRrigid}, \ref{gRgMRrigid2}]\label{gRgMRrigidintro}
Let $R$ be a normal 3-sCY ring with an MM module $M$.  Then\\
\t{(1)} The functor $\Hom_R(M,-):\mod R\to\mod\End_{R}(M)$ induces bijections
\[
\begin{array}{rcl}
\{\mbox{maximal modifying $R$-modules}\}&\stackrel{1:1}\longleftrightarrow&\{\mbox{reflexive tilting $\End_R(M)$-modules}\}.\\
\{\mbox{modifying $R$-modules}\}&\stackrel{1:1}\longleftrightarrow&\{\mbox{reflexive partial tilting $\End_R(M)$-modules}\}.
\end{array}
\]
\t{(2)} $N$ is modifying $\iff$ $N$ is a direct summand of a maximal modifying module.\\
\t{(3)} $R$ has an MM module which is a CM generator.
\end{thm}

\subsection{Mutation of Modifications} Recall:

\begin{defin}\label{add approx def} Let $\Lambda$ be a ring.  For $\Lambda$-modules $M$ and $N$, we say that
a morphism $f:N_0\to M$ is a \emph{right $(\add N)$-approximation}
if $N_0\in \add N$ and further
\[
\Hom_\Lambda(N,N_0)\stackrel{\cdot f}{\to}\Hom_\Lambda(N,M)
\] 
is surjective.  Dually we define a \emph{left $(\add N)$-approximation}.
\end{defin}
Now let $R$ be a normal $d$-sCY ring.  We introduce categorical mutations as a method of producing modifying modules (together with a derived equivalence) from a given one.  
For a given modifying $R$-module $M$, and $N$ such that $0\neq N\in \add M$ we consider 
\begin{itemize}
\item[(1)] a right $(\add N)$-approximation of $M$, denoted $N_0\stackrel{a}\to M$.
\item[(2)] a right $(\add N^{*})$-approximation of $M^{*}$, denoted $N_{1}^{*}\stackrel{b}\to M^{*}$.  
\end{itemize}
Note that the above $a$ and $b$ are surjective if $N$ is a generator.  
In what follows we denote the kernels by
\[
\begin{array}{ccc}
0\to K_{0}\stackrel{c}\to N_0\stackrel{a}\to M&\mbox{and}&
0\to K_{1}\stackrel{d}\to N_1^{*}\stackrel{b}\to M^{*}
\end{array}
\]
and call these \emph{exchange sequences}.
\begin{defin}\label{mutationdef}
With notation as above, we define the \emph{right mutation} of $M$ at $N$ to be $\MU{N}(M):=N \oplus K_{0}$ and we define the \emph{left mutation} of $M$ at $N$ to be $\NU{N}(M):=N \oplus K_{1}^{*}$.
\end{defin} 
Note that by definition $\NU{N}(M)=( \MU{N^{*}}(M^{*}))^{*}$. 
\begin{thm}[=\ref{stillmodifying}, \ref{inverseoperations}]
\t{(1)} Both $\MU{N}(M)$ and $\NU{N}(M)$ are modifying $R$-modules.\\
\t{(2)} $\MU{N}$ and $\NU{N}$ are mutually inverse operations, i.e. we have that $\NU{N}(\MU{N}(M))= M$ and $\MU{N}(\NU{N}(M))= M$, up to additive closure.
\end{thm}

What is remarkable is that this process always produces derived equivalences, even in dimension $d$:
\begin{thm}[=\ref{welldefined2}, \ref{stillmodifying}]\label{ddimmut_intro1}
Let $R$ be a normal $d$-sCY ring with modifying module $M$.   Suppose that $0\neq N\in \add M$.  Then\\
\t{(1)} $\End_R(M)$, $\End_R(\NU{N}(M))$ and $\End_R(\MU{N}(M))$ are all derived equivalent. \\
\t{(2)} If $M$ gives an NCCR, so do $\MU{N}(M)$ and $\NU{N}(M)$.\\
\t{(3)} Whenever $N$ is a generator, if $M$ is a CT module so are $\MU{N}(M)$ and $\NU{N}(M)$.\\
\t{(4)} Whenever $\dim\Sing R\leq 1$ (e.g.\ if $d=3$), if $M$ is a MM module so are $\MU{N}(M)$ and $\NU{N}(M)$.
\end{thm}

In particular the above allows us to mutate any NCCR, in any dimension, at any direct summand, and will give another NCCR together with a derived equivalence.  In particular, we can do this when the ring $R$ is not complete local, and also we can do this when the NCCR may be given by a quiver with relations where the quiver has both loops and 2-cycles, in contrast to cluster theory.  This situation happens very frequently in the study of one-dimensional fibres, where this form of mutation seems to have geometric consequences.

One further corollary in full generality is the following surprising result on syzygies $\Omega$ and cosyzygies $\Omega^{-1}$, since they are a special case of left and right mutation.  Note that we have to be careful when defining our syzygies and cosyzygies so that they have free summands; see \S\ref{mutations1} for more details.

\begin{cor}[=\ref{syzcosyz result}]
Suppose that $R$ is a normal $d$-sCY ring and $M\in\refl R$ is a modifying generator.  Then\\
\t{(1)} $\Omega^iM\in\CM R$ is a modifying generator for all $i\in\mathbb{Z}$, and further all $\End_R(\Omega^iM)$ are derived equivalent.\\
\t{(2)} If $M$ is CT (i.e.\ gives an NCCR), then all $\Omega^iM\in\CM R$ are CT, and further all $\End_R(\Omega^iM)$ are derived equivalent.
\end{cor}

We remark that when $\dim R=3$, in nice situations we can calculate the mutated algebra from the original algebra by using various combinatorial procedures (see e.g.\ \cite[5.1]{BIRS}, \cite[3.2]{KY} and \cite[3.5]{V}), but we note that our mutation is categorical and much more general, and expecting a combinatorial rule is unfortunately too optimistic a hope.  We also remark that since we are dealing with algebras that have infinite global dimension, there is no reason to expect that they possess superpotentials and so explicitly describing their relations is in general a very difficult problem.

When $\dim R=3$ and $R$ is complete local, we can improve the above results.  In this setting, under fairly weak assumptions it turns out that left mutation is the same as right mutation, as in the case of cluster theory \cite{Iyama-Yoshino}.  If $0\neq N\in \add M$ then we define $[N]$ to be the two-sided ideal of $\End_R(M)$ consisting of morphisms $M\to M$ which factor through a member of $\add N$, and denote $\Lambda_N:=\Lambda/[N]$.  The behaviour of mutation at $N$ is controlled by $\Lambda_N$, in particular whether or not $\Lambda_N$ is artinian.  Note that when $R$ is finitely generated over a field $k$, $\Lambda_N$ is artinian if and only if it is finite dimensional over $k$ (see \ref{flartinian}).  

For maximal modifying modules the mutation picture is remarkably clear, provided that we mutate at only one indecomposable summand at a time:
\begin{thm}[=\ref{main}]\label{mainintro}
Suppose $R$ is complete normal $3$-sCY with MM module $M$.  Denote $\Lambda=\End_R(M)$, let $M_i$ be an indecomposable summand of $M$ and consider $\Lambda_i:=\Lambda/\Lambda(1-e_i)\Lambda$ where  $e_i$ is the idempotent in $\Lambda$ corresponding to $M_i$.  To ease notation denote $\MU{i}=\MU{\frac{M}{M_i}}$ and $\NU{i}=\NU{\frac{M}{M_i}}$.  Then\\
\t{(1)}  If $\Lambda_i$ is not artinian then $\MU{i}(M)=M=\NU{i}(M)$.\\
\t{(2)}  If $\Lambda_i$ is artinian then $\MU{i}(M)=\NU{i}(M)$ and this is not equal to $M$.\\
In either case denote $\mu_{i}:=\MU{i}=\NU{i}$ then it is also true that\\
\t{(3)} $\mu_{i}\mu_{i}(M)=M$.\\
\t{(4)} $\mu_{i}(M)$ is a MM module.\\
\t{(5)} $\End_R(M)$ and $\End_R(\mu_{i}(M))$ are derived equivalent, via the tilting $\End_{R}(M)$-module $\Hom_{R}(M,\mu_{i}(M))$.
\end{thm}

Some of the above proof works in greater generality, but we suppress the details here.

\subsection{Conventions}
We now state our conventions.  All modules will be left modules, so for a ring $A$  we denote $\mod A$ to be the category of finitely generated left $A$-modules.  Throughout when composing maps $fg$ will mean $f$ then $g$, similarly for quivers $ab$ will mean $a$ then $b$.  Note that with this convention $\Hom_R(M,X)$ is a $\End_R(M)$-module and $\Hom_R(X,M)$ is a $\End_R(M)^{\rm op}$-module.  For $M\in\mod A$ we denote $\add M$ to be the full subcategory consisting of summands of finite direct sums of copies of $M$ and we denote $\proj A:=\add A$ to be the category of finitely generated projective $A$-modules.  Throughout we will always use the letter $R$ to denote some kind of \emph{commutative noetherian} ring.  We always strive to work in the global setting, so we write $(R,\m)$ if $R$ is local.  We use the notation $\widehat{R_\p}$ to denote the completion of the localization $R_\p$ at its unique maximal ideal.

\section{Preliminaries}
\subsection{Depth and CM Modules}\label{DepthandCMmodules}

Here we record the preliminaries we shall need in subsequent sections, especially some global-local arguments that will be used extensively.  For a commutative noetherian local ring $(R,\m)$ and $M\in\mod R$ recall that the \emph{depth} of $M$ is defined to be
\[
\depth_R M:=\inf \{ i\geq 0: \Ext^i_R(R/\m,M)\neq 0 \},
\]
which coincides with the maximal length of a $M$-regular sequence.  Keeping the assumption that $(R,\m)$ is local we say that $M\in\mod R$ is \emph{maximal Cohen-Macaulay} (or simply, \emph{CM}) if $\depth_R M=\dim R$.  This definition generalizes to the non-local case as follows: if $R$ is an arbitrary commutative noetherian ring we say that $M\in\mod R$ is \emph{CM} if $M_\p$ is CM for all prime ideals $\p$ in $R$, and we say that $R$ is a \emph{CM ring} if $R$ is a CM $R$-module.

It is often convenient to lift the CM property to noncommutative rings, which we do as follows:
\begin{defin}\label{CMLambda}
Let $\Lambda$ be a module finite $R$-algebra, then we call $M\in\mod\Lambda$ a \emph{CM $\Lambda$-module} if it is CM when viewed as an $R$-module.  We denote the category of CM $\Lambda$-modules by $\CM\Lambda$.
\end{defin}
To enable us to bring the concept of positive depth to non-local rings, the following is convenient:
\begin{defin}
Let $R$ be a commutative noetherian ring and $M\in \mod R$.  We denote $\fl M$ to be the unique maximal finite length $R$-submodule of $M$.
\end{defin}
It is clear that $\fl M$ exists because of the noetherian property of $M$; when $(R,\m)$ is local $\fl M=\{ x\in M:\exists\, r\in\mathbb{N} \mbox{ with } \m^rx=0 \}$.  The following is well-known.
\begin{lemma}\label{depthofhom}
Let $(R,\m)$ be a local ring of dimension $d\geq 2$ and let $\Lambda$ be a module finite $R$-algebra. Then for all $M,N\in\mod \Lambda$ with $\depth_{R} N\geq 2$ we have $\depth_{R}\Hom_{\Lambda}(M,N)\geq 2$.
\end{lemma}
\begin{proof}
A free presentation $\Lambda^a \to \Lambda^b\to M\to 0 $ gives $0 \to  \Hom_{\Lambda}(M,N) \to N^b\to  N^a$ so the result follows from the depth lemma.
\end{proof}
In particular if $\depth R\geq 2$ then reflexive $R$-modules always have depth at least two. 

\begin{lemma}\label{CMinheight2}
Suppose $R$ is a $d$-dimensional CM ring with $d\geq 2$ and let $\Lambda$ be a module finite $R$-algebra.  For any $X\in\refl \Lambda$ we have $X_\p\in\CM \Lambda_\p$  for all $\p\in\Spec R$ with $\hgt\p\leq 2$.  
\end{lemma}
\begin{proof}
Since $X$ is reflexive as an $R$-module we can find an exact sequence $0\to X\to P\to Q$ with $P,Q\in\add R$ and so on localizing we see that $X_\p$ is a second syzygy for all primes $\p$. Consequently if $\p$ has height $\leq 2$ then $X_\p$ is a second syzygy for the CM ring $R_\p$ which has $\dim R_\p\leq 2$ and so $X_\p\in\CM R_{\p}$.
\end{proof}

\subsection{Reflexive Equivalence and Symmetric Algebras}  Here we introduce and fix notation for reflexive modules and symmetric algebras.  All the material in this subsection can be found in \cite{IR}. Recall from the introduction (\ref{reflexivedefinition}) our convention on the definition of reflexive modules.  Recall also that if $\Lambda$ is a module finite $R$-algebra, we say $M\in\refl \Lambda$ is called a \emph{height one progenerator} (respectively, \emph{height one projective}) if $M_{\p}$ is a progenerator (respectively, projective) over $\Lambda_{\p}$ for all $\p\in\Spec R$ with $\hgt\p\leq 1$.

In this paper, when the underlying commutative ring $R$ is a normal domain the following reflexive equivalence is crucial: 
\begin{lemma}\label{reflequiv}
If $\Lambda$ is a module finite $R$-algebra, then\\
\t{(1)} If $M\in\mod \Lambda$ is a generator then
\[\Hom_{\Lambda}(M,-):\mod {\Lambda}\to\mod\End_{\Lambda}(M)\]
is fully faithful, restricting to an equivalence $\add M\stackrel{\simeq}{\to}\proj\End_{\Lambda}(M)$.\\
If further $R$ is a normal domain, then the following assertions hold.\\
\t{(2)} $\Hom_\Lambda(X,Y)\in\refl R$ for any $X\in\mod\Lambda$ and any $Y\in\refl\Lambda$.\\
\t{(3)} Every non-zero $M\in \refl R$ is a height one progenerator.\\
\t{(4)} Suppose $\Lambda$ is a reflexive $R$-module and let $M\in\refl\Lambda$ be a height one progenerator.  Then 
\[\Hom_{\Lambda}(M,-):\refl {\Lambda}\to\refl\End_{\Lambda}(M)\]
is an equivalence.  In particular $\Hom_{R}(N,-):\refl R\to\refl\End_{R}(N)$ is an equivalence for all non-zero $N\in\refl R$.
\end{lemma}
\begin{proof}
(1) is standard.\\
(2) follows easily from the fact that reflexives are closed under kernels; see \cite[2.4(1)]{IR}.\\
(3) If $\p$ is a height one prime then by \ref{CMinheight2} $M_\p\in\CM R_\p$.  But $R$ is normal so $R_\p$ is regular; thus $M_\p$ is free.\\
(4) follows by (3) and \cite[1.2]{RV89} (see also \cite[2.4(2)(i)]{IR}).
\end{proof}

Throughout this paper, we will often use the following observation.
\begin{lemma}\label{Extisfl}
Let $R$ be a CM ring, $X,Y\in \CM R$.  Then $\Supp_R\Ext^i_R(X,Y)\subseteq\Sing R$ for all $i>0$.  In particular, if $R$ has isolated singularities then $\Ext^i_R(X,Y)$ is a finite length $R$-module for all $X,Y\in \CM R$ and $i>0$.
\end{lemma}
\begin{proof} 
This is well-known \cite{Aus78}, \cite[3.3]{Y}.
\end{proof}

The following lemma is convenient and will be used extensively.

\begin{lemma}\label{reflandCM}
Let $R$ be a 3-dimensional, equi-codimensional CM ring and let $\Lambda$ be a module finite $R$-algebra.  If $X\in\mod\Lambda$ and $Y\in\refl \Lambda$ then
\[
\Hom_{\Lambda}(X,Y)\in\CM R\Rightarrow \fl \Ext^1_{\Lambda}(X,Y)=0 .
\]
If further $Y\in\CM\Lambda$ then the converse holds.
\end{lemma}
\begin{proof}
($\Rightarrow$) For each $\m\in\Max R$ there is an exact sequence 
\begin{eqnarray}
0 \to  \Hom_{{\Lambda}}(X,Y)_\m \stackrel{f_\m}{\to}  \Hom_{{\Lambda}}(P,Y)_\m\to \Hom_{{\Lambda}}(\Omega X,Y)_\m \to  \Ext^1_{{\Lambda}}(X,Y)_\m\to  0\label{sesfllemma}
\end{eqnarray}
obtained by localizing the exact sequence obtained from $0 \to  \Omega X \to P\to X\to 0 $ with $P\in\add {\Lambda}$.  Now $\depth_{R_\m} \Hom_{{\Lambda}_\m}(X_\m,Y_\m)=3$ and further by \ref{depthofhom} $\depth_{R_\m}\Hom_{{\Lambda}_\m}(P_\m,Y_\m)\geq 2$, thus $\depth_{R_\m}\Cok f_\m\geq 2$.  Since again by \ref{depthofhom} $\depth_{R_\m}\Hom_{{\Lambda}_\m}(\Omega X_\m,Y_\m)\geq 2$, we conclude that $\depth_{R_\m}\Ext^1_{{\Lambda}_\m}(X_\m,Y_\m)>0$ for all $\m\in\Max R$ and so $\fl \Ext^1_{\Lambda}(X,Y)=0$.\\
($\Leftarrow$) Suppose now that $Y\in\CM\Lambda$.  Then in (\ref{sesfllemma}) $\depth_{R_\m}\Hom_{{\Lambda}_\m}(P_\m,Y_\m)=3$ and so by a similar argument $\depth_{R_\m}\Ext^1_{{\Lambda}_\m}(X_\m,Y_\m)>0$ implies that $\depth_{R_\m} \Hom_{{\Lambda}_\m}(X_\m,Y_\m)=3$.
\end{proof}

\begin{defin}\label{symmdefin}
Let $\Lambda$ be a module finite $R$-algebra where $R$ is an arbitrary commutative ring.  We call $\Lambda$ a \emph{symmetric $R$-algebra} if $\Hom_R(\Lambda,R)\cong\Lambda$ as $\Lambda$-bimodules.  We call $\Lambda$ a \emph{locally symmetric $R$-algebra} if $\Lambda_{\p}$ is a symmetric $R_{\p}$-algebra for all $\p\in\Spec R$.
\end{defin}
Note that if $\Lambda$ is a symmetric $R$-algebra then as functors $\mod\Lambda\to\mod\Lambda^{\rm op}$
\[
\Hom_{\Lambda}(-,\Lambda)\cong\Hom_{\Lambda}(-,\Hom_{R}(\Lambda,R))\cong\Hom_{R}(\Lambda\otimes_{\Lambda}-,R)=\Hom_{R}(-,R).
\]
We have the following well-known observation.  Recall that throughout our paper, $(-)^*:=\Hom_R(-,R)$.
\begin{lemma}\label{HomsymmDirect}
Let $R$ be a normal domain and $\Lambda$ be a symmetric $R$-algebra.
Then there is a functorial isomorphism
$\Hom_\Lambda(X,Y)\cong \Hom_\Lambda(Y,X)^*$ for all
$X,Y\in \refl\Lambda$ such that $Y$ is height one projective.
\end{lemma}
\begin{proof}
For the convenience of the reader, we give a detailed proof here.
We have a natural map
$f:\Hom_\Lambda(Y,\Lambda)\otimes_\Lambda X\to \Hom_\Lambda(Y,X)$
sending $a\otimes x$ to $(y\mapsto a(y)x)$. Consider the map
$f^*:\Hom_\Lambda(Y,X)^*\to (\Hom_\Lambda(Y,\Lambda)\otimes_\Lambda X)^*$
between reflexive $R$-modules. Since $Y$ is height one projective,
$f_\p$ and $(f^*)_\p$ are isomorphisms for any prime $\p$ of height at most one.
Thus $f^*$ is an isomorphism since $R$ is normal. Thus we have
\[
\Hom_{\Lambda}(Y,X)^{*}\stackrel{f^{*}}{\cong} (\Hom_{\Lambda}(Y,\Lambda)\otimes_{\Lambda} X)^{*} 
\cong (Y^{*}\otimes_{\Lambda} X)^{*}
\cong \Hom_{\Lambda}(X,Y^{**})
\cong \Hom_{\Lambda}(X,Y)
\]
as required.
\end{proof}
This immediately gives the following result, which implies that symmetric algebras are closed under reflexive equivalence.
\begin{lemma}\cite[2.4(3)]{IR}\label{Endissymm}
If $\Lambda$ is a symmetric $R$-algebra then so is $\End_{\Lambda}(M)$ for any height one projective $M\in\refl \Lambda$.  In particular if $R$ is a normal domain and $N\in\refl R$ then $\End_R(N)$ is a symmetric $R$-algebra.
\end{lemma}

When discussing derived equivalence of modifying algebras later, we will require the following result due to Auslander--Goldman.

\begin{prop}[{cf.\ \cite[4.2]{AG}}]\label{Auslander-Goldman}
Let $R$ be a normal domain with $\dim R\geq 1$, and let $\Lambda$ be a module finite $R$-algebra. Then the following conditions are equivalent:\\
\t{(1)} There exists $M\in\refl R$ such that $\Lambda\cong\End_R(M)$ as $R$-algebras.\\
\t{(2)} $\Lambda\in\refl R$ and further $\Lambda_\p$ is Morita equivalent to $R_\p$ for all $\p\in\Spec R$ with $\hgt\p=1$.
\end{prop}
\begin{proof}
(1)$\Rightarrow$(2) is trivial.\\
(2)$\Rightarrow$(1) For the convenience of the reader, we give a simple direct proof.   Let $K$ be the quotient field of $R$.
As we have assumed that $\Lambda_\p$ is Morita equivalent to $R_\p$, $K\otimes_R\Lambda\cong M_n(K)$ as $K$-algebras for some $n>0$.  Thus for any $M\in \mod\Lambda$, we can regard $K\otimes_RM$ as an $M_n(K)$-module. We denote by $V$ the simple $M_n(K)$-module.

First, we show that there exists $M\in \refl\Lambda$ such that $K\otimes_RM \simeq V$ as $M_n(K)$-modules. For example, take a split epimorphism $f:M_n(K)\to V$ of $M_n(K)$-modules
and let $M:=(f(\Lambda))^{**}$. Then clearly $M$ satisfies the desired properties.

Next, for $M$ above, we show that the natural map $g:\Lambda\to \End_R(M)$ given by the action of $\Lambda$ on $M$ is an isomorphism. Since both $\Lambda$ and $\End_R(M)$ are reflexive $R$-modules by our assumption, it is enough to show that $g_\p$ is an isomorphism for all $\p\in \Spec R$ with $\hgt \p=1$. Since $K\otimes_R\End_R(M)=\End_K(K\otimes_RM)=\End_K(V)=M_n(K)$, we have that $K\otimes g:K\otimes_R\Lambda\to K\otimes_R\End_R(M)$ is an isomorphism.
In particular, $g_\p:\Lambda_\p\to \End_R(M)_\p$ is injective.
Since $R_\p$ is local and $\Lambda_\p$ is Morita equivalent to $R_\p$, we have that $\Lambda_\p$ is a full matrix ring over $R_\p$, which is well-known to be a maximal order over $R_\p$ (see e.g.\ \cite[3.6]{AG}, \cite[\S37]{CR90}). Thus we have that $g_\p$ is an isomorphism. 
\end{proof}

\subsection{Non-Singular and Gorenstein Orders} Recall from \ref{nonsingorder} the definition of a non-singular $R$-order.  By definition the localization of a non-singular $R$-order is again a non-singular $R$-order --- we shall see in \ref{nonsingGoren} that in most situations we may check whether an algebra is a non-singular $R$-order by checking only at the maximal ideals. 

For some examples of non-singular $R$-orders, recall that for a ring $\Lambda$ and a finite group $G$ together with a group homomorphism $G\to {\rm Aut}_{k\rm{-alg}}(\Lambda)$, we define the \emph{skew group ring} $\Lambda\# G$ as follows \cite{Auslander_rational, Y}:
As a set, it is a free $\Lambda$-module with the basis $G$.
The multiplication is given by
\[(sg)(s'g'):=(s\cdot g(s'))(gg')\]
for any $s,s'\in S$ and $g,g'\in G$.

\begin{lemma}\label{skew is non-singular}
Let $R$ be a CM ring containing a field $k$. Let $\Lambda$ be a non-singular $R$-order, let $G$ be a finite group together with a group homomorphism $G\to {\rm Aut}_{R\mbox{-}{\rm{alg}}}(\Lambda)$ and suppose $|G|\neq0$ in $k$. Then $\Lambda\# G$ is a non-singular $R$-order.
\end{lemma}
\begin{proof}
Since $\Lambda\# G$ is a direct sum of copies of $\Lambda$ as an $R$-module and $\Lambda\in\CM R$, we have $\Lambda\# G\in \CM R$.
Now if $X,Y\in \mod\Lambda\# G$ then $G$ acts on $\Hom_\Lambda(X,Y)$ by
\[
(gf)(x):=g\cdot f(g^{-1}x)
\]
for all $g\in G$, $f\in\Hom_\Lambda(X,Y)$ and $x\in X$. Clearly we have a functorial isomorphism
\[
\Hom_{\Lambda\# G}(X,Y)=\Hom_\Lambda(X,Y)^G
\]
for all $X,Y\in\mod\Lambda\# G$. Taking $G$-invariants $(-)^G$ is an exact functor since $kG$ is semisimple. Thus we have a functorial isomorphism
\[\Ext^i_{\Lambda\# G}(X,Y)=\Ext^i_\Lambda(X,Y)^G\]
for all $X,Y\in\mod\Lambda\# G$ and $i\ge 0$. In particular, $\gl\Lambda\# G\le \gl\Lambda$ holds, and we have the assertion.
\end{proof}

In the remainder of this subsection we give basic properties of non-singular orders.
\begin{lemma}\label{nonsingMorita}
Non-singular $R$-orders are closed under Morita equivalence.
\end{lemma}
\begin{proof}
Suppose that $\Lambda$ is a non-singular $R$-order and $\Gamma$ is Morita equivalent to $\Lambda$. Then $\Lambda_\p$ is Morita equivalent to $\Gamma_\p$ for all primes $\p$. Thus since global dimension is an invariant of the abelian category, we have $\gl \Gamma_\p=\gl \Lambda_\p=\dim R_\p$ for all primes $\p$.  To see why the CM property passes across the Morita equivalence let $P$ denote the progenerator in $\mod \Lambda$ such that $\Gamma\cong \End_\Lambda(P)$.  Since $P$ is a summand of $\Lambda^{n}$ for some $n$ we know that $\Gamma$ is a summand of $\End_{\Lambda}(\Lambda^{n})=M_{n}(\Lambda)$ as an $R$-module.  Since $\Lambda$ is CM, so is $\Gamma$.
\end{proof}

Recall from the introduction (\S\ref{ARsectionintro}) the definition of a canonical module $\omega_{R}$ for a non-local CM ring $R$.   If $\Lambda$ is an $R$-order we have an exact duality
\[
\Hom_R(- ,\omega_R):\CM\Lambda\leftrightarrow \CM\Lambda^{\rm op}
\]
and so the $\Lambda$-module $\omega_\Lambda:=\Hom_R(\Lambda,\omega_R)$ is an injective cogenerator in the category $\CM\Lambda$.
\begin{defin}\cite{CR90, GN}
Assume $R$ has a canonical module $\omega_{R}$.  An $R$-order $\Lambda$ is called \emph{Gorenstein} if $\omega_{\Lambda}$ is a projective $\Lambda$-module.
\end{defin}

It is clear that if $\Lambda$ is a Gorenstein $R$-order then $\Lambda_{\p}$ is a Gorenstein $R_{\p}$-order for all $\p\in\Spec R$.  If $R$ is Gorenstein and $\Lambda$ is a symmetric $R$-order (i.e.\ an $R$-order which is a symmetric $R$-algebra), then $\Lambda$ is clearly a Gorenstein $R$-order. Moreover if both $R$ and $\Lambda$ are $d$-sCY (see \S\ref{3CY}), we shall see in \ref{2.14a} that $\Lambda$ is a Gorenstein order.    Also, we have the following.

\begin{lemma}\label{addLambda=addOmega}
Let $\Lambda$ be an $R$-order, then the following are equivalent.\\
\t{(1)} $\Lambda$ is Gorenstein.\\
\t{(2)} $\add_\Lambda(\Lambda)=\add_\Lambda(\omega_{\Lambda})$.\\
\t{(3)} $\add_{\Lambda^{\op}}(\Lambda)=\add_{\Lambda^{\op}}(\omega_{\Lambda})$.\\
\
\t{(4)} $\Lambda^{\op}$ is Gorenstein.
\end{lemma}
\begin{proof}
Due to  \ref{addlocal}, we can assume that $R$ is complete local, and so $\mod \Lambda$ is Krull--Schmidt \cite[6.12]{CR90}.\\
(1)$\Rightarrow$(2) If $\Lambda$ is Gorenstein then by definition $\add_\Lambda(\omega_\Lambda)\subseteq\add_\Lambda(\Lambda)$.  The number of non-isomorphic indecomposable projective $\Lambda$-modules is equal to that of $\Lambda^{\op}$-modules.  Moreover, the latter is equal to the number of non-isomorphic indecomposable summands of $\omega_\Lambda$ by the duality $\Hom_R(- ,\omega_R)$.  By the Krull--Schmidt property, $\add_\Lambda(\Lambda)\subseteq\add_\Lambda(\omega_\Lambda)$.
(2)$\Rightarrow$(1) is trivial.\\
(2)$\Leftrightarrow$(3) follows by applying the duality $\Hom_R(- ,\omega_R)$.\\
(3)$\Leftrightarrow$(4) is identical to (1)$\Leftrightarrow$(2).
\end{proof}

When $R$ is local, Gorenstein $R$-orders $\Lambda$ are especially important since we have the following Auslander--Buchsbaum type equality, which in particular says that the $\Lambda$-modules which have finite projective dimension and are CM as $R$-modules are precisely the projective $\Lambda$-modules.

\begin{lemma}\label{ABlocal}
Let $(R,\m)$ be a local CM ring with a canonical module $\omega_{R}$ and let $\Lambda$ be a Gorenstein $R$-order.  Then for any $X\in\mod \Lambda$ with $\pd_\Lambda X<\infty$,
\[
\depth_R X+\pd_{\Lambda} X=\dim R.
\]
\end{lemma}
\begin{proof}
Let $X$ be a $\Lambda$-module with $\pd_{\Lambda} X<\infty$.\\
(i)  We will show that if $X\in\CM \Lambda$ then $X$ is projective.  We know that $\Ext^{i}_{\Lambda}(-,\omega_{\Lambda})=0$ on $\CM\Lambda$ for all $i>0$ since $\omega_{\Lambda}$ is injective in $\CM\Lambda$.  Since $\Lambda$ is Gorenstein $\add\Lambda=\add\omega_\Lambda$ by \ref{addLambda=addOmega} and so we have $\Ext^{i}_{\Lambda}(-,\Lambda)=0$ on $\CM\Lambda$ for all $i>0$.  Since $\Ext^{n}_{\Lambda}(X,\Lambda)\neq 0$ for $n=\pd X$, we have that $X$ is projective.\\
(ii)  Let $n=\pd X$ and $t=\depth X$.  Take a minimal projective resolution
\[
0\to P_{n}\to \hdots\to P_{0}\to X\to 0.
\]
By the depth lemma necessarily $t\geq d-n$.  On the other hand by the depth lemma we have $\Omega^{d-t}X\in\CM\Lambda$.  By (i) we know $\Omega^{d-t}X$ is projective so $n\leq d-t$.  Thus $d=n+t$. 
\end{proof}

The following result is well-known to experts (e.g.\ \cite[1.5]{Auslanderisolated}).
\begin{prop}\label{nonsingGoren}
Let $\Lambda$ be an $R$-order where $R$ is a CM ring with a canonical module $\omega_{R}$. Then the following are equivalent:\\
\t{(1)} $\Lambda$ is non-singular.\\
\t{(2)} $\gl\Lambda_{\m}=\dim R_{\m}$ for all $\m\in\Max R$.\\
\t{(3)} $\CM\Lambda=\proj\Lambda$.\\
\t{(4)} $\Lambda$ is Gorenstein and $\gl\Lambda<\infty$.
\end{prop}
\begin{proof}
(1)$\Rightarrow$(2) This is immediate.\\
(2)$\Rightarrow$(3) Let $X\in\CM\Lambda$. Then $X_{\m}\in\CM\Lambda_{\m}$.
Let $x_{1},\hdots,x_{d}$ be a $X_{\m}$-regular sequence with $d=\dim R_{\m}$.  Since we have an exact sequence
\[
0 \to X_\m \xrightarrow{x_1} X_\m \to X_\m/x_1X_\m \to 0
\]
which induces an exact sequence
\[
\Ext^i_{\Lambda_\m}(X_\m,-)\xrightarrow{x_1}\Ext^i_{\Lambda_\m}(X_\m,-)\to\Ext^{i+1}_{\Lambda_\m}(X_\m/x_1X_\m,-)\to\Ext^{i+1}_{\Lambda_\m}(X_\m,-),
\]
we have $\pd_{\Lambda_\m}(X_\m/x_1X_\m)=\pd_{\Lambda_\m}X_\m+1$ by Nakayama's Lemma. Using this repeatedly, we have $\pd_{\Lambda_{\m}} (X_{\m}/(x_{1},\hdots,x_{d})X_{\m})=\pd_{\Lambda_{\m}}X_{\m}+d$.  Since $\gl\Lambda_\m=d$, this implies that $X_{\m}$ is a projective $\Lambda_{\m}$-module.  Since this holds for all $\m\in\Max R$, $X$ is a projective $\Lambda$-module (see e.g.\ \ref{addlocal}).\\
(3)$\Rightarrow$(4) We have $\omega_\Lambda\in \CM\Lambda=\proj\Lambda$.
Since $\Omega^{\dim R}X\in \CM\Lambda=\proj\Lambda$ for any $X\in \mod\Lambda$,
we have $\gl\Lambda\leq \dim R$.\\
(4)$\Rightarrow$(1) Pick $\p\in\Spec R$ and suppose $Y\in\mod\Lambda_{\p}$.  Since $\gl\Lambda_\p<\infty$, by Auslander--Buchsbaum \ref{ABlocal} $\pd_{\Lambda_{\p}}Y\leq \dim R_{\p}$ and so $\gl\Lambda_{\p}\leq \dim R_{\p}$.  Since $\Lambda_{\p}$ is an $R_{\p}$-order, the $\Lambda_\p$-regular sequence $x_1,\ldots,x_d$ with $d=\dim R_\p$ gives an $\Lambda_\p$-module $X:=\Lambda_\p/(x_1,\ldots,x_d)\Lambda_\p$ with $\pd_{\Lambda_\p}X=d$. Thus we have $\gl\Lambda_{\p}\geq \dim R_{\p}$. 
\end{proof}

\subsection{$d$-sCY Algebras}\label{3CY}

Throughout this paper we shall freely use the notion of $d$-CY and $d$-sCY as in \cite[\S 3]{IR}:  let $R$ be a commutative noetherian ring with $\dim R=d$ and let $\Lambda$ be a module finite $R$-algebra.  For any $X\in\mod \Lambda$ denote by $E(X)$ the injective hull of $X$, and set $E:=\bigoplus_{\m\in\Max R}E(R/\m)$.  This gives rise to Matlis duality $D_0:=\Hom_R(-,E)$ (see for example \cite[\S 1]{oo}).  Matlis duality always gives a duality from
the category of finite length $R$-modules to itself.
This is true without assuming that $R$ is (semi-)local
because any finite length $R$-module is the finite direct sum
of finite length $R_\m$-modules for maximal ideals $\m$,
so the statement follows from that for the local setting \cite[3.2.12]{BH}.  Recall from the introduction:
\begin{defin}
For $n\in\mathbb{Z}$ we call $\Lambda$ $n$-CY if there is a functorial isomorphism 
\[
\Hom_{\D(\Mod \Lambda)}(X,Y[n])\cong D_0\Hom_{\D(\Mod \Lambda)}(Y,X)
\]
for all $X\in\Db(\fl \Lambda)$, $Y\in\Db(\mod\Lambda)$.  Similarly we call $\Lambda$ $n$-sCY if the above functorial isomorphism holds for all $X\in\Db(\fl\Lambda)$ and $Y\in\Kb(\proj \Lambda)$.
\end{defin}

The next three results can be found in \cite{IR}; we include them here since we will use and refer to them extensively.

\begin{prop}\label{IRDB}
\t{(1)}  $\Lambda$ is $d$-CY if and only if
it is $d$-sCY and $\gl\Lambda<\infty$.\\
\t{(2)} $d$-sCY (respectively $d$-CY) algebras are closed under derived equivalences.
\end{prop}
\begin{proof}
(1) is \cite[3.1(7)]{IR} whilst (2) is \cite[3.1(1)]{IR}.
\end{proof}

\begin{prop}\label{3.2IR}\cite[3.10]{IR}
Let $R$ be a commutative noetherian ring, $d\in\mathbb{N}$.  Then $R$ is $d$-sCY if and only if $R$ is Gorenstein and equi-codimensional with $\dim R=d$.
\end{prop}
From this, for brevity we often say `$R$ is $d$-sCY' instead of saying
`$R$ is Gorenstein and equi-codimensional with $\dim R=d$'.

\begin{prop}\label{2.14a} Let $R$ be $d$-sCY and let $\Lambda$ be a module finite $R$-algebra. Then
\begin{center} 
$\Lambda$ is $d$-sCY $\iff$ $\Lambda$ is an $R$-order which is a locally symmetric $R$-algebra.
\end{center}
Thus if $\Lambda$ is $d$-sCY then $\Lambda$ is a Gorenstein $R$-order.
\end{prop}
\begin{proof}
The first statement is \cite[3.3(1)]{IR}.  For the second, suppose $\Lambda$ is $d$-sCY then since it is locally symmetric we have $\Lambda_{\m}\cong\Hom_{R_{\m}}(\Lambda_{\m},R_{\m})=\Hom_{R}(\Lambda,R)_{\m}$ is a projective $\Lambda_{\m}$-module for all $\m\in\Max R$.  Hence $\Hom_{R}(\Lambda,R)$ is a projective $\Lambda$-module, as required.
\end{proof}

The following picture for $d$-sCY rings $R$ may help the reader navigate the terminology introduced above.
\[
{\SelectTips{cm}{10}
\xy0;/r.37pc/:
(7,0)*+{\mbox{{\rm $d$-CY $R$-algebra}}}="1",
(63,0)*+{\mbox{{\rm non-singular $R$-order}}}="3",
(35,-18)*+{\mbox{{\rm symmetric $R$-order}}}="c2",
(7,-9)*+{\mbox{{\rm $d$-sCY $R$-algebra}}}="b1",
(35,-9)*+{\mbox{{\rm locally symmetric $R$-order}}}="b2",
(63,-9)*+{\mbox{{\rm Gorenstein $R$-order}}}="b3"
\ar@{=>}"1";"3"
\ar@<-3ex>@{=>}"b1";"1"_{{\scriptsize \textnormal{if}\,\gl<\infty}}
\ar@{<=>}^(0.425){\scriptsize\ref{2.14a}}"b1";"b2"
\ar@{=>}^(0.54){\scriptsize\ref{2.14a}}"b2";"b3"
\ar@{=>}"1";"b1"
\ar@{=>}_{{\scriptsize\ref{nonsingGoren}}}"3";"b3"
\ar@<-3ex>@{=>}"b3";"3"_{{\scriptsize \textnormal{if}\,\gl<\infty}}
\ar@{=>}"c2";"b2"
\endxy}
\]

The following non-local result is also useful.
\begin{lemma}\label{3CY-nonlocal}
Suppose that $R$ is a $d$-sCY normal domain.\\
\t{(1)} If $\Lambda$ is a module finite $R$-algebra which is $d$-sCY and $M\in\refl \Lambda$ is a height one projective, then $
\End_{\Lambda}(M) \mbox{ is $d$-sCY} \iff \End_{\Lambda}(M)\in\CM R
$.\\
\t{(2)} If $N\in\refl R$ then $\End_{R}(N)$ is $d$-sCY $\iff \End_{R}(N)\in\CM R$. 
\end{lemma}
\begin{proof}
(1) Let $\Gamma:=\End_{R}(M)$.  By \ref{2.14a} $\Lambda_\m$ is a symmetric $R_\m$-algebra for all $\m\in \Max R$, thus $\Gamma_\m$ is a symmetric $R_\m$-algebra by \ref{Endissymm}.  By \ref{2.14a}, $\Gamma$ is $d$-sCY if and only if $\Gamma_\m\in \CM R_\m$ for all $\m\in \Max R$ if and only if $\Gamma\in \CM R$.  Thus the assertion follows.\\
(2) This follows immediately from (1) since any $N\in\refl R$ is a height one progenerator by \ref{reflequiv}(3). 
\end{proof}

Throughout we shall use the  definition of NCCR in the introduction (\ref{NCCR}) due to its suitability for global-local arguments. However, we have the following:
\begin{lemma}\label{NCCRdefequiv}
Let $R$ be a $d$-sCY normal domain, then $M\in\refl R$ gives a NCCR if and only if $\gl\End_R(M)<\infty$ and $\End_R(M)\in\CM R$.
\end{lemma}
\begin{proof}
($\Rightarrow$) obvious.\\
($\Leftarrow$) Set $\Lambda:=\End_{R}(M)$, $d:=\dim R$.  By \ref{3CY-nonlocal}(2)  $\Lambda$ is $d$-sCY hence by \ref{2.14a} $\Lambda$ is a Gorenstein order, with $\gl\Lambda<\infty$.  By \ref{nonsingGoren} $\Lambda$ is non-singular.
\end{proof}

\subsection{Global--local properties}  In this paper we work in the global setting of non-local rings so that we can apply our work to algebraic geometry \cite{IW5}.   To do this requires the following technical lemmas.
\begin{lemma}\label{derivedlocalcomp}
Derived equivalences between module finite $R$-algebras are preserved under localization and completion.
\end{lemma}
\begin{proof}
Let $A$ and $B$ be module finite $R$-algebras with $A$  derived equivalent to $B$ via a tilting complex $T$ \cite[6.5]{Rickard}.  Since Ext groups localize (respectively, complete), $T_\p$ and $\widehat{T_\p}$ both have no self-extensions.  Further $A$ can be constructed from $T$ using cones, shifts and summands of $T$, so using the localizations (respectively, completions) of these triangles we conclude that $A_\p$ can be constructed from $T_\p$ and also $\widehat{A_\p}$ can be reached from $\widehat{T_\p}$.  Thus $T_\p$ is a tilting $A_\p$ complex and $\widehat{T_\p}$ is a tilting $\widehat{A_\p}$ complex.
\end{proof}
The following ensure that membership of $\add M$ can be shown locally or even complete locally, and we will use this often.
\begin{lemma}\label{addsurj}
Let $\Lambda$ be a module finite $R$-algebra, where $R$ is a commutative noetherian ring, and let $M,N\in\mod \Lambda$.  Denote by  $N_0\stackrel{g}\to M$ a right $(\add N)$-approximation of $M$.  Then $\add M\subseteq \add N$ if and only if the induced map $\Hom_\Lambda(M,N_0)\xrightarrow{(\cdot g)} \End_\Lambda(M)$ is surjective.
\end{lemma}
\begin{proof}
($\Leftarrow$) If $\Hom_\Lambda(M,N_0)\xrightarrow{(\cdot g)}\End_\Lambda(M)$ is surjective we may lift ${\rm id}_M$ to obtain a splitting for $g$ and hence $M$ is a summand of $N_0$.\\
($\Rightarrow$) If $M\in\add N$ then there exists $M\stackrel{a}\to N^n\stackrel{b}\to M$ with $ab={\rm id}_M$.  Since $g$ is an approximation, for every $\varphi\in\End_\Lambda(M)$ there is a commutative diagram  
\[
{\SelectTips{cm}{10}
\xy
(30,0)*+{N^n}="2",(45,0)*+{M}="3",
(30,-12.5)*+{N_0}="b2",(45,-12.5)*+{M}="b3"
\ar^{b}"2";"3"
\ar^{g}"b2";"b3"
\ar@{.>}^{\psi}"2";"b2"
\ar^{\varphi}"3";"b3"
\endxy}
\]
Consequently $\varphi=ab\varphi=a\psi g$ and so $\varphi$ is the image of $a\psi$ under the map $(\cdot g)$.
\end{proof}

\begin{prop}\label{addlocal}
Let $\Lambda$ be a module finite $R$-algebra, where $R$ is a commutative noetherian ring, and let $M,N\in\mod \Lambda$.  Then the following are equivalent:\\
\t{1.} $\add M\subseteq\add N$.\\
\t{2.} $\add M_\p\subseteq\add N_\p$ for all $\p\in\Spec R$.\\
\t{3.} $\add M_\m\subseteq\add N_\m$ for all $\m\in\Max R$.\\
\t{4.} $\add \widehat{M}_\p\subseteq\add \widehat{N}_\p$ for all $\p\in\Spec R$.\\
\t{5.} $\add \widehat{M}_\m\subseteq\add \widehat{N}_\m$ for all $\m\in\Max R$.\\
Furthermore we can replace $\subseteq$ by equality throughout and the result is still true.
\end{prop}
\begin{proof}
Let $g$ be as in \ref{addsurj}. Then $g_\p:(N_0)_\p\to M_\p$ is a right $(\add N_\p)$-approximation and $\widehat{g}_\p:\widehat{(N_0)_\p}\to\widehat{M}_\p$ is a right $(\add\widehat{N_\p})$-approximation for any $\p\in \Spec R$.  Since the vanishing of $\Cok(\Hom_\Lambda(M,N_0)\xrightarrow{(\cdot g)}\End_\Lambda(M))$ can be checked locally
or complete locally, all conditions are equivalent.  The last statement holds by symmetry.
\end{proof}

\section{Auslander--Reiten Duality for Non-Isolated Singularities}\label{AR}

Let $R$ be a $d$-dimensional, equi-codimensional CM ring 
with a canonical module $\omega_{R}$, and let $\Lambda$ be an $R$-order.  If $R$ is $d$-sCY, we always choose $\omega_R:=R$. We denote $\u{\CM}\Lambda$ to be the stable category of maximal CM $\Lambda$-modules and $\overline{\CM}\Lambda$ to be the costable category.  By definition these have the same objects as $\CM\Lambda$, but  morphism spaces are defined as $\u{\Hom}_{\Lambda}(X, Y):=\Hom_{\Lambda}(X, Y)/\mathcal{P}(X, Y)$ (respectively $\overline{\Hom}_{\Lambda}(X, Y):=\Hom_{\Lambda}(X, Y)/\mathcal{I}(X, Y)$) where $\mathcal{P}(X, Y)$ (respectively $\mathcal{I}(X, Y)$) is the subspace of morphisms factoring through $\add \Lambda$ (respectively $\add\omega_\Lambda$).  

We denote $\Tr:=\Tr_\Lambda:\u{\mod}\Lambda\to\u{\mod}\Lambda^{\op}$ the Auslander--Bridger transpose duality, and $\Omega_{\Lambda^{\op}}:
\u{\mod}\Lambda^{\op}\to\u{\mod}\Lambda^{\op}$
the syzygy functor.  Then we have \emph{AR translation}
\[
\tau:=\Hom_R(\Omega_{\Lambda^{\op}}^d\Tr_{\Lambda}(-) ,\omega_R):
\underline{\CM}\Lambda\to\overline{\CM}\Lambda.
\]
We denote $D_i:=\Ext_R^{d-i}(-,\omega_R)$ to be the duality of the category of Cohen--Macaulay modules of dimension $i$, so $D_0$ is the Matlis dual (as in \S \ref{3CY}).

If $\Lambda$ is an $R$-order as above we define $\Sing_R \Lambda:=\{\p\in\Spec R: \gl\Lambda_\p>\dim R_\p \}$ to be the singular locus of $\Lambda$ (see \ref{nonsingorder}(2)).  Our main theorem is the following:
\begin{thm}\label{ARduality}
Let $R$ be a $d$-dimensional, equi-codimensional CM ring with a canonical module $\omega_{R}$.  Let $\Lambda$ be an $R$-order with $\dim\Sing_R \Lambda\leq 1$.  Then there exist functorial isomorphisms
\begin{eqnarray*}
\fl \u{\Hom}_{\Lambda}(X,Y)&\cong& D_0(\fl\Ext^1_{\Lambda}(Y,\tau X)),\\
\frac{\u{\Hom}_{\Lambda}(X,\Omega Y)}{\fl\u{\Hom}_{\Lambda}(X,\Omega Y)}&\cong&
D_1\left(\frac{\Ext^1_{\Lambda}(Y,\tau X)}{\fl\Ext^1_{\Lambda}(Y,\tau X)}\right)
\end{eqnarray*}
for all $X,Y\in\CM \Lambda$.
\end{thm}
In fact \ref{ARduality} immediately follows from the more general \ref{ARduality2} below.  Recall for $X\in\mod \Lambda$ that ${\rm NP}(X):=\{\p\in\Spec R: X_\p \notin\proj\Lambda_\p \}$ and $\CM_1\Lambda:=\{ X\in\CM \Lambda:\dim {\rm NP}(X)\leq 1 \}$. 
\begin{thm}\label{ARduality2}
Let $R$ be a $d$-dimensional, equi-codimensional CM ring with a canonical module $\omega_{R}$.  Let $\Lambda$ be an $R$-order.  Then there exist functorial isomorphisms
\begin{eqnarray*}
\fl \u{\Hom}_{\Lambda}(X,Y)&\cong& D_0(\fl\Ext^1_{\Lambda}(Y,\tau X)),\\
\frac{\u{\Hom}_{\Lambda}(X,\Omega Y)}{\fl\u{\Hom}_{\Lambda}(X,\Omega Y)}&\cong&
D_1\left(\frac{\Ext^1_{\Lambda}(Y,\tau X)}{\fl\Ext^1_{\Lambda}(Y,\tau X)}\right)
\end{eqnarray*}
for all $X\in\CM_1\Lambda$ and $Y\in\CM \Lambda$
\end{thm} 
The proof of \ref{ARduality2} requires the next three easy lemmas.  For a finitely generated $R$-module $M$, denote $E_{R}(M)$ to be the injective hull of $M$.
\begin{lemma}\label{Homzero}
If $X\in \mod R$ and $Y\in \Mod R$ satisfies $\Supp X\cap\Ass Y=\emptyset$, then $\Hom_R(X,Y)=0$.  
\end{lemma}
\begin{proof}
Let $f:X\to Y$ be any map and $X':=\Im f$.  Then $X'\subset Y$ is a finitely generated submodule such that $\Ass X'\subset \Supp X\cap\Ass Y$.  Thus $\Ass X'=\emptyset$ and so since $X'$ is finitely generated, $X'=0$.
\end{proof}

\begin{lemma}\cite[3.2.7(a)]{BH}\label{asslemma}
We have $\Ass E_{R}(R/\p)=\{ \p\}$.
\end{lemma}

Now recall that if $R$ is a $d$-dimensional equi-codimensional CM ring with canonical $\omega_R$ then the minimal $R$-injective resolution of $\omega_R$, denoted
\begin{align}
0\to \omega_{R}\to I_0\rightarrow I_1\to\hdots \to I_{d-1}\to I_d\to 0,\label{min inj res}
\end{align}
satisfies
\begin{equation}\label{terms}
I_i\stackrel{\mbox{\scriptsize\cite[3.2.9, 3.3.10(b)]{BH}}}{=}\bigoplus_{\p: \t{ht}\p=i} E(R/\p)\stackrel{\scriptsize{\ref{heightcoheight}}}{=}\bigoplus_{\p:\dim R/\p=d-i} E(R/\p).
\end{equation}
In particular the Matlis dual is $D_0=\Hom_R(-,I_d)$.  

\begin{lemma}\label{key exact sequence}
Let $R$ be a $d$-dimensional equi-codimensional CM ring with canonical module $\omega_R$.  If $N\in\mod R$ with $\dim_RN\leq 1$, then\\
\t{(1)} $\Ext^{d-1}_R(N,\omega_R)\cong \Ext^{d-1}_R(\frac{N}{\fl N},\omega_R)$.\\
\t{(2)} $\Ext^d_R(N,\omega_R)\cong \Ext^d_R(\fl N,\omega_R)$.\\
\t{(3)} There is an exact sequence
\[
0\to\Ext^{d-1}_R(\tfrac{N}{\fl N},\omega_R)\to\Hom_R(N,I_{d-1})\to\Hom_R(N,I_{d})\to\Ext^d_R(\fl N,\omega_R)\to 0.
\]
\end{lemma}
\begin{proof}
There is an exact sequence $0\to\fl N\to N\to\tfrac{N}{\fl N}\to 0$ from which applying $\Hom_R(-,\omega_R)$ gives
\[
\Ext^{d-2}_R(\fl N,\omega_R)\to\Ext^{d-1}_R(\tfrac{N}{\fl N},\omega_R)\to\Ext^{d-1}_R(N,\omega_R)\to \Ext^{d-1}_R({\fl N},\omega_R).
\]
Since $\dim_R(\fl N)=0$, it is well-known that $\Ext^{i}_R(\fl N,\omega_R)=0$ for all $i\neq d$ \cite[3.5.11]{BH}, hence the outer two ext groups vanish, establishing (1).  But we also have an exact sequence
\[
\Ext^{d}_R(\tfrac{N}{\fl N},\omega_R)\to\Ext^{d}_R(N,\omega_R)\to \Ext^{d}_R({\fl N},\omega_R)\to\Ext^{d+1}_R(\tfrac{N}{\fl N},\omega_R)
\]
and so since $\tfrac{N}{\fl N}$ has positive depth (or is zero) at all maximal ideals, $\Ext^i_R(\tfrac{N}{\fl N},\omega_R)=0$ for all $i>d-1$, again by \cite[3.5.11]{BH}.  This establishes (2).  For (3), note first that $\Hom_R(N,I_{d-2})=0$ by \ref{Homzero}, since by \ref{asslemma} and the assumption that $\dim N\leq 1$ we have that $\Supp N\cap\Ass I_{d-2}=\emptyset$.  Consequently simply applying $\Hom_R(N,-)$ to (\ref{min inj res}) gives an exact sequence
\[
0\to\Ext^{d-1}_R(N,\omega_R)\to\Hom_R(N,I_{d-1})\to\Hom_R(N,I_{d})\to\Ext^d_R(N,\omega_R)\to 0,
\]
and so (3) follows from (1) and (2).
\end{proof}

We are now ready to prove \ref{ARduality2}.  To ease notation, we often drop $\Tor$ and $\Ext$, and for example write ${}_{R}^1(X,Y)$ for $\Ext_{R}^1(X,Y)$, and ${}^{R}_1(X,Y)$ for $\Tor^{R}_1(X,Y)$.
\begin{proof}
Denote $T:=\Tr X$. Now since $Y\in\CM \Lambda$ we have $\Ext^{i}_{R}(Y,\omega_{R})=0$ for all $i>0$ and so applying $\Hom_R(Y,-)$ to (\ref{min inj res}) gives an exact sequence
\[
0\to {}_{R}(Y,\omega_{R})\to {}_{R}(Y,I_0)\to {}_{R}(Y,I_1)\to\hdots\to {}_{R}(Y,I_{d-1})\to {}_{R}(Y,I_d)\to 0
\]
of left $\Lambda^{\rm op}$-modules, which we split into short exact sequences as
\[
\SelectTips{cm}{10}
\xy0;/r.275pc/:
\POS(5,0)*+{0}="0",(15,0)*+{{}_{R}(Y,\omega_{R})}="1",(30,0)*+{{}_{R}(Y,I_0)}="2",(45,0)*+{{}_{R}(Y,I_1)}="3",(60,0)*+{{}_{R}(Y,I_2)}="4",(69,0)*+{\,}="4a",(69.5,0)*+{\,}="4b",(80,0)*+{{}_{R}(Y,I_{d-2})}="5",(97,0)*+{{}_{R}(Y,I_{d-1})}="6",(112,0)*+{{}_{R}(Y,I_{d})}="7",(121,0)*+{0}="8",(37.5,-6.5)*+{C_1}="b1",(52.5,-6.5)*+{C_2}="b2",(87.5,-6.5)*+{C_{d-1}}="b3"
\ar"0";"1"
\ar"1";"2"
\ar"2";"3"
\ar"3";"4"
\ar"4";"4a"
\ar@{}"4a";"4b"_{\hdots}
\ar"4b";"5"
\ar"5";"6"
\ar"6";"7"
\ar"7";"8"
\ar@{->>}"2";"b1"
\ar@{^{(}->}"b1";"3"
\ar@{->>}"3";"b2"
\ar@{^{(}->}"b2";"4"
\ar@{->>}"5";"b3"
\ar@{^{(}->}"b3";"6"
\endxy .
\]
Applying $\Hom_{\Lambda^{\rm op}}(T,-)$ gives exact sequences
{
\[
\begin{array}{c}
{\SelectTips{cm}{10}
\xy0;/r.37pc/:
(0,0)*+{{}^{1}_{\Lambda^{\rm op}}(T,{}_{R}(Y,I_{d-1}))}="1",
(18,0)*+{{}^{1}_{\Lambda^{\rm op}}(T,{}_{R}(Y,I_d))}="2",
(34,0)*+{{}^{2}_{\Lambda^{\rm op}}(T,C_{d-1})}="3",
(51,0)*+{{}^{2}_{\Lambda^{\rm op}}(T,{}_{R}(Y,I_{d-1}))}="4",
(69,0)*+{{}^{2}_{\Lambda^{\rm op}}(T,{}_{R}(Y,I_d))}="5"
\ar"1";"2"
\ar"2";"3"
\ar"3";"4"
\ar"4";"5"
\endxy}\\
{\SelectTips{cm}{10}
\xy0;/r.37pc/:
(10,0)*+{{}^{2}_{\Lambda^{\rm op}}(T,{}_{R}(Y,I_{d-2}))}="1",
(28,0)*+{{}^{2}_{\Lambda^{\rm op}}(T,C_{d-1})}="2",
(46,0)*+{{}^{3}_{\Lambda^{\rm op}}(T,C_{d-2})}="3",
(64,0)*+{{}^{3}_{\Lambda^{\rm op}}(T,{}_{R}(Y,I_{d-2}))}="4"
\ar"1";"2"
\ar"2";"3"
\ar"3";"4"
\endxy}\\
\vdots\\
{\SelectTips{cm}{10}
\xy0;/r.37pc/:
(10,0)*+{{}^{d-1}_{\Lambda^{\rm op}}(T,{}_{R}(Y,I_1))}="1",
(28,0)*+{{}^{d-1}_{\Lambda^{\rm op}}(T,C_2)}="2",
(46,0)*+{{}^{d}_{\Lambda^{\rm op}}(T,C_1)}="3",
(64,0)*+{{}^{d}_{\Lambda^{\rm op}}(T,{}_{R}(Y,I_1))}="4",(70,0)*+{\,}="5"
\ar"1";"2"
\ar"2";"3"
\ar"3";"4"
\ar@{}"4";"5"
\endxy}\\
{\SelectTips{cm}{10}
\xy0;/r.37pc/:
(10,0)*+{{}^{d}_{\Lambda^{\rm op}}(T,{}_{R}(Y,I_0))}="1",
(28,0)*+{{}^{d}_{\Lambda^{\rm op}}(T,C_1)}="2",
(46,0)*+{{}^{d+1}_{\Lambda^{\rm op}}(T,{}_{R}(Y,\omega_{R}))}="3",
(64,0)*+{{}^{d+1}_{\Lambda^{\rm op}}(T,{}_{R}(Y,I_0))}="4"
\ar"1";"2"
\ar"2";"3"
\ar"3";"4"
\endxy}.
\end{array}
\]
}By the functorial isomorphism \cite[VI.5.1]{CE99}
\[
\Ext^j_{\Lambda}(A,{}_{R}(B,I))\cong\Hom_R(\Tor^{\Lambda}_j(A,B),I)
\]
where $I$ is an injective $R$-module, we have exact sequences
\begin{eqnarray}
{\SelectTips{cm}{10}
\xy0;/r.37pc/:
(0,0)*+{{}_R({}^{\Lambda}_1(T,Y),I_{d-1})}="1",
(17,0)*+{{}_R({}^{\Lambda}_1(T,Y),I_d)}="2",
(32,0)*+{{}^{2}_{\Lambda^{\rm op}}(T,C_{d-1})}="3",
(48,0)*+{{}_R({}_{2}^{\Lambda}(T,Y),I_{d-1})}="4",
(65,0)*+{{}_R({}_{2}^{\Lambda}(T,Y),I_d)}="5"
\ar"1";"2"
\ar"2";"3"
\ar"3";"4"
\ar"4";"5"
\endxy}\label{1stline}\\
\left.\begin{array}{c}
{\SelectTips{cm}{10}
\xy0;/r.37pc/:
(10,0)*+{{}_R({}_{2}^{\Lambda}(T,Y),I_{d-2})}="1",
(28,0)*+{{}^{2}_{\Lambda^{\rm op}}(T,C_{d-1})}="2",
(46,0)*+{{}^{3}_{\Lambda^{\rm op}}(T,C_{d-2})}="3",
(64,0)*+{{}_R({}_{3}^{\Lambda}(T,Y),I_{d-2})}="4"
\ar"1";"2"
\ar"2";"3"
\ar"3";"4"
\endxy}\\
\vdots\\
{\SelectTips{cm}{10}
\xy0;/r.37pc/:
(10,0)*+{{}_R({}_{d-1}^{\Lambda}(T,Y),I_1)}="1",
(28,0)*+{{}^{d-1}_{\Lambda^{\rm op}}(T,C_2)}="2",
(46,0)*+{{}^{d}_{\Lambda^{\rm op}}(T,C_1)}="3",
(64,0)*+{{}_R({}_{d}^{\Lambda}(T,Y),I_1)}="4",(70,0)*+{\,}="5"
\ar"1";"2"
\ar"2";"3"
\ar"3";"4"
\ar@{}"4";"5"
\endxy}\\
{\SelectTips{cm}{10}
\xy0;/r.37pc/:
(10,0)*+{{}_R({}_{d}^{\Lambda}(T,Y),I_0)}="1",
(28,0)*+{{}^{d}_{\Lambda^{\rm op}}(T,C_1)}="2",
(46,0)*+{{}^{d+1}_{\Lambda^{\rm op}}(T,{}_{R}(Y,\omega_{R}))}="3",
(64,0)*+{{}_R({}_{d+1}^{\Lambda}(T,Y),I_0)}="4"
\ar"1";"2"
\ar"2";"3"
\ar"3";"4"
\endxy}.
\end{array}\right\}\label{2ndgroup}
\end{eqnarray}
By the assumption that $X\in\CM_1\Lambda$, for all primes $\p$ such that $\dim R/\p> 1$, we have $X_{\p}\in\proj\Lambda_{\p}$ and so $T_{\p}\in\proj\Lambda_{\p}^{\rm op}$.  Thus for all such primes and any $j>0$, we have $\Tor^{\Lambda}_j(T,Y)_\p\cong\Tor_j^{\Lambda_\p}(T_\p,Y_\p)=0$.  Hence for all $i=0,1,\hdots, d-2$ and all $j>0$, by \ref{asslemma} and \eqref{terms} it follows that $\Supp\Tor^{\Lambda}_j(T,Y)\cap\Ass I_i=\emptyset$ and so consequently $ \Hom_R(\Tor^{\Lambda}_j(T,Y),I_i)=0$ for all $j>0$ and all $i=0,1,\hdots ,d-2$ by \ref{Homzero}.
Thus (\ref{2ndgroup}) reduces to
\begin{eqnarray*}
\Ext_{\Lambda^{\rm op}}^{2}(T,C_{d-1})\cong \Ext_{\Lambda^{\rm op}}^3(T,C_{d-2})\cong\hdots\cong \Ext_{\Lambda^{\rm op}}^{d}(T,C_{1})\cong \Ext_{\Lambda^{\rm op}}^{d+1}(T,{}_{R}(Y,\omega_{R}))
\end{eqnarray*}
and so it follows that 
\begin{equation}\label{tau}
\mbox{\small$\Ext_{\Lambda^{\rm op}}^{2}(T,C_{d-1})\cong\Ext^1_{\Lambda^{\rm op}}(\Omega_{\Lambda^{\rm op}}^d\Tr X,{}_{R}(Y,\omega_{R}))\cong \Ext^1_{\Lambda}(Y,{}_{R}(\Omega_{\Lambda^{\rm op}}^d\Tr X,\omega_{R}))=\Ext^1_\Lambda(Y,\tau X)$}.
\end{equation}
Using the well-known functorial isomorphism \cite[3.2]{Aus78},\cite[3.9]{Y}
\begin{equation}\label{tor u}
\Tor^{\Lambda}_1(\Tr X,Y)\cong \u{\Hom}_\Lambda(X,Y),
\end{equation}
(\ref{1stline}), (\ref{tau}) and (\ref{tor u}) combine to give the following commutative diagram of exact sequences: 
{\scriptsize
\[
{\SelectTips{cm}{10}
\xy0;/r.37pc/:
(0.5,0)*+{{}_{R}(\Tor^{\Lambda}_1(T,Y),I_{d-1})}="1",
(19,0)*+{{}_{R}(\Tor^{\Lambda}_1(T,Y),I_d)}="2",
(35,0)*+{\Ext_{\Lambda^{\rm op}}^2(T,C_{d-1})}="3",
(52,0)*+{{}_{R}(\Tor^{\Lambda}_2(T,Y),I_{d-1})}="4",
(72,0)*+{{}_{R}(\Tor^{\Lambda}_2(T,Y),I_d)}="5",
(0.5,-7)*+{{}_{R}(\u{\Hom}_{\Lambda}(X,Y),I_{d-1})}="b0",
(19,-7)*+{{}_{R}(\u{\Hom}_{\Lambda}(X,Y),I_d)}="b1",
(35,-7)*+{\Ext_\Lambda^1(Y,\tau X)}="b2",
(52,-7)*+{{}_{R}(\u{\Hom}_{\Lambda}(X,\Omega Y),I_{d-1})}="b3",
(72,-7)*+{{}_{R}(\u{\Hom}_{\Lambda}(X,\Omega Y),I_d)}="b4"
\ar"1";"2"
\ar"2";"3"
\ar"3";"4"
\ar"4";"5"
\ar"b0";"b1"
\ar^(0.55){\psi}"b1";"b2"
\ar"b2";"b3"
\ar"b3";"b4"
\ar_{\cong}^{\eqref{tor u}}"1";"b0"
\ar_{\cong}^{\eqref{tor u}}"2";"b1"
\ar_{\cong}^{\eqref{tau}}"3";"b2"
\ar_{\cong}^{\eqref{tor u}}"4";"b3"
\ar_{\cong}^{\eqref{tor u}}"5";"b4"
\endxy}
\]
}which we splice as
\begin{align}
{}_{R}(\u{\Hom}_{\Lambda}(X,Y),I_{d-1})\to{}_{R}(\u{\Hom}_{\Lambda}(X,Y),I_d)\to
\Im\psi\to 0\label{splice1}
\end{align}
\begin{align}
0\to\Im\psi\to\Ext_\Lambda^1(Y,\tau X)\to\Cok\psi\to 0\label{splice2}
\end{align}
\begin{align}
0\to\Cok\psi\to{}_{R}(\u{\Hom}_{\Lambda}(X,\Omega Y),I_{d-1})\to {}_{R}(\u{\Hom}_{\Lambda}(X,\Omega Y),I_d).\label{splice3}
\end{align}
By applying \ref{key exact sequence}(3) to $N:=\u{\Hom}_\Lambda(X,Y)$ and comparing to (\ref{splice1}) we see that 
\begin{align}
\Im\psi\cong\Ext^{d}_R(\fl\u{\Hom}_\Lambda(X,Y),\omega_R)=D_0(\fl\u{\Hom}_\Lambda(X,Y)).\label{im is ext}
\end{align}  
Similarly, applying \ref{key exact sequence}(3) to $N:=\u{\Hom}_\Lambda(X,\Omega Y)$ and comparing to (\ref{splice3}) we see that 
\begin{align}
\Cok\psi\cong\Ext^{d-1}_R\left(\tfrac{\u{\Hom}_\Lambda(X,\Omega Y)}{\fl\u{\Hom}_\Lambda(X,\Omega Y)},\omega_R\right)=D_1\left(\tfrac{\u{\Hom}_\Lambda(X,\Omega Y)}{\fl\u{\Hom}_\Lambda(X,\Omega Y)}\right).\label{cok is ext}
\end{align}

Now \eqref{im is ext} and \eqref{cok is ext} show that $\Im\psi=\fl\Ext_\Lambda^1(Y,\tau X)$, and together with (\ref{im is ext}) this establishes the first required isomorphism, and together with (\ref{splice2}) and (\ref{cok is ext}) this establishes the second required isomorphism.  
\end{proof}

When $R$ has only isolated singularities the above reduces to classical Auslander--Reiten duality.  If moreover $R$ is a $d$-sCY ring with isolated singularities (i.e.\ $R$ is a Gorenstein $d$-dimensional equi-codimensional ring with isolated singularities), AR duality implies that the category $\u{\CM} R$ is $(d-1)$-CY.  We now apply \ref{ARduality} to possibly non-isolated $d$-sCY rings to obtain some analogue of this $(d-1)$-CY property (see \ref{CMsymmcomplete}(1) below).  The following lemma is well-known \cite[III.1.3]{Aus78}.

\begin{lemma}\label{tausyz}
Suppose $R$ is $d$-sCY and let $\Lambda$ be a symmetric $R$-order.  Then $\tau\cong\Omega_{\Lambda}^{2-d}$.
\end{lemma}
\begin{proof}
We have $\Omega^2\Tr(-)\cong \Hom_\Lambda(- ,\Lambda)$.
Since $R$ is $d$-sCY (and so $\omega_R:=R$), and $\Lambda$ is symmetric, we have
$\Omega^2\Tr(-)\cong \Hom_\Lambda(- ,\Lambda)\cong \Hom_R(-,R)$.
Thus
\begin{align*}
\tau=\Hom_R(\Omega_{\Lambda^{\rm op}}^d\Tr(-) ,R)
&\cong \Hom_R(\Omega_{\Lambda^{\rm op}}^{d-2}\Hom_R(-,R),R)\\
&\cong\Omega_\Lambda^{2-d}\Hom_R(\Hom_R(-,R),R)\\
&\cong\Omega_\Lambda^{2-d}.
\end{align*}
\end{proof}

\begin{cor}\label{CMsymmcomplete}
Let $R$ be a $d$-sCY ring and let $\Lambda$ be a symmetric $R$-order with $\dim\Sing_R\Lambda\leq 1$.  Then\\
\t{(1)} There exist functorial isomorphisms
\begin{align*}
\fl \u{\Hom}_{\Lambda}(X,Y)&\cong D_0(\fl\u{\Hom}_{\Lambda}(Y,X[d-1])),\\
\frac{\u{\Hom}_{\Lambda}(X,Y)}{\fl\u{\Hom}_{\Lambda}(X,Y)}&\cong
D_1\left(\frac{\u{\Hom}_{\Lambda}(Y,X[d-2])}{\fl\u{\Hom}_{\Lambda}(Y,X[d-2])}\right)
\end{align*}for all $X,Y\in\CM \Lambda$.\\
\t{(2)} If $d=3$ then for all $X,Y\in\CM \Lambda$, $\Hom_\Lambda(X,Y)\in \CM R$ if and only if $\Hom_\Lambda(Y,X)\in \CM R$.
\end{cor}
\begin{proof}
(1)  It is well-known that in $\underline{\CM}\Lambda$ the shift functor $[1]=\Omega^{-1}$ so by \ref{tausyz} $\tau=[d-2]$.  Thus the result follows directly from \ref{ARduality}, using the fact that $\underline{\Hom}_\Lambda(A,B[1])\cong\Ext^1_\Lambda(A,B)$ for all $A,B\in\CM\Lambda$.\\
(2) Immediate from (1) and \ref{reflandCM}.
\end{proof}
Note that by \ref{HomsymmDirect}, \ref{CMsymmcomplete}(2) also holds for arbitrary $d$ (with no assumptions on the singular locus) provided that $R$ is normal.  When $R$ is not necessarily normal, we improve \ref{CMsymmcomplete}(2) in \ref{HomsymmNOTnormal} below.

\section{Modifying and Maximal Modifying Modules}\label{gMR}
Motivated by the fact that $\Spec R$ need not have a crepant resolution, we want to be able to control algebras of infinite global dimension and hence partial resolutions of singularities.  

\subsection{Modifications in Dimension $d$} We begin with our main definition.
\begin{defin}
Let $R$ be a $d$-dimensional CM ring, $\Lambda$ a module finite $R$-algebra.  We call $N\in\refl \Lambda$ a \emph{modifying module} if $\End_\Lambda(N)\in\CM R$, whereas we call $N$ a \emph{maximal modifying (MM) module} if $N$ is modifying and furthermore it is maximal with respect to this property, that is to say if there exists $X\in\refl \Lambda$ with $N\oplus X$ modifying, necessarily $X\in\add N$.  
\end{defin}
The following is immediate from the definition.

\begin{lemma}\label{CMiffgen}
Suppose $R$ is a $d$-dimensional CM ring, $\Lambda$ a module finite $R$-algebra.  Then\\
\t{(1)} The modifying $\Lambda$-modules which are generators are always CM.\\
\t{(2)} If further $\Lambda$ is a Gorenstein $R$-order then the MM generators are precisely the MM modules which are CM.
\end{lemma}
\begin{proof}
(1) Since $M$ is a modifying $\End_\Lambda(M)\in\CM R$, and since $M$ is a generator, $\Lambda\in\add M$.  Hence $M^{\oplus n}\cong\Hom_\Lambda(\Lambda,M^{\oplus n})\in\CM R$ is a summand of $\Hom_\Lambda(M^{\oplus n},M^{\oplus n})\cong\End_{\Lambda}(M)^{\oplus n^2}$ for some $n\in\mathbb{N}$, thus $M^{\oplus n}$ and so consequently $M$ itself are CM.  \\
(2) Conversely suppose that $M$ is an MM module which is CM.  Then certainly we have $\Hom_\Lambda(\Lambda,M)\cong M\in\CM R$ and also $\Hom_\Lambda(M,\omega_\Lambda)\cong\Hom_R(M,\omega_R)\in\CM R$.  Since $\Lambda$ is a Gorenstein $R$-order, $\add\Lambda=\add\omega_\Lambda$ by \ref{addLambda=addOmega}, thus $\End_\Lambda(M\oplus\Lambda)\in\CM R$. Since $M$ is maximal necessarily $\Lambda\in\add M$.
\end{proof}

Under assumptions on the singular locus, we can check whether a CM module is modifying by examining Ext groups.  The following is a generalization of \ref{reflandCM} for $d=3$, and \cite[2.5.1]{IyamaAR} for isolated singularities.
\begin{thm}\label{depthford}
Suppose that $R$ is $d$-sCY with $d=\dim R\geq 2$ and $\dim\Sing R\leq 1$, let $\Lambda$ be an $R$-order and let $X,Y\in\CM \Lambda$.  Then $\Hom_\Lambda(X,Y)\in\CM R$ if and only if $\Ext^i_\Lambda(X,Y)=0$ for all $i=1,\hdots,d-3$ and $\fl\Ext^{d-2}_\Lambda(X,Y)=0$.
\end{thm}
\begin{proof}
Without loss of generality, we can assume that $R$ is local.
Consider a projective resolution $\hdots\to P_1\to P_0\to X\to0$.
Applying $\Hom_\Lambda(-,Y)$, we have a complex
\begin{multline}\label{(X,Y)}
0\to\Hom_\Lambda(X,Y)\to\Hom_\Lambda(P_0,Y)\to\hdots\\
\hdots\to\Hom_\Lambda(P_{d-3},Y)\to \Hom_\Lambda(\Omega^{d-2}X,Y)\to\Ext^{d-2}_\Lambda(X,Y)\to0
\end{multline}
with homologies $\Ext^i_\Lambda(X,Y)$ at $\Hom_\Lambda(P_i,Y)$ for $i=1,\hdots,d-3$.

($\Leftarrow$) By assumption the sequence \eqref{(X,Y)} is exact. Since
$\depth\Ext^{d-2}_\Lambda(X,Y)\ge1$, $\depth\Hom_\Lambda(\Omega^{d-2}X,Y)\ge2$ by \ref{depthofhom},
and $\Hom_\Lambda(P_i,Y)\in\CM R$, we have $\Hom_\Lambda(X,Y)\in\CM R$ by
the depth lemma.

($\Rightarrow$) By \ref{Extisfl} and our assumption $\dim\Sing R\le 1$, we have
$\dim\Ext^i_\Lambda(X,Y)\le1$ for any $i>0$. 
Assume $\Ext^i_\Lambda(X,Y)\neq0$ for some $i=1,\hdots,d-3$. Take minimal $i$
such that $\Ext^i_\Lambda(X,Y)\neq0$. We have an exact sequence
\begin{multline*}
0\to\Hom_\Lambda(X,Y)\to\Hom_\Lambda(P_0,Y)\to\hdots\\
\hdots\to\Hom_\Lambda(P_{i-1},Y)\to
\Hom_\Lambda(\Omega^iX,Y)\to\Ext^i_\Lambda(X,Y)\to0.
\end{multline*}
Localizing at prime ideal $\p$ of $R$ with height at least $d-1$ and using
$\depth_{R_\p}\Hom_{\Lambda_\p}(\Omega^iX_\p,Y_\p)\ge2$ by \ref{depthofhom}, $\Hom_{\Lambda_\p}(P_i{}_\p,Y_\p)\in\CM R_\p$
and $\Hom_{\Lambda_\p}(X_\p,Y_\p)\in\CM R_\p$ by our assumption, we have
$\depth_{R_\p}\Ext^i_\Lambda(X,Y)_\p\ge 1$ by the depth lemma.
If $\p$ has height $d-1$, then $\dim_{R_\p}\Ext^i_\Lambda(X,Y)_\p=0$ and we have $\Ext^i_\Lambda(X,Y)_\p=0$. Thus $\dim\Ext^i_\Lambda(X,Y)=0$ holds, and we have $\Ext^i_\Lambda(X,Y)=0$, a contradiction.

Thus we have $\Ext^i_\Lambda(X,Y)=0$ for all $i=1,\hdots,d-3$ and so the sequence \eqref{(X,Y)} is exact. Since
$\depth\Hom_\Lambda(\Omega^{d-2}X,Y)\ge2$, $\Hom_\Lambda(P_i,Y)\in\CM R$ and
$\Hom_\Lambda(X,Y)\in\CM R$, we have $\depth\Ext^{d-2}_\Lambda(X,Y)\ge1$
by the depth lemma.
\end{proof}

The following improves \ref{CMsymmcomplete}(2).
\begin{cor}\label{HomsymmNOTnormal}
Let $R$ be a $d$-sCY ring and let $\Lambda$ be a symmetric $R$-order with $\dim \Sing_R\Lambda\leq 1$. Then for all $X,Y\in\CM\Lambda$,
$\Hom_\Lambda(X,Y)\in\CM R$ if and only if $\Hom_\Lambda(Y,X)\in\CM R$.
\end{cor}
\begin{proof}
By the statement and proof of \ref{depthofhom}, when $d\leq 2$, 
if $M$ and $N$ are CM then so are both $\Hom_\Lambda(M, N)$ and $\Hom_\Lambda(N,M)$.  Thus we can assume that $d\geq 3$.  By symmetry, we need only show ($\Rightarrow$). Assume that $\Hom_\Lambda(X,Y)\in\CM R$.
Then by \ref{depthford} $\Ext_\Lambda^i(X,Y)=0$ for any $i=1,\hdots,d-3$ and $\fl\Ext_\Lambda^{d-2}(X,Y)=0$.
Since ${\displaystyle\frac{\Ext_\Lambda^i(X,Y)}{\fl\Ext_\Lambda^i(X,Y)}}=0$ for any $i=1,\hdots,d-3$,
the $D_1$ duality in \ref{CMsymmcomplete}(1) implies ${\displaystyle\frac{\Ext_\Lambda^i(Y,X)}{\fl\Ext_\Lambda^i(Y,X)}}=0$ for any $i=1,\hdots,d-3$.
Thus $\Ext_\Lambda^i(Y,X)$ has finite length for any $i=1,\hdots,d-3$.
Since $\fl\Ext_\Lambda^i(X,Y)=0$ for any $i=1,\hdots,d-2$, the $D_0$ duality in
\ref{CMsymmcomplete}(1) implies $\fl\Ext_\Lambda^i(Y,X)=0$ for any $i=1,\hdots,d-2$.
Consequently we have $\Ext_\Lambda^i(Y,X)=0$ for any $i=1,\hdots,d-3$ and
$\fl\Ext_\Lambda^{d-2}(Y,X)=0$. 
Again by \ref{depthford} we have $\Hom_\Lambda(Y,X)\in\CM R$.
\end{proof}

Recall from \ref{NCCR} the definition of an NCCR.  The following asserts that, in arbitrary dimension, NCCRs are a special case of MMAs:
\begin{prop} \label{NCCRgiveMM}
Let $R$ be a $d$-dimensional, normal, equi-codimensional
CM ring with a canonical module $\omega_R$ (e.g.\ if $R$ is a normal $d$-sCY ring). Then reflexive $R$-modules $M$ giving NCCRs are MM modules.
\end{prop}
\begin{proof}
Assume that $X\in \refl R$ satisfies $\End_R(M\oplus X)\in \CM R$.
Then $\Hom_R(M,X)\in \CM\Gamma$ for $\Gamma:=\End_R(M)$.
By \ref{nonsingGoren} we have $\Hom_R(M,X)\in \proj \Gamma$.
By \ref{reflequiv}(4) $X\in \add M$ as required.
\end{proof}

We now investigate the derived equivalence classes of modifying algebras, maximal modifying algebras, and NCCRs. 

\begin{thm}\label{closed under derived equivalences}
Let $R$ be a normal $d$-sCY ring, then\\
\t{(1)} Modifying algebras of $\Lambda$ are closed under derived equivalences.\\
\t{(2)} NCCRs of $\Lambda$ are closed under derived equivalences.
\end{thm}
\begin{proof}
(1) Let $\Lambda=\End_R(M)$ be a modifying algebra of $R$, and let $\Gamma$ be a ring that is derived equivalent to $\Lambda$.
Then $\Gamma$ is a module finite $R$-algebra since it is the
endomorphism ring of a tilting complex of $\Lambda$.  Since $\Lambda$ is a modifying algebra of $R$, it is $d$-sCY by \ref{3CY-nonlocal}(2). But $d$-sCY algebras are closed under derived equivalences, hence $\Gamma$ is also $d$-sCY and so $\Gamma\in\CM R$ by \ref{2.14a}.  In particular, $\Gamma$ is reflexive as an $R$-module.

Now we fix a height one prime ideal $\p$ of $R$.
Since $M_\p$ is a free $R_\p$-module of finite rank, $\Lambda_\p=\End_{R_\p}(M_\p)$ is a full matrix algebra of $R_\p$.
Since $R_\p$ is local, the Morita equivalence class of $R_\p$ coincides with the derived equivalence class of $R_\p$ \cite[2.12]{RZ}, so we have that $\Gamma_\p$ is Morita equivalent to $R_\p$.  Thus $\Gamma$ satisfies the conditions in \ref{Auslander-Goldman}, so there exists a reflexive $R$-module $N$ such that  $\Gamma\cong\End_R(N)$ as $R$-algebras. We have already observed that $\Gamma\in\CM R$, hence it is a modifying algebra of $R$.\\
(2) Since $R$ is normal $d$-sCY, by \ref{NCCRdefequiv} NCCRs of $R$ are nothing but modifying algebras of $R$ which have finite global dimension.  We know modifying algebras are $d$-sCY by \ref{3CY-nonlocal}, so the result follows by combining (1) and \ref{IRDB}.
\end{proof}

\begin{question}\label{question MMA}
Let $R$ be a normal $d$-sCY ring.  Are the maximal modifying algebras of $R$ closed under derived equivalences?
\end{question}

We now show that the question has a positive answer when $d\geq 2$ provided that $\dim\Sing R\leq 1$.  In particular, this means that \ref{question MMA} is true when $d\leq 3$.

\begin{thm}\label{MMAs dim3 closed}
Suppose $R$ is a normal $d$-sCY ring with $\dim R=d\geq 2$ and $\dim\Sing R\leq 1$.  Let $N$ be a modifying module and set $\Gamma:=\End_R(N)$.  Then\\
\t{(1)} Then $N$ is MM if and only if $\underline{\CM}\Gamma$ has no non-zero objects $Y$ satisfying $\underline{\Hom}_{\Gamma}(Y,Y[i])=0$ for all $i=1,\hdots,d-3$ and $\fl\underline{\Hom}_{\Gamma}(Y,Y[d-2])=0$.\\
\t{(2)}  MMAs are closed under derived equivalences.
\end{thm}
\begin{proof}
(1) By reflexive equivalence \ref{reflequiv}(4) it is easy to show that there exists $X\in\refl R$ with $X\notin\add N$ such that $\End_R(N\oplus X)\in\CM R$ if and only if there exists $Y\in\CM \Gamma$ with $Y\notin\add\Gamma$ such that $\End_\Gamma(Y)\in\CM R$.  Since for $Y\in\CM\Gamma$ we have $\Ext^i_\Gamma(Y,Y)=\underline{\Hom}_{\Gamma}(Y,Y[i])$, by \ref{depthford} we have the assertion.\\
(2) Suppose that $\Lambda$ is derived equivalent to $\Gamma=\End_R(N)$ where $\Gamma$ is an MMA.  By \ref{closed under derived equivalences}(1) we know that $\Lambda\cong\End_R(M)$ for some modifying $M$.  Since the equivalence $\D(\Mod\Lambda)\simeq\D(\Mod\Gamma)$ induces equivalences $\Db(\mod\Lambda)\simeq\Db(\mod\Gamma)$ and $\Kb(\proj\Lambda)\simeq\Kb(\proj\Gamma)$ by \cite[8.1, 8.2]{Rickard}, we have $\underline{\CM}\Lambda\simeq\underline{\CM}\Gamma
$ by \cite[4.4.1]{Buch}.  By (1), the property of being an MMA can be characterized on the level of this stable category, hence $\Lambda$ is also an MMA.
\end{proof}

\subsection{Derived Equivalence in Dimension $3$}
We now restrict to dimension three.  In this case, we can strengthen \ref{closed under derived equivalences} to obtain one of our main results (\ref{closed in dimension three}).  Leading up to our next proposition (\ref{gMRapprox})  we require three technical lemmas.  Recall from the introduction (\ref{add approx def}) the notion of an approximation.
\begin{lemma}\label{cokerfl}
Let $R$ be a normal 3-sCY ring and let $\Lambda$ be a module finite $R$-algebra which is 3-sCY.  Let $B\in\refl \Lambda$ be a modifying height one progenerator and let $C\in\refl \Lambda$.  If $0\to  A\stackrel{f}{\to} B_0 \stackrel{g}{\to} C \to 0$ is an exact sequence where $g$ is a right $(\add B)$-approximation, then the cokernel of the natural map
\[
\Hom_{\Lambda}(B_0,B) \stackrel{f\cdot}\to \Hom_{\Lambda}(A,B)
\]
has finite length.
\end{lemma}
\begin{proof}
Set $\Gamma:=\End_\Lambda(B)$. Since $B$ is a height one progenerator
we have a reflexive equivalence
$\mathbb{F}:=\Hom_\Lambda(B,-):\refl\Lambda\to \refl\Gamma$ by \ref{reflequiv}(4).  Moreover $\Gamma$ is 3-sCY by \ref{3CY-nonlocal}.

Since $g$ is a right $(\add B)$-approximation, we have an exact sequence
\[
0\to\mathbb{F}A \to \mathbb{F}B_0 \to \mathbb{F}C \to 0
\]
of $\Gamma$-modules.  Then since $\mathbb{F}B_0\in \proj\Gamma$, applying $\Hom_\Gamma(-,\Gamma)=\Hom_\Gamma(-,\mathbb{F}B)$ gives an exact sequence
\[
\SelectTips{cm}{10}
\xymatrix@R=10pt@C=10pt{
\Hom_\Gamma(\mathbb{F}B_0,\mathbb{F}B)\ar[r]\ar@{<-}[d]_{\cong}&\Hom_\Gamma(\mathbb{F}A,\mathbb{F}B)\ar[r]\ar@{<-}[d]_{\cong}& \Ext^1_\Gamma(\mathbb{F}C,\Gamma)\ar[r]& 0\\
\Hom_{\Lambda}({B_0},B)\ar[r]^{f\cdot}&\Hom_{\Lambda}(A,B)&&
}
\]
and thus $\Cok(f\cdot)=\Ext^1_\Gamma(\mathbb{F}C,\Gamma)$.  Hence we only have to show that $\Ext^1_\Gamma(\mathbb{F}C,\Gamma)_\p=0$ for any non-maximal prime
ideal $\p$ of $R$. By \ref{depthofhom} and \ref{CMinheight2} we have $(\mathbb{F}C)_\p\in \CM\Gamma_\p$.
Since $\Gamma$ is 3-sCY, $\Gamma_\p$ is a Gorenstein $R_\p$-order by \ref{2.14a}. Consequently $\Ext^1_\Gamma(\mathbb{F}C,\Gamma)_\p=\Ext^1_{\Gamma_\p}((\mathbb{F}C)_\p,\Gamma_\p)=0$, as required.
\end{proof}

\begin{lemma}\label{technical2}
Let $R$ be a normal 3-sCY ring and let $\Lambda$ be a module finite $R$-algebra which is 3-sCY.  Suppose $N\in\refl\Lambda$ and $M\in\CM\Lambda$ with both $M$ and $N$ modifying such that $M$ is a height one progenerator.   If $0\to L\to M_0\stackrel{h}{\to} N\to 0$ is an exact sequence where $h$ is a right $(\add M)$-approximation, then $\End_{\Lambda}(L\oplus M)\in\CM R$.
\end{lemma}
\begin{proof}
Note first that since $N$ is reflexive and $M\in\CM \Lambda$ we have $L\in\CM \Lambda$ by the depth lemma. From the exact sequence
\[
0\to \Hom_{\Lambda}(M,L)\to \Hom_{\Lambda}(M,M_0)\to \Hom_{\Lambda}(M,N)\to 0
\]
with $\Hom_{\Lambda}(M,M_{0})\in\CM R$ we see, using \ref{depthofhom} and the depth lemma, that $\Hom_{\Lambda}(M,L)\in\CM R$.  By \ref{HomsymmDirect} $\Hom_{\Lambda}(L,M)\in\CM R$. Since $\End_{\Lambda}(M)\in\CM R$ by assumption, it suffices to show that $\End_{\Lambda}(L)\in\CM R$.  By \ref{reflandCM} we only need to show that $\fl\Ext^1_{\Lambda}(L,L)=0$.

Consider now the following exact commutative diagram
\[
\SelectTips{cm}{10}
\xy0;/r.5pc/:
(15,0)*+{\Hom_{\Lambda}(L,M_0)}="1",(30,0)*+{\Hom_{\Lambda}(L,N)}="2",(45,0)*+{\Ext_{\Lambda}^1(L,L)}="3",(60,0)*+{\Ext_{\Lambda}^1(L,M_0)}="4",
(15,-6.5)*+{\Hom_{\Lambda}(M_0,M_0)}="b1",(30,-6.5)*+{\Hom_{\Lambda}(M_0,N)}="b2",
(37.5,4)*+{\Cok f}="c1",(52.5,-4)*+{K}="c2"
\ar^{f}"1";"2"
\ar"2";"3"
\ar"3";"4"
\ar@{->>}^t"b1";"b2"
\ar^b"b1";"1"
\ar^c"b2";"2"
\ar@{->>}"2";"c1"
\ar@{^{(}->}"c1";"3"
\ar@{->>}"3";"c2"
\ar@{^{(}->}"c2";"4"
\endxy .
\]
Since $\Hom_{\Lambda}(L,M)\in\CM R$ we know by  \ref{reflandCM} that $\fl \Ext^1_{\Lambda}(L,M_0)=0$ and so $\fl K=0$. Hence to show that $\fl\Ext^1_{\Lambda}(L,L)=0$ we just need to show that $\fl\Cok f=0$.  To do this consider the exact sequence
\begin{eqnarray}
\Cok b\to\Cok bf\to \Cok f\to 0. \label{cok}
\end{eqnarray}
By \ref{cokerfl} applied with $B=M_0$, $\Cok b$ has finite length and thus the image of the first map in (\ref{cok}) has finite length.  Second, note that $\Cok bf=\Cok tc=\Cok c$ and $\fl\Cok c=0$ since $\Cok c$ embeds inside $\Ext^1_{\Lambda}(N,N)$ and furthermore $\fl\Ext^1_{\Lambda}(N,N)=0$ by \ref{reflandCM}.  This means that the image of the first map is zero, hence $\Cok f\cong \Cok c$ and so in particular $\fl\Cok f=0$.
\end{proof}

In fact, using reflexive equivalence we have the following improvement of \ref{technical2} which does not assume that $M$ is CM, which is the analogue of \cite[5.1]{GLS}.
\begin{lemma}\label{4.5refl}
Let $R$ be a normal 3-sCY ring and let $M$ and $N$ be modifying modules.
If $0\to L\to M_0\stackrel{h}{\to} N$ is an exact sequence where $h$ is a
right $(\add M)$-approximation, then $L\oplus M$ is modifying.
\end{lemma}
\begin{proof}
Note that $L$ is reflexive since $R$ is normal.  Denote $\Lambda:=\End_{R}(M)$ and $\mathbb{F}:=\Hom_{R}(M,-):\refl R\to \refl \Lambda$ the reflexive equivalence in \ref{reflequiv}(4). Then $\Lambda$ is 3-sCY by \ref{3CY-nonlocal}, $\mathbb{F}N\in\refl \Lambda$, $\mathbb{F}M\in\CM \Lambda$ and both $\mathbb{F}N$ and $\mathbb{F}M$ are modifying $\Lambda$-modules.  Further
\[
0\to\mathbb{F}L\to\mathbb{F}M_{0}\xrightarrow{\mathbb{F}h}\mathbb{F}N\to 0
\]
is exact and $\mathbb{F}M=\Lambda$ so trivially $\mathbb{F}h$ is a right $(\add\mathbb{F}M)$-approximation.  It is also clear that $\mathbb{F}M=\Lambda$  is a height one progenerator.
By \ref{technical2} we see that $\End_{\Lambda}(\mathbb{F}L\oplus\mathbb{F}M)\in\CM R$, hence  $\End_{R}(L\oplus M)\cong\End_{\Lambda}(\mathbb{F}L\oplus\mathbb{F}M)\in\CM R$ as required.   
\end{proof}

Now we are ready to prove the following crucial result (c.f.\ \ref{CT2} later), which is the analogue of \cite[5.2]{GLS}.
\begin{thm}\label{gMRapprox}
Let $R$ be a normal 3-sCY ring and let $M$ be a non-zero modifying module. Then the following are equivalent\\
\t{(1)} $M$ is an MM module.\\
\t{(2)} If $N$ is any modifying module then there exists an exact sequence $0\to M_1\to M_0\stackrel{f}{\to} N$ with each $M_i\in\add M$ such that $f$ is a right $(\add M)$-approximation.
\end{thm}
\begin{proof}
Set $\Lambda:=\End_R(M)$.  Since $M$ is a height one progenerator, we have a reflexive equivalence $\mathbb{F}:=\Hom_R(M,-):\refl R\to \refl\Lambda$ by \ref{reflequiv}(4). Moreover $\Lambda$ is 3-sCY by \ref{3CY-nonlocal} and so a Gorenstein $R$-order by \ref{2.14a}.\\
(1)$\Rightarrow$(2) We have an exact sequence $0\to L\to M_0\stackrel{f}{\to}N$ where $f$ is a right $(\add M)$-approximation of $N$.  By \ref{4.5refl} $\End_{R}(L\oplus M)\in\CM R$ thus since $M$ is an MM module, $L\in\add M$.\\
(2)$\Rightarrow$(1) Suppose $N$ is reflexive with $\End_{R}(M\oplus N)\in\CM R$. Then $\mathbb{F}N\in \CM R$. We have $\pd_\Lambda \mathbb{F}N\leq 1$ since $N$ is a modifying module and so there is an exact sequence $0\to \mathbb{F}M_1\to \mathbb{F}M_0\to \mathbb{F}N\to0$ by assumption. Since $\Lambda$ is a Gorenstein $R$-order it follows that $\mathbb{F}N$ is a projective $\Lambda$-module by using localization and Auslander--Buchsbaum \ref{ABlocal}. Hence $N\in \add M$.
\end{proof}

The following version of the Bongartz completion \cite{Bongartz}\cite[VI.2.4]{ASS} is convenient for us.  Recall from the introduction that throughout this paper when we say tilting module we mean a tilting module of projective dimension $\leq 1$ (see \ref{tiltingdef}).
\begin{lemma}\label{Bongartz}
Suppose $R$ is normal, $M\in\refl R$ and denote $\Lambda:=\End_R(M)$.  If $N\in\refl R$ is such that $\Hom_R(M,N)$ is a partial tilting $\Lambda$-module then there exists $L\in\refl R$ such that $\Hom_R(M,N\oplus L)$ is a tilting $\Lambda$-module. 
\end{lemma}
\begin{proof}
By \ref{reflequiv} $T:=\Hom_R(M,N)$ and $\Lambda$ are both reflexive.  Thus since $R$ is normal we can invoke \cite[2.8]{IR} to deduce that there exists an $X\in\refl \Lambda$ such that $T\oplus X$ is tilting.  Again by \ref{reflequiv} $X=\Hom_R(M,L)$ for some $L\in\refl R$.  
\end{proof}

We have the following analogue of \cite[8.7]{IR}.
\begin{prop}\label{Homtilting}
Let $R$ be a normal 3-sCY ring and assume $M$ is an MM module. Then \\
\t{(1)} $\Hom_R(M,-)$ sends modifying $R$-modules to partial tilting $\End_R(M)$-modules.\\
\t{(2)} $\Hom_R(M,-)$ sends MM $R$-modules to tilting $\End_R(M)$-modules.
\end{prop}
\begin{proof}
(1) Denote $\Lambda:=\End_R(M)$, let $N$ be a modifying module and denote $T:=\Hom_R(M,N)$. Note first that $\pd_{\Lambda}T\leq 1$ by \ref{gMRapprox} and also $\Lambda$ is a Gorenstein $R$-order by \ref{2.14a} and \ref{3CY-nonlocal}.

Since projective dimension localizes $\pd_{\Lambda_{\p}}T_{\p}\leq 1$ for all primes $\p$, and further if $\hgt\p=2$ then $T_{\p}\in\CM R_{\p}$ by \ref{depthofhom}. Since $\Lambda_{\p}$ is a Gorenstein $R_{\p}$-order, Auslander--Buchsbaum (\ref{ABlocal}) implies that $T_{\p}$ is a projective $\Lambda_{\p}$-module and so $\Ext_{\Lambda_{\p}}^{1}(T_{\p},T_{\p})=0$ for all primes $\p$ with $\hgt\p=2$.  Consequently $\Ext_{\Lambda_{\m}}^{1}(T_{\m},T_{\m})$ has finite length for all $\m\in\Max R$.  On the other hand $\Lambda$ is 3-sCY by \ref{3CY-nonlocal} and $\End_{\Lambda}(T)\cong\End_{R}(N)\in\CM R$ by \ref{reflequiv}.  Thus $\fl \Ext_{\Lambda_{\m}}^{1}(T_{\m},T_{\m})=0$ for all $\m\in\Max R$ by \ref{reflandCM} and so $\Ext^{1}_{\Lambda}(T,T)=0$, as required.\\
(2) Now suppose that $N$ is also MM.  By Bongartz completion \ref{Bongartz} we may find $L\in\refl R$ such that $\Hom_{R}(M,N\oplus L)$ is a tilting $\End_R(M)$-module, thus $\End_{R}(M)$ and $\End_{R}(N\oplus L)$ are derived equivalent.  Since $\End_R(M)$ is 3-sCY so is $\End_{R}(N\oplus L)$ by \ref{IRDB} and thus by \ref{3CY-nonlocal} $\End_R(N\oplus L)\in\CM R$.  Consequently $L\in\add N$ and so $\Hom_R(M,N)$ is a tilting module. 

Now for the convenience of the reader we give a second proof, which shows us more explicitly how our tilting module generates the derived category.  If $N$ is also MM then since $(-)^{*}:\refl R\to \refl R$ is a duality, certainly $N^{*}$ (and $M^{*}$) is MM.  By \ref{gMRapprox} we can find
\[
0\to N_{1}^{*}\to N_{0}^{*}\to M^{*}
\]
such that
\begin{eqnarray}
0\to\mathbb{F}N_{1}^{*}\to \mathbb{F}N_{0}^{*}\to\mathbb{F}M^{*}\to 0\label{rem1}
\end{eqnarray}
is exact, where $\mathbb{F}=\Hom_{R}(N^{*},-)$.  Denote $\Gamma:=\End_{R}(N^{*})$, then $\Ext^{1}_{\Gamma}(\mathbb{F}M^{*},\mathbb{F}M^{*})=0$ by (1).  Thus applying $\Hom_{\Gamma}(-,\mathbb{F}M^{*})$ to (\ref{rem1}) gives us the following commutative diagram
\[
{\SelectTips{cm}{10}
\xy
(0,0)*+{0}="0",
(20,0)*+{\Hom_{\Gamma}(\mathbb{F}M^{*},\mathbb{F}M^{*})}="1",
(55,0)*+{\Hom_{\Gamma}(\mathbb{F}N_{0}^{*},\mathbb{F}M^{*})}="2",
(90,0)*+{\Hom_{\Gamma}(\mathbb{F}N_{1}^{*},\mathbb{F}M^{*})}="3",
(110,0)*+{0}="4",
(0,-10)*+{0}="0a",
(20,-10)*+{\Hom_{R}(M^{*},M^{*})}="1a",
(55,-10)*+{\Hom_{R}(N_{0}^{*},M^{*})}="2a",
(90,-10)*+{\Hom_{R}(N_{1}^{*},M^{*})}="3a",
(110,-10)*+{0}="4a",
\ar"0";"1"
\ar"1";"2"
\ar"2";"3"
\ar"3";"4"
\ar"0a";"1a"
\ar"1a";"2a"
\ar"2a";"3a"
\ar@{=}"1";"1a"
\ar@{=}"2";"2a"
\ar@{=}"3";"3a"
\ar"3a";"4a"
\endxy}
\]
where the top row is exact and the vertical maps are isomorphisms by \ref{reflequiv}.  Hence the bottom row is exact.  Since $(-)^{*}:\refl R\to \refl R$ is a duality, this means that
\[
0\to\Hom_{R}(M,M)\to\Hom_{R}(M,N_{0})\to\Hom_{R}(M,N_{1})\to 0
\]
is exact.  But denoting $\Lambda:=\End_{R}(M)$ and $T:=\Hom_{R}(M,N)$, this means we have an exact sequence
\[
0\to\Lambda\to T_{0}\to T_{1}\to 0
\]
with each $T_{i}\in \add T$.  Hence $T$ is a tilting $\Lambda$-module.
\end{proof}

The following is now immediate:
\begin{cor}\label{allgMRderived} 
Let $R$ be a normal 3-sCY ring and assume $M$ is an MM module.  Then \\
\t{(1)} If $N$ is any modifying module then the partial tilting $\End_{R}(M)$-module $T:=\Hom_R(M,N)$ induces a recollement
\[
{\SelectTips{cm}{10}
\xy
(-23,0)*+{\rm K}="0",
(0,0)*+{\D(\Mod\End_R(M))}="1",
(36,0)*+{\D(\Mod\End_R(N))}="2",
\ar"0";"1",
\ar@<-1.5ex>"1";"0",
\ar@<1.5ex>"1";"0",
\ar|{F}"1";"2",
\ar@<-1.5ex>"2";"1",
\ar@<1.5ex>"2";"1",
\endxy}
\]
where $F=\mathbf{R}{\rm Hom}(T,-)$ and ${\rm K}$ is a certain triangulated subcategory of $\D(\Mod\End_R(M))$.\\
\t{(2)} If further $N$ is maximal modifying then the above functor $F$ is an equivalence.
\end{cor}
\begin{proof}
(1) Set $\Lambda:=\End_R(M)$ then $T:=\Hom_R(M,N)$ is a partial tilting $\Lambda$-module by \ref{Homtilting}.  The fact that $\End_{\Lambda}(T)\cong\End_R(N)$ follows since $\Hom_R(M,-)$ is a reflexive equivalence by \ref{reflequiv}(4).  By Bongartz completion $T$ is a summand of a tilting $\Lambda$-module $U$. We have a derived equivalence $\D(\Mod \End_R(M))\simeq \D(\Mod \End_\Lambda(U))$. Moreover there exists an idempotent $e$ of $\End_\Lambda(U)$ such that
$e\End_\Lambda(U)e\cong \End_\Lambda(T)\cong \End_R(N)$.
Thus we have the desired recollement (e.g. \cite[4.16]{Krause}, see also \cite{M03}).\\
(2) is an immediate consequence by taking $U:=T$ in the argument above.
\end{proof}

We can now improve \ref{MMAs dim3 closed}.

\begin{thm}\label{closed in dimension three}
Let $R$ be a normal $3$-sCY ring.  Then MMAs of $R$ form a complete derived equivalence class.
\end{thm}
\begin{proof}
By \ref{MMAs dim3 closed}(2), MMAs of $R$ are closed under derived equivalence.  On the other hand, all MMAs are derived equivalent by \ref{allgMRderived}(2).
\end{proof}

Moreover, we have the following bijections (cf.\ \cite[8.9]{IR}).

\begin{thm}\label{gRgMRrigid}
Let $R$ be a normal 3-sCY ring and assume $M$ is an MM module. Then the functor $\Hom_R(M,-):\mod R\to\mod\End_{R}(M)$ induces bijections
\[
\begin{array}{rrcl}
\t{(1)} &\{\mbox{maximal modifying $R$-modules}\}&\stackrel{1:1}\longleftrightarrow&\{\mbox{reflexive tilting $\End_R(M)$-modules}\}.\\
\t{(2)} & \{\mbox{modifying $R$-modules}\}&\stackrel{1:1}\longleftrightarrow&\{\mbox{reflexive partial tilting $\End_R(M)$-modules}\}.
\end{array}
\]
\end{thm}
\begin{proof}
(1) In light of \ref{Homtilting}(2) we only need to show that every reflexive tilting module is the image of some MM module.  Thus let $X$ be a reflexive tilting $\End_R(M)$-module; by reflexive equivalence \ref{reflequiv}(4) there exists some $N\in\refl R$ such that $\Hom_R(M,N)\cong X$.  We claim that $N$ is MM.
Since $\Hom_R(M,N)$ is a tilting $\End_R(M)$-module certainly $\End_R(M)$ and $\End_R(N)$ are derived equivalent; the fact that $N$ is MM follows from \ref{MMAs dim3 closed}(2) above.\\
(2) By \ref{Homtilting}(1) we only need to show that every reflexive partial tilting $\End_R(M)$-module is the image of some modifying module.  Suppose $X$ is a reflexive partial tilting $\End_R(M)$-module, say $X\cong\Hom_R(M,N)$.  Then by Bongartz completion \ref{Bongartz} there exists $N_1\in\refl R$ such that $\Hom_R(M,N\oplus N_1)$ is a tilting $\End_R(M)$-module.  By (1) $N\oplus N_1$ is MM, thus $\End_R(N)$ is a summand of the CM $R$-module $\End_R(N\oplus N_1)$ and so is itself CM. 
\end{proof}

\begin{cor}\label{gRgMRrigid2}
Let $R$ be a normal 3-sCY ring and assume $M$ is an MM module.  Then\\
\t{(1)} $N$ is a modifying module $\iff$ $N$ is the direct summand of an MM module.\\
\t{(2)} $R$ has an MM module which is a CM generator.
\end{cor}
\begin{proof}
(1) `if' is clear.  For the `only if' let $N$ be a modifying module, then by \ref{Homtilting}(1) $\Hom_R(M,N)$ is a partial tilting $\End_R(M)$-module, so the proof of \ref{gRgMRrigid}(2) shows that there exists $N_1\in\refl R$ such that $N\oplus N_1$ is MM.\\
(2) Apply (1) to $R$.  The corresponding MM module is necessarily CM by \ref{CMiffgen}.
\end{proof}

Recall from \ref{nonsingorder}(3) we say that an $R$-order $\Lambda$ has \emph{isolated singularities} if $\gl \Lambda_\p=\dim R_{\p}$ for all non-maximal primes $\p$ of $R$.

\begin{remark}
It is unclear in what generality every maximal modification algebra $\End_R(M)$ has isolated singularities.  In many cases this is true --- for example if $R$ is itself an isolated singularity this holds, as it does whenever $M$ is CT by \ref{CT}.  Also, if $R$ is Gorenstein, $X\to\Spec R$ projective birational with $M\in\refl R$ modifying such that $\Db(\coh X)\cong\Db(\mod\End_R(M))$, then provided $X$ has at worst isolated singularities (e.g. if $X$ is a 3-fold with terminal singularities) then $\End_R(M)$ has isolated singularities too.  This is a direct consequence of the fact that in this case the singular derived category has finite dimensional Hom-spaces.  Also note that if $R$ is normal 3-sCY then the existence of an MM algebra $\End_R(M)$ with isolated singularities implies that $R_\p$ has finite CM type for all primes $\p$ of height 2 by a result of Auslander (see \cite[2.13]{IW}).  Finally note that it follows immediately from \ref{allgMRderived}(2) (after combining \ref{IRDB} and \ref{derivedlocalcomp}) that if $R$ is  normal 3-sCY and there is one MM algebra $\End_R(M)$ with isolated singularities then necessary all MM algebras $\End_R(N)$ have isolated singularities.
\end{remark}
The above remark suggests the following conjecture.
\begin{conj}
Let $R$ be a normal 3-sCY ring with rational singularities.  Then\\
(1)  $R$ always has an MM module $M$ (which may be $R$).\\
(2)  For all such $M$, $\End_R(M)$ has isolated singularities. 
\end{conj}
This is closely related to a conjecture of Van den Bergh regarding the equivalence of the existence of crepant and noncommutative crepant resolutions when $R$ is a rational normal Gorenstein 3-fold \cite[4.6]{VdBNCCR}.  We remark that given the assumption on rational singularities, any proof is likely to be geometric.  We also remark that the restriction to rational singularities is strictly necessary, since we can consider any normal surface singularity $S$ of infinite CM type.  Since $S$ is a surface $\End_S(M)\in\CM S$ for all $M\in\CM S$, so since $S$ has infinite CM type it cannot admit an MMA.  Now consider $R:=S\otimes_{\mathbb{C}}\C{}[t]$, then $R$ is a 3-fold that does not admit an MMA.  A concrete example is given by $R=\mathbb{C}[x,y,z,t]/x^3+y^3+z^3$.

\section{Relationship Between CT modules, NCCRs and MM modules}

In this section we define CT modules for singularities that are not necessarily isolated, and we show that they are a special case of the MM modules introduced in \S\ref{gMR}.  We also show (in \ref{rigidisolated}) that all these notions recover the established ones when $R$ is an isolated singularity.

When $R$ is a normal 3-sCY domain, below we prove the implications in the following picture which summarizes the relationship between CT modules, NCCRs and MM modules:
\[
{
\xy
(0,0)*+{\mbox{{\rm CT modules}}}="1",
(37,0)*+{\mbox{{\rm modules giving NCCRs}}}="2",
(74,0)*+{\mbox{{\rm MM modules}}}="3",
(105,0)*+{\mbox{{\rm modifying modules}}}="4",
\ar@<1.5ex>@{=>}^(0.39){\scriptsize\ref{CT}}"1";"2"
\ar@<-1.5ex>@{<=}_(0.4){\scriptsize \begin{array}{c}\mbox{if generator} \\ \scriptsize\ref{CT}\end{array}}"1";"2"
\ar@<1.5ex>@{=>}^(0.6){\scriptsize\ref{NCCRgiveMM}}"2";"3"
\ar@<-1.5ex>@{<=}_(0.6){\scriptsize\begin{array}{c} \mbox{if } \exists \mbox{NCCR}\\ \ref{NCCRCTisgMR} \end{array}}"2";"3"
\ar@{=>}"3";"4"
\endxy}
\]
We remark that the relationship given by \ref{NCCRgiveMM} and \ref{CT} holds in arbitrary dimension $d$, whereas \ref{NCCRCTisgMR} requires $d=3$.
 
\begin{defin}\label{CTdefin}
Let $R$ be a $d$-dimensional CM ring with a canonical module $\omega_{R}$.  We call $M\in\CM R$ a CT module if
\[
\add M=\{ X\in\CM R : \Hom_R(M,X)\in\CM R \}=\{ X\in\CM R : \Hom_R(X,M)\in\CM R\}.
\]
\end{defin}
The name CT is inspired by (but is normally different than) the notion of `cluster tilting' modules.  We explain the connection in \ref{rigidisolated}. 

\begin{lemma}\t{(1)} Any CT module is a generator-cogenerator.\\
\t{(2)} Any CT module is maximal modifying.
\end{lemma}
\begin{proof}
Let $M$ be a CT module.\\
(1) This is clear from $\Hom_R(M,\omega_{R})\in\CM R$ and  $\Hom_R(R,M)\in \CM R$.\\
(2) Suppose $N\in\refl R$ with $\End_R(M\oplus N)\in \CM R$, then certainly $\Hom_R(M,N)\in \CM R$. Since $R\in \add M$ by (1), we have $N\in \CM R$.
Hence since $M$ is a CT module, necessarily $N\in \add M$.
\end{proof}
Not every MM module is CT, however in the situation when $R$ has a CT module (equivalently, by  \ref{NCCRhasCM}(2) below, $R$ has a NCCR) we give a rather remarkable relationship between CT modules, MM modules and NCCRs in \ref{NCCRCTisgMR} at the end of this subsection.

If $R$ is a CM ring with a canonical module $\omega_{R}$ we denote the duality $(-)^\vee:=\Hom_R(-,\omega_{R})$. We shall see shortly that if $M$ or $M^\vee$ is a generator then we may test the above CT condition on one side (see \ref{CT} below), but before we do this we need the following easy observation.

\begin{lemma}\label{Romega}
Let $R$ be a CM ring with a canonical module $\omega_{R}$, $M\in\CM R$.  If $\End_R(M)$ is a non-singular $R$-order, then $R\in\add M\iff \omega_{R}\in\add M$.
\end{lemma}
\begin{proof}
Since $(-)^\vee :\CM R\to \CM R$ is a duality we know that $\End_R(M^\vee)\cong \End_R(M)^{\op}$ and so $\End_R(M^\vee)$ is also a non-singular $R$-order.  Moreover $R^\vee=\omega_{R}$ and $\omega_{R}^\vee=R$ so by the symmetry of this situation we need only prove the `only if' part.  Thus assume that $R\in\add M$.  In this case since $\Hom_R(M,\omega_{R})=M^\vee\in\CM R$, by \ref{nonsingGoren} $\Hom_R(M,\omega_{R})$ is a projective $\End_R(M)$-module and thus $\omega_{R}\in \add M$ by \ref{reflequiv}(1).
\end{proof}

We reach one of our main characterizations of CT modules.  Note that if $R$ is a normal $d$-sCY ring, then by \ref{HomsymmDirect} the definition of CT modules is equivalent to 
\[
\add M=\{ X\in\CM R\mid\Hom_R(M,X)\in\CM R  \},
\]
however the following argument works in greater generality:

\begin{thm}\label{CT}
Let $R$ be a $d$-dimensional, equi-codimensional CM ring (e.g.\ if $R$ is $d$-sCY) with a canonical module $\omega_{R}$.  Then for any $M\in \CM R$ the following are equivalent\\
$\t{(1)}$ $M$ is a CT module.\\
$\t{(2)}$  $R\in\add M$ and $\add M=\{ X\in\CM R : \Hom_R(M,X)\in\CM R \}$.\\
$\t{(2)}^\prime$  $\omega_{R}\in\add M$ and $\add M=\{ X\in\CM R : \Hom_R(X,M)\in\CM R \}$.\\
$\t{(3)}$ $R\in\add M$ and $\End_R(M)$ is a non-singular $R$-order.\\
$\t{(3)}^\prime$ $\omega_{R}\in\add M$ and $\End_R(M)$ is a non-singular $R$-order.\\
In particular CT modules are precisely the CM generators which give NCCRs.
\end{thm}
\begin{proof}
(2)$\Rightarrow$(3)  By assumption we have $R\in\add M$ and $\End_R(M)\in\CM R$.   Now let $Y\in\mod\End_R(M)$ and consider a projective resolution $P_{d-1}\to P_{d-2}\to...\to P_0\to Y\to0$.
By \ref{reflequiv}(1) there is an exact sequence
$M_{d-1}\stackrel{f}{\to}M_{d-2}\to...\to M_0$ with each $M_i\in \add M$
such that the projective resolution above is precisely
\[
\Hom_R(M,M_{d-1})\stackrel{\cdot f}{\to}\Hom_R(M,M_{d-2})\to...\to \Hom_R(M,M_0)\to Y\to0.
\]
Denote $K_d=\Ker f$. Then we have an exact sequence
\[ 
0\to \Hom_R(M,K_d)\to P_{d-1}\to P_{d-2}\to...\to P_0\to Y\to 0.
\]
Localizing the above and counting depths we see that $\Hom_R(M,K_d)_\m\in \CM R_\m$ for all $\m\in \Max R$, thus $\Hom_R(M,K_d)\in \CM R$ and so by definition $K_d\in \add M$. Hence $\pd_{\End_R(M)} Y\leq d$ and so $\gl \End_R(M)\leq d$. \\
(3)$\Rightarrow$(2)  Since $\End_R(M)\in\CM R$, automatically $\add M\subseteq \{ X\in\CM R : \Hom_R(M,X)\in\CM R \}$.  To obtain the reverse inclusion assume that $X\in \CM R$ with $ \Hom_R(M,X)\in\CM R$, then since $\End_R(M)$ is a non-singular $R$-order $\Hom_R(M,X)$ is a projective $\End_R(M)$-module by \ref{nonsingGoren}.  This implies that $X\in\add M$ by \ref{reflequiv}(1).\\
$\t{(2)}^{\prime}\iff\t{(3)}^{\prime}$ We have a duality $(-)^\vee:\CM R\to\CM R$ thus apply (2)$\iff$(3) to $M^\vee$ and use the fact that $\End_R(M^\vee)=\End_{R}(M)^{\rm op}$ has finite global dimension if and only if  $\End_R(M)$ does.\\
$\t{(3)}\iff\t{(3)}^{\prime}$ This is immediate from \ref{Romega}.\\
In particular by the above we have (2)$\iff$(2)$^{\prime}$.  Since we clearly have (1)$\iff$(2)$+$(2)$^\prime$, the proof is completed.
\end{proof}
Note that the last assertion in \ref{CT} is improved when $R$ is a 3-sCY normal domain in \ref{NCCRCTisgMR}(3). From the definition it is not entirely clear that CT is a local property:

\begin{cor}\label{CTislocal}
Let $R$ be a $d$-dimensional, equi-codimensional CM ring (e.g.\ if $R$ is $d$-sCY) with a canonical module $\omega_{R}$.  Then the following are equivalent\\
\t{(1)} $M$ is a CT $R$-module\\
\t{(2)} $M_\p$ is a CT $R_\p$-module for $\p\in\Spec R$.\\
\t{(3)} $M_\m$ is a CT $R_\m$-module for $\m\in\Max R$.\\
\t{(4)} $\widehat{M_\p}$ is a CT $\widehat{R_\p}$-module for $\p\in\Spec R$.\\
\t{(5)} $\widehat{M_\m}$ is a CT $\widehat{R_\m}$-module for $\m\in\Max R$.\\
Thus CT can be checked locally, or even complete locally.
\end{cor}
\begin{proof}
By \ref{CT}(3) $M$ is a CT $R$-module if and only if $R\in \add M$ and $\End_R(M)$ is a non-singular $R$-order.  Non-singular $R$-orders can be checked either locally or complete locally (\ref{nonsingGoren}), and $R\in\add M$ can be checked locally or complete locally by  \ref{addlocal}.
\end{proof}

Theorem~\ref{CT} also gives an easy method to find examples of CT modules. Recall that an element $g\in\GL(d,k)$ is called \emph{pseudo-reflection} if the rank of $g-1$ is at most one. A finite subgroup $G$ of $\GL(d,k)$ is called \emph{small} if it does not contain pseudo-reflections except the identity.   The following is well-known:

\begin{prop}\label{reflection}
Let $k$ be a field of characteristic zero, let $S$ be the polynomial ring $k[x_1,\hdots,x_d]$ (respectively formal power series ring $k[[x_1,\hdots,x_d]]$) and let $G$ be a finite subgroup of $\GL(d,k)$.\\
(1) If $G$ is generated by pseudo-reflections, then $S^G$ is a polynomial ring (respectively a formal power series ring) in $d$ variables.\\
(2) If $G$ is small, then the natural map $S\# G\to\End_R(S)$ given by $sg\mapsto(t\mapsto s\cdot g(t))$ ($s,t\in S,\ g\in G$) is an isomorphism.
\end{prop}
\begin{proof}
(1) See \cite[\S5 no.~5]{Bou68} for example.\\
(2) This is due to Auslander \cite[\S4]{Auslander_rational}, \cite[10.8]{Y}.  See also \cite[3.2]{IT} for a detailed proof.
\end{proof}

This immediately gives us a rich source of CT modules.  The following result is shown in \cite[2.5]{IyamaAR} under the assumption that $G$ is a small subgroup of $\GL(d,k)$ and $S^G$ is an isolated singularity. We can drop both assumptions under our definition of CT modules.
\begin{thm}\label{skewgp}
Let $k$ be a field of characteristic zero, and
let $S$ be the polynomial ring $k[x_1,\hdots,x_d]$ (respectively formal power series ring $k[[x_1,\hdots,x_d]]$). For
a finite subgroup $G$ of $\GL(d,k)$, let $R=S^G$.  Then $S$ is a CT $R$-module.
\end{thm}
\begin{proof}
We prove the case when $S$ is a polynomial ring, since the case when $S$ is a formal power series ring then follows by \ref{CTislocal}.  We proceed by induction on $|G|$, the case $|G|=1$ being trivial.  If $G$ is small, then  $\End_{R}(S)$ is isomorphic to $S\#G$  by \ref{reflection}(2).  This shows, by \ref{skew is non-singular}, that $\End_{R}(S)$ is a non-singular $R$-order and so $S$ is a CT $R$-module by \ref{CT}.

Hence we can assume that $G$ is not small, so if $N$ denotes the subgroup of $G$ generated by pseudo-reflections, we have $|G/N|<|G|$.  Now $G/N$ acts on $S^N$, which by \ref{reflection}(1) is a polynomial ring.  In fact the graded subring $S^N$ of $S$ has a free generating set of homogeneous polynomials \cite[Thm.\ A]{Chev}.  Let $V(d)$ be the vector space of $N$-invariant polynomials of degree $d$ (with respect to the original grading of $S$).  We prove, by induction on $d$, that generators of $S^N$ can be chosen so that $G/N$ acts linearly on these generators.

Clearly the action of $G/N$ is linear on $U(d_1):=V(d_1)$, where $d_1>0$ is the smallest such that $V(d_1)$ is non-empty.  Consider now $V(d)$.  It has a subspace, say $U(d)$, of linear combinations of products of $N$-invariant polynomials of smaller degree.  By induction, $G/N$ acts on $U(d)$, so $U(d)$ is a $G/N$-submodule.  By Maschke's theorem we can take the $G/N$-complement to $U(d)$ in $V(d)$.  Then $G/N$ acts linearly on this piece, so it acts linearly on $V(d)$.  Hence a $k$-basis of $U(d)$ for each $d>0$ gives a free generating set on which $G/N$ acts linearly.

Hence, with these new generators, we have $S^{G}=(S^N)^{G/N}\cong k[X_1,\hdots,X_d]^{G/N}$ where $G/N$ a subgroup of $\GL(d,k)$.
 Thus 
\[
\CM S^{G}\simeq \CM k[X_1,\hdots,X_d]^{G/N}
\]
and further under this correspondence 
\[
\add_{S^{G}} S\simeq\add_{k[X_1,\hdots,X_d]^{G/N}}S^N=\add_{k[X_1,\hdots,X_d]^{G/N}} k[X_1,\hdots,X_d].
\]
Hence $\End_{S^{G}}(S)$ is Morita equivalent to $\End_{k[X_1,\hdots,X_d]^{G/N}}(k[X_1,\hdots,X_d]):=\Lambda$.  By induction $\Lambda$ is a non-singular $R$-order, so it follows that $\End_{S^G}(S)$ is a non-singular $R$-order by \ref{nonsingMorita}.  Consequently $S$ is a CT $S^G$-module by \ref{CT} since $R$ is a direct summand of the $R$-module $S$ by the Reynolds operator.
\end{proof}

As another source of CT modules, we have:

\begin{example}\label{1dfibre}
Let $Y\stackrel{f}{\to} X=\Spec R$ be a projective birational morphism such that $Rf_*\s{O}_Y=\s{O}_X $ and every fibre has dimension at most one, where $R$ is a $d$-dimensional normal Gorenstein ring $R$ finitely generated over a field.  Then provided $Y$ is smooth and crepant there exists  a NCCR $\End_R(M)$ \cite[3.2.10]{VdB1d} in which $M$ is CM containing $R$ as a summand.  By \ref{CT} $M$ is a CT module. 
\end{example}
We now show that for $R$ normal 3-sCY, the existence of a CT module is equivalent to the existence of a NCCR.  Note that (2) below answers a question of Van den Bergh \cite[4.4]{VdBNCCR}.
\begin{cor}\label{NCCRhasCM}
Let $R$ be a 3-sCY normal domain.   Then \\
\t{(1)} CT modules are precisely those reflexive generators which give NCCRs.\\
\t{(2)} $R$ has a NCCR $\iff$ $R$ has a NCCR given by a CM generator $M$ $\iff$ $\CM R$ contains a CT module.
\end{cor}
\begin{proof}
Notice that any reflexive generator $M$ which gives a NCCR is CM since
$R$ is a summand of $M$ and further $M\cong \Hom_R(R,M)$ is a summand of
$\End_R(M)\in \CM R$ as an $R$-module.\\
(1) By \ref{CT} CT modules are precisely the CM generators which give NCCRs.
The assertion follows from the above remark.\\
(2) The latter equivalence was shown in \ref{CT}. 
We only have to show ($\Rightarrow$) of the former assertion.   If $R$ has a NCCR $\Lambda$, then $\Lambda$ is an MMA (by \ref{NCCRgiveMM}) and so by \ref{gRgMRrigid2}(2) $R$ has an MMA $\Gamma=\End_R(M)$ where $M$ is a CM generator. But by \ref{allgMRderived}(2) $\Gamma$ and $\Lambda$ are derived equivalent, so since $\Lambda$ is an NCCR, so too is $\Gamma$ (\ref{closed under derived equivalences}(2)).
\end{proof}

Below is another characterization of CT modules, which is analogous to \cite[2.2.3]{IyamaAR}.  Compare this to the previous \ref{gMRapprox}.
\begin{prop}\label{CT2}
Assume $R$ is a 3-sCY normal domain and let $M\in\CM R$ with $R\in\add M$.  Then the following are equivalent\\
\t{(1)} $M$ is a CT module.\\
\t{(2)} $\End_R(M)\in\CM R$ and further for all $X\in\CM R$ there exists an exact sequence $0\to M_1\to M_0\stackrel{f}{\to} X\to 0$ with $M_1,M_0\in\add M$ and   a right $(\add M)$-approximation $f$.
\end{prop}
\begin{proof}
(1)$\Rightarrow$(2).  Fix $X\in \CM R$.  Since $R$ is 3-sCY, we have an exact sequence $0\to X\to P_0\to P_1$ with each $P_i\in \add R$. Applying $\Hom_R(M,-)$ gives an exact sequence
\[
0\to\Hom_R(M,X)\to\Hom_R(M,P_0)\stackrel{g}{\to}\Hom_R(M,P_1)\to\Cok g\to 0.
\]
Since both $\Hom_R(M,P_i)$ are projective $\End_R(M)$-modules by \ref{reflequiv}(1), and $\gl \End_R(M)=3$ by \ref{CT}, it follows that  $\pd_{\End_R(M)}\Hom_R(M,X)\leq 1$. Consequently we may take a projective resolution $0\to\Hom_R(M,M_1)\to\Hom_R(M,M_0)\to\Hom_R(M,X)\to 0$ which necessarily comes from a complex $0\to M_1\to M_0\to X\to 0$, again using \ref{reflequiv}(1).  This complex is itself exact since $M$ is a generator.\\
(2)$\Rightarrow$(1). Denote $\Gamma=\End_R(M)$.  By \ref{3CY-nonlocal} and \ref{2.14a} $\Gamma$ is a Gorenstein $R$-order.  By \ref{CT} we only have to show that  $\add M=\{ X\in\CM R : \Hom_R(M,X)\in\CM R \}$.  The assumption $\End_R(M)\in\CM R$ shows that the inclusion $\subseteq$ holds so let $X\in\CM R$ be such that $\Hom_R(M,X)\in\CM R$.  By assumption we may find $M_1,M_0\in\add M$ such that $0\to \Hom_R(M,M_1)\to \Hom_R(M,M_0)\to \Hom_R(M,X)\to 0$ is exact, hence $\pd_{\Gamma_\m}\Hom_R(M,X)_\m\leq 1$ for all $\m\in\Max R$ by \ref{reflequiv}(1).  Since $\Hom_R(M,X)_\m\in\CM R_\m$, Auslander--Buchsbaum (\ref{ABlocal}) implies that $\pd_{\Gamma_\m} \Hom_R(M,X)_\m=0$ for all $\m\in\Max R$ and hence $\Hom_R(M,X)$ is a projective $\Gamma$-module.  Since $M$ is a generator, $X\in\add M$.
\end{proof}

Provided an NCCR exists, the following shows the precise relationship between MM modules, CT modules and NCCRs.  Note that \ref{NCCRCTisgMR}(2) says that CT modules are really a special case of MM modules.
\begin{prop}\label{NCCRCTisgMR}
Let $R$ be a 3-sCY normal domain, and assume that $R$ has a NCCR (equivalently, by \ref{NCCRhasCM}, a CT module).  Then\\
\t{(1)} MM modules are precisely the reflexive modules which give NCCRs.\\
\t{(2)} MM modules which are CM (equivalently, by \ref{CMiffgen}, the MM generators) are precisely the CT modules.\\
\t{(3)} CT modules are precisely those CM modules which give NCCRs.
\end{prop}
\begin{proof}
(1)  ($\Leftarrow$) This is shown in \ref{NCCRgiveMM} above. \\
($\Rightarrow$)  Suppose that $M$ is an MM module, and let $\End_R(N)$ be a NCCR.  Then $\End_R(N)$ is an MMA by \ref{NCCRgiveMM} and so $\End_R(M)$ and $\End_R(N)$ are derived equivalent by \ref{closed in dimension three}.  This implies that $\End_R(M)$ is also an NCCR by \ref{closed under derived equivalences}(2).\\
(2) By (1) MM generators are precisely the CM generators which give NCCRs.
By \ref{CT} these are precisely the CT modules.\\
(3) Follows immediately from (1) and (2).
\end{proof}

In the remainder of this section we relate our work to that of the more common notions of $n$-rigid, maximal $n$-rigid and maximal $n$-orthogonal (=cluster tilting) modules in the case when $R$ is an isolated singularity.

Recall that $M\in\refl R$ is called \emph{$n$-rigid} if $\Ext^i_R(M,M)=0$ for all $1\leq i\leq n$.  We call $M\in \refl R$ \emph{maximal $n$-rigid} if $M$ is $n$-rigid and furthermore
it is maximal with respect to this property, namely if there exists
$X\in \refl R$ such that $M\oplus X$ is $n$-rigid, then $X\in \add M$.

Recall that $M\in\CM R$ is called a \emph{maximal $n$-orthogonal} module if  
\begin{eqnarray*}
\add M&=&\{ X\in\CM R \mid \Ext^{i}_R(M,X)=0 \mbox{ for all }1\leq i\leq n\}\\
&=&\{ X\in\CM R \mid \Ext^{i}_R(X,M)=0 \mbox{ for all }1\leq i\leq n\}. 
\end{eqnarray*}

\begin{prop}\label{rigidisolated}
Let $R$ be $d$-sCY with only isolated singularities, $M\in\CM R$.  Then \\
\t{(1)} $M$ is a modifying module if and only if it is ($d-2$)-rigid.\\
\t{(2)} $M$ is a maximal modifying module if and only if it is maximal ($d-2$)-rigid.\\
\t{(3)} $M$ is a CT module if and only if it is maximal ($d-2$)--orthogonal.
\end{prop}
\begin{proof}
Let $X,Y\in \CM R$. 
By \ref{depthford} and \ref{Extisfl}, it follows that $\Hom_R(X,Y)\in \CM R$ if and only if $ \Ext^i_R(X,Y)=0$ for all $1\leq i\leq d-2$.  Thus the assertions for (1), (2) and (3) follow. 
\end{proof}

\section{Mutations of Modifying Modules}\label{mutations}

\subsection{Mutations and Derived Equivalences in Dimension $d$}\label{mutations1}

Mutation is a technique used to obtain new modifying, maximal modifying and CT modules from a given one.  Many of our arguments work in the full generality of modifying modules although sometimes it is necessary to restrict to the maximal modifying level to apply certain arguments.  

Throughout this section $R$ will be a normal $d$-sCY ring, $d\geq 2$, and $M$ will be a modifying module with $N$ such that $0\neq N\in \add M$.  Note that $N$ may or may not be decomposable.  Given this, we define left and right mutation as in \ref{mutationdef} in the introduction: we have exact sequences 
\begin{align}
&0\to K_{0}\stackrel{c}\to N_0\stackrel{a}\to M\label{K0}\\
&0\to K_{1}\stackrel{d}\to N_1^{*}\stackrel{b}\to M^{*} \label{K1}
\end{align}
where $a$ is a right $(\add N)$-approximation and $b$ is a right $(\add N^{*})$-approximation.  We call them \emph{exchange sequences}.  From this we define $\MU{N}(M):=N \oplus K_{0}$ and $\NU{N}(M):=N \oplus K_{1}^{*}$.

Note that by the definition of approximations, $N_{0}, N_{1}\in\add N$ and we have exact sequences
\begin{align}
&0\to \Hom_{R}(N,K_{0})\stackrel{\cdot c}\to \Hom_{R}(N,N_0)\stackrel{\cdot a}\to \Hom_{R}(N,M)\to 0\label{K0a}\\
&0\to \Hom_{R}(N^{*},K_{1})\stackrel{\cdot d}\to \Hom_{R}(N^{*},N_1^{*})\stackrel{\cdot b}\to \Hom_{R}(N^{*},M^{*})\to 0. \label{K1a}
\end{align}

\begin{remark}
\t{(1)}  In general $\MU{N}(M)$ and $\NU{N}(M)$ are not the same.  Nevertheless, we will see later in some special cases that $\MU{N}(M)=\NU{N}(M)$ holds (\ref{main}), as in cluster tilting theory \cite[5.3]{Iyama-Yoshino}, \cite{GLS,BMRRT}. \\
\t{(2)}  A new feature of our mutation which is different from cluster tilting theory is that $\MU{N}(M)=M=\NU{N}(M)$ can happen.  A concrete example is given by taking $R=k[x,y,z]^G$ with $G=\frac{1}{2}(1,1,0)$, $M=k[x,y,z]$ and $N=R$. 
\end{remark}

\begin{remark}
If $d=3$, then both $\MU{N}(M)$ and $\NU{N}(M)$ are modifying $R$-modules by \ref{4.5refl}.  We will show in \ref{stillmodifying} that this is the case in any dimension.  
\end{remark}

We note that mutation is unique up to additive closure.  This can be improved if $R$ is complete local.
\begin{lemma}\label{uptoadd}
Suppose $N_0\stackrel{a}\to M$ and $N_0^\prime\stackrel{a^\prime}\to M$ are two right $(\add N)$-approximations of $M$.  Then $\add (N\oplus\Ker a)= \add (N\oplus\Ker a^{\prime})$.  A similar statement holds for left approximations.
\end{lemma}
\begin{proof}
Let $K:=\Ker a$ and $K':=\Ker a'$.
Then we have a commutative diagram
\[
{\SelectTips{cm}{10}
\xy
(0,0)*+{0}="0",(10,0)*+{K}="1",(25,0)*+{N_0}="2",(40,0)*+{M}="3",
(0,-10)*+{0}="0a",(10,-10)*+{K'}="1a",(25,-10)*+{N_0'}="2a",(40,-10)*+{M}="3a",
\ar"0";"1"
\ar"1";"2"^{c}
\ar"2";"3"^{a}
\ar"0a";"1a"
\ar"1a";"2a"^{c'}
\ar"2a";"3a"^{a'}
\ar"1";"1a"^s
\ar"2";"2a"^{t}
\ar@{=}"3";"3a"
\endxy}
\]
of exact sequences, giving an exact sequence
\begin{eqnarray}
0\to K\xrightarrow{(-s\ c)} K'\oplus N_{0}\xrightarrow{{c'\choose t}} N_{0}'.\label{61}
\end{eqnarray}
From the commutative diagram
\[
{\SelectTips{cm}{10}
\xy
(0,0)*+{0}="0",
(20,0)*+{\Hom_{R}(N,K)}="1",
(50,0)*+{\Hom_{R}(N,N_{0})}="2",
(80,0)*+{\Hom_{R}(N,M)}="3",
(98,0)*+{0}="4",
(0,-10)*+{0}="0a",
(20,-10)*+{\Hom_{R}(N,K')}="1a",
(50,-10)*+{\Hom_{R}(N,N'_{0})}="2a",
(80,-10)*+{\Hom_{R}(N,M)}="3a",
(98,-10)*+{0}="4a",
\ar"0";"1"
\ar^{\cdot c}"1";"2"
\ar^{\cdot a}"2";"3"
\ar"3";"4"
\ar"0a";"1a"
\ar^{\cdot c'}"1a";"2a"
\ar^{\cdot a'}"2a";"3a"
\ar"3a";"4a"
\ar^{\cdot s}"1";"1a"
\ar^{\cdot t}"2";"2a"
\ar@{=}"3";"3a"
\endxy}
\]
we see that 
\[
\Hom_{R}(N,K'\oplus N_{0})\xrightarrow{\cdot{c'\choose t}}\Hom_{R}(N,N_{0}')\to 0
\]
is exact. Thus (\ref{61}) is a split short exact sequence, so in particular $K\in\add (N\oplus K')$. Similarly $K'\in\add (N\oplus K)$.
\end{proof}

\begin{prop}\label{dualityofapprox}
Let $R$ be a normal $d$-sCY ring and let $M$ be a modifying module with $0\neq N\in \add M$ (i.e.\ notation as above). Then \\
\t{(1)} Applying $\Hom_{R}(-,N)$ to (\ref{K0}) induces an exact sequence
\begin{align}
&0\to \Hom_{R}(M,N)\stackrel{a\cdot}\to \Hom_{R}(N_{0},N)\stackrel{c\cdot}\to \Hom_{R}(K_{0},N)\to 0.\label{K0b}
\end{align}
In particular $c$ is a left $(\add N)$-approximation. \\
\t{(2)} Applying $\Hom_{R}(-,N^{*})$ to (\ref{K1}) induces an exact sequence
\begin{align}
&0\to \Hom_{R}(M^{*},N^{*})\stackrel{b\cdot}\to \Hom_{R}(N_{1}^{*},N^{*})\stackrel{d\cdot}\to \Hom_{R}(K_{1},N^{*})\to 0\label{K1b}
\end{align}
In particular $d$ is a left $(\add N^{*})$-approximation.\\
\t{(3)} We have that 
\begin{align}
&0\to M^{*}\stackrel{a^{*}}\to N_0^{*}\stackrel{c^{*}}\to K_{0}^{*}\label{K0D}\\
&0\to M\stackrel{b^{*}}\to N_1\stackrel{d^{*}}\to K_{1}^{*} \label{K1D}
\end{align}
are exact, inducing exact sequences
\begin{align}
&0\to \Hom_{R}(N^{*},M^{*})\stackrel{\cdot a^{*}}\to \Hom_{R}(N^{*},N_0^{*})\stackrel{\cdot c^{*}}\to \Hom_{R}(N^{*},K_{0}^{*})\to 0\label{K0Da}\\
&0\to \Hom_{R}(K_{0}^{*},N^{*})\stackrel{c^{*}\cdot}\to \Hom_{R}(N_{0}^{*},N^{*})\stackrel{a^{*}\cdot }\to \Hom_{R}(M^{*},N^{*})\to 0\label{K0Db}\\
&0\to \Hom_{R}(N,M)\stackrel{\cdot b^{*}}\to \Hom_{R}(N,N_1)\stackrel{\cdot d^{*}}\to \Hom_{R}(N,K_{1}^{*})\to 0 \label{K1Da}\\
&0\to \Hom_{R}(K_{1}^{*},N)\stackrel{d^{*}\cdot }\to \Hom_{R}(N_{1},N)\stackrel{b^{*}\cdot }\to \Hom_{R}(M,N)\to 0 \label{K1Db}
\end{align}
\end{prop}
\begin{proof}
Denote $\Lambda:=\End_{R}(N)$ and $\mathbb{F}:=\Hom_{R}(N,-)$.\\
(1) We note that (\ref{K0a}) is 
\[
0\to\mathbb{F}K_{0}\to\mathbb{F}N_{0}\to\mathbb{F}M\to 0
\]
so applying $\Hom_{\Lambda}(-,\mathbb{F}N)$ gives 
\[
0\to\Hom_{\Lambda}(\mathbb{F}M,\mathbb{F}N)\to\Hom_{\Lambda}(\mathbb{F}N_{0},\mathbb{F}N)\to\Hom_{\Lambda}(\mathbb{F}K_{0},\mathbb{F}N)\to \Ext^{1}_{\Lambda}(\mathbb{F}M,\Lambda).
\]
But by \ref{3CY-nonlocal} $\Lambda$ is $d$-sCY and thus a Gorenstein $R$-order by \ref{2.14a}.  Since $\mathbb{F}M\in\CM \Lambda$ and $\add\Lambda=\add\omega_\Lambda$ by \ref{addLambda=addOmega}, it follows that $\Ext^{1}_{\Lambda}(\mathbb{F}M,\Lambda)=0$ and hence we have a commutative diagram of complexes
\[
{\SelectTips{cm}{10}
\xy
(0,0)*+{0}="0",
(20,0)*+{\Hom_{\Lambda}(\mathbb{F}M,\mathbb{F}N)}="1",
(50,0)*+{\Hom_{\Lambda}(\mathbb{F}N_{0},\mathbb{F}N)}="2",
(80,0)*+{\Hom_{\Lambda}(\mathbb{F}K_{0},\mathbb{F}N)}="3",
(100,0)*+{0}="4",
(0,-10)*+{0}="0a",
(20,-10)*+{\Hom_{R}(M,N)}="1a",
(50,-10)*+{\Hom_{R}(N_{0},N)}="2a",
(80,-10)*+{\Hom_{R}(K_{0},N)}="3a",
(100,-10)*+{0}="4a",
\ar"0";"1"
\ar"1";"2"
\ar"2";"3"
\ar"3";"4"
\ar"0a";"1a"
\ar"1a";"2a"^{a\cdot}
\ar"2a";"3a"^{c\cdot}
\ar"3a";"4a"
\ar@{=}"1";"1a"
\ar@{=}"2";"2a"
\ar@{=}"3";"3a"
\endxy}
\]
in which the top row is exact and the vertical maps are isomorphisms by reflexive equivalence \ref{reflequiv}(4).  It follows that the bottom row is exact.\\
(2) is identical to (1) since $\Hom_{R}(N^{*},M^{*})\in\CM R$.\\
(3) As in (1) applying $\Hom_{\Lambda}(-,\mathbb{F}R)$ to (\ref{K0a}) gives an commutative diagram of complexes
\[
{\SelectTips{cm}{10}
\xy
(0,0)*+{0}="0",
(20,0)*+{\Hom_{\Lambda}(\mathbb{F}M,\mathbb{F}R)}="1",
(50,0)*+{\Hom_{\Lambda}(\mathbb{F}N_{0},\mathbb{F}R)}="2",
(80,0)*+{\Hom_{\Lambda}(\mathbb{F}K_{0},\mathbb{F}R)}="3",
(0,-10)*+{0}="0a",
(20,-10)*+{\Hom_{R}(M,R)}="1a",
(50,-10)*+{\Hom_{R}(N_{0},R)}="2a",
(80,-10)*+{\Hom_{R}(K_{0},R)}="3a",
\ar"0";"1"
\ar"1";"2"
\ar"2";"3"
\ar"0a";"1a"
\ar"1a";"2a"^{a^{*}}
\ar"2a";"3a"^{c^{*}}
\ar@{=}"1";"1a"
\ar@{=}"2";"2a"
\ar@{=}"3";"3a"
\endxy}
\]
in which the top row is exact.  Hence the bottom row (i.e.\ (\ref{K0D})) is exact.  The proof that (\ref{K1D}) is exact is identical.  Now since $(-)^*:\refl R\to \refl R$ is a duality, the sequences (\ref{K0Da}), (\ref{K0Db}), (\ref{K1Da}) and (\ref{K1Db}) are identical with (\ref{K0b}), (\ref{K0a}), (\ref{K1b}) and (\ref{K1a}) respectively.
Thus they are exact.
\end{proof}

\begin{prop}\label{inverseoperations}
$\MU{N}$ and $\NU{N}$ are mutually inverse operations, i.e.\ we have $\NU{N}(\MU{N}(M))= M$ and $\MU{N}(\NU{N}(M))= M$, up to additive closure.
\end{prop}
\begin{proof}
Since (\ref{K0D}) and (\ref{K0Da}) are exact, we have $\NU{N}(\MU{N}(M))= M$. The other assertion follows dually.
\end{proof}

The following is standard in the theory of tilting mutation \cite{RS}.
\begin{lemma}\label{tiltingmutation}
Let $\Lambda$ be a ring, let $Q$ be a projective $\Lambda$-module and consider an exact sequence $\Lambda\stackrel{f}\to Q^\prime\stackrel{g}{\to}\Cok f\to 0$ where $f$ is a left $(\add Q)$-approximation. If $f$ is injective then $Q\oplus \Cok f$ is a tilting $\Lambda$-module of projective dimension at most one.
\end{lemma}
\begin{proof}
For the convenience of the reader we give a complete proof here.
It is clear that $\pd_\Lambda (Q\oplus \Cok f)\leq 1$ and it generates the derived category.  We need only to check that $\Ext^1_\Lambda(Q\oplus\Cok f,Q\oplus\Cok f)=0$. 
Applying $\Hom_\Lambda(-,Q)$, we have an exact sequence
\[
\Hom_\Lambda(Q',Q)\stackrel{f\cdot}{\to}\Hom_\Lambda(\Lambda,Q)\to\Ext^1_\Lambda(\Cok f,Q)\to0.
\]
Since $(f\cdot)$ is surjective, we have $\Ext^1_\Lambda(\Cok f,Q)=0$.
Applying $\Hom_{\Lambda}(-,\Cok f)$, we have an exact sequence
\[
\Hom_\Lambda(Q',\Cok f)\stackrel{f\cdot}{\to}\Hom_\Lambda(\Lambda,\Cok f)\to\Ext^1_\Lambda(\Cok f,\Cok f)\to 0.
\]
Here $(f\cdot)$ is surjective since $\Hom_{\Lambda}(Q',Q')\stackrel{f\cdot}{\to}\Hom_{\Lambda}(\Lambda,Q')$ and $\Hom_{\Lambda}(\Lambda,Q')\stackrel{\cdot g}{\to}\Hom_{\Lambda}(\Lambda,\Cok f)$ are surjective.
Thus we have  $\Ext^1_\Lambda(\Cok f,\Cok f)=0$.
Consequently we have $\Ext^1_{\Lambda}(Q\oplus\Cok f,Q\oplus\Cok f)=0$ since $Q$ is projective.
\end{proof}

The proof of \ref{welldefined2} requires the following technical lemma.  

\begin{lemma}\label{isook}
Let $R$ be a normal domain, let $\Lambda\in\refl R$ be a module finite $R$-algebra and let $T\in\mod\Lambda$ be a height one projective (i.e.\ $T_\p$ is a projective $\Lambda_\p$-module for all $\p\in\Spec R$ with $\hgt\p\leq 1$) such that $\End_\Lambda(T)\in\refl R$.  Then $\End_{\Lambda}(T)\cong \End_{\Lambda^{\op}}(T^{*})^{\op}$.
\end{lemma}
\begin{proof}
Consider the natural ring homomorphism
\[
\End_{\Lambda}(T)\stackrel{\psi:=(-)^*}{\xrightarrow{\hspace*{0.75cm}}} \End_{\Lambda^{\op}}(T^{*})^{\op}
\]
where recall $(-)^*:=\Hom_{R}(-,R)$.  Note that $T^*\in\refl R$ by \ref{reflequiv}(2), i.e. $T^*\in\refl\Lambda^{\op}$. This implies $\End_{\Lambda^{\op}}(T^{*})^{\op}\in\refl R$ by \ref{reflequiv}(2).

Since $T$ is a height one projective and $\Lambda\in\refl R$,  it follows that $T_\p\in\refl \Lambda_\p$ for all height one primes $\p$.  Hence, via the anti--equivalence
\[
\refl\Lambda_{\p}\stackrel{(-)^*_\p}{\xrightarrow{\hspace*{0.5cm}}}\refl\Lambda_{\p}^{\op} ,
\]
we have that $\psi$ is a height one isomorphism.  

By assumption $\End_\Lambda(T)\in\refl R$ holds. Since $R$ is normal, $\psi$, being a height one isomorphism between reflexive $R$-modules, is actually an isomorphism.
\end{proof}

\begin{thm}\label{welldefined2}
Let $R$ be a normal $d$-sCY ring with modifying module $M$.   Suppose $0\neq N\in \add M$.  Then\\
\t{(1)} $\End_R(M)$ and $\End_R(\NU{N}(M))$ are derived equivalent.\\
\t{(2)} $\End_R(M)$ and $\End_R(\MU{N}(M))$ are derived equivalent.
\end{thm}
\begin{proof}
(1) Denote $\Lambda:=\End_{R}(M)$ and $\mathbb{F}:=\Hom_R(M,-)\colon\refl R\to \refl\Lambda$.  Applying $\mathbb{F}$ to (\ref{K1D}) and denoting
$V:=\Cok(\cdot b^*)$, we have an exact sequence 
\begin{equation}
\SelectTips{cm}{10}
\xy0;/r.275pc/:
\POS(18,0)*+{0}="0",(30,0)*+{\mathbb{F}M}="1",(50,0)*+{\mathbb{F}N_{1}}="2",(70,0)*+{\mathbb{F}K_{1}^{*}}="3",(60,-6.5)*+{V}="b1",
\ar"0";"1"
\ar"1";"2"^{(\cdot b^*)}
\ar"2";"3"
\ar@{->>}"2";"b1"
\ar@{^{(}->}"b1";"3"_(0.45){h}
\endxy .\label{splicedsequence}
\end{equation}
We now claim that $(\cdot b^*)$ is a left $(\add Q)$-approximation where $Q:=\Hom_{R}(M,N)=\mathbb{F}N$. Simply applying $\Hom_{\Lambda}(-,Q)=\Hom_{\Lambda}(-,\mathbb{F}N)$ to the above we obtain
\[
{\SelectTips{cm}{10}
\xy
(0,0)*+{\Hom_{\Lambda}(\mathbb{F}N_{1},\mathbb{F}N)}="1",
(40,0)*+{\Hom_{\Lambda}(\mathbb{F}M,\mathbb{F}N)}="2",
(0,-10)*+{\Hom_{R}(N_{1},N)}="1a",
(40,-10)*+{\Hom_{R}(M,N)}="2a",
(60,-10)*+{0}="3a",
\ar"1";"2"
\ar"1a";"2a"
\ar"2a";"3a"
\ar@{=}"1";"1a"
\ar@{=}"2";"2a"
\endxy}
\]
where the bottom is just (\ref{K1Db}) and so is exact, and the vertical maps are isomorphisms by reflexive equivalence \ref{reflequiv}(4). Hence the top is surjective, showing that $(\cdot b^*)$ is a left $(\add Q)$-approximation.  By \ref{tiltingmutation} it follows that $Q\oplus V$ is a tilting $\Lambda$-module.

We now show that $\End_{\Lambda}(V\oplus Q)\cong \End_R(\NU{N}(M))$ by using \ref{isook}.  To do this, note first that certainly $\Lambda\in\refl R$ since $\Lambda\in\CM R$, and further $\Lambda$ is $d$-sCY by \ref{3CY-nonlocal}(2).  Hence $\End_\Lambda(V\oplus Q)$, being derived equivalent to $\Lambda$, is also $d$-sCY and so $\End_\Lambda(V\oplus Q)\in\refl R$ by \ref{2.14a}.  We now claim that $V\oplus Q$ is a height one projective $\Lambda$-module.

Let $\p\in\Spec R$ be a height one prime, then $M_\p\in\refl R_\p=\add R_\p$.  Hence $M_\p$ is a free $R_\p$-module, and so  $\add N_\p=\add M_\p$.  Localizing (\ref{K1Da}) gives an exact sequence
\[
0\to \Hom_{R_\p}(N_\p,M_\p)\to \Hom_{R_\p}(N_\p,(N_1)_\p)\to \Hom_{R_\p}(N_\p,(K_{1}^{*})_\p)\to 0
\]
and so since $\add N_\p=\add M_\p$,  
\[
0\to \Hom_{R_\p}(M_\p,M_\p)\to \Hom_{R_\p}(M_\p,(N_1)_\p)\to \Hom_{R_\p}(M_\p,(K_{1}^{*})_\p)\to 0
\]
is exact.  This is (\ref{splicedsequence}) localized at $\p$, hence we conclude that $h$ is a height one isomorphism.  In particular $V_\p=\Hom_{R_\p}(M_\p,(K_{1}^{*})_\p)$ with both $M_\p,(K_{1}^{*})_\p\in\add R_\p$.  Consequently $V$, thus $V\oplus Q$, is a height one projective $\Lambda$-module. 

Thus by \ref{isook} we have an isomorphism
\[
\End_{\Lambda}(V\oplus Q)\cong \End_{\Lambda^{\op}}(V^{*}\oplus Q^*)^{\op}.
\]
Now since $h$ is a height one isomorphism, it follows that $h^{*}$ is a height one isomorphism.  But $h^{*}$ is a morphism between reflexive modules, so $h^{*}$ must be an isomorphism.  We thus have
\[
V^{*}\oplus Q^*=(\mathbb{F}(K_1^*))^{*}\oplus Q^{*}= (\mathbb{F}(K_1^*))^*\oplus (\mathbb{F}N)^*=(\mathbb{F}(K_1^*\oplus N))^*.
\]
Consequently 
\[
\End_{\Lambda}(V\oplus Q)\cong
\End_{\Lambda^{\op}}((\mathbb{F}(K_1^*\oplus N))^*)^{\op}\cong  
\End_{\Lambda}(\mathbb{F}(K_1^*\oplus N))
\]
since
\[
(\mathbb{F}(K_1^*\oplus N))^*\in\refl\Lambda^{\op}\stackrel{(-)^{*}}{\xrightarrow{\hspace*{0.75cm}}}\refl\Lambda
\]
is an anti--equivalence.  This then yields
\[
\End_{\Lambda}(V\oplus Q)\cong 
\End_{\Lambda}(\mathbb{F}(K_1^*\oplus N))
\cong\End_{R}(K_1^*\oplus N)=\End_R(\NU{N}(M)),
\]
where the second isomorphism follows from reflexive equivalence \ref{reflequiv}.\\
(2) Since $M^*$ is a modifying $R$-module, by (1) $\End_R(M^*)$ and $\End_{R}(\NU{N^*}(M^*))$ are derived equivalent.  But $\NU{N^*}(M^*)=( \MU{N}(M))^{*}$, so $\End_R(M^*)$ and $\End_{R}(( \MU{N}(M))^{*})$ are derived equivalent.  Hence $\End_R(M)^{\op}$ and $\End_{R}(\MU{N}(M))^{\op}$ are derived equivalent, which forces $\End_R(M)$ and $\End_{R}(\MU{N}(M))$ to be derived equivalent \cite[9.1]{Rickard}.
\end{proof}

\begin{remark}\label{definTJ}
By \ref{welldefined2}, for every $0\neq N\in \add M$ we obtain an equivalence 
\[
T_{N}:=\RHom(V\oplus Q,-) :\Db(\mod\End_{R}(M))\to \Db(\mod\End_{R}(\NU{N}(M))).
\]
Sometimes $\NU{N}(M)=M$ can happen (see next subsection),
but the functor $T_N$ is never the identity provided $\add N\neq \add M$.  This gives a way of generating  autoequivalences of the derived category.
\end{remark}

\begin{thm}\label{stillmodifying}
Let $R$ be a normal $d$-sCY ring with modifying module $M$.   Suppose $0\neq N\in \add M$.  Then\\
\t{(1)} $\MU{N}(M)$ and $\NU{N}(M)$ are modifying $R$-modules.\\
\t{(2)} If $M$ gives an NCCR, so do $\MU{N}(M)$ and $\NU{N}(M)$.\\
\t{(3)} Whenever $N$ is a generator, if $M$ is a CT module so are $\MU{N}(M)$ and $\NU{N}(M)$.\\
\t{(4)} Whenever $\dim\Sing R\leq 1$ (e.g.\ if $d=3$), if $M$ is a MM module so are $\MU{N}(M)$ and $\NU{N}(M)$.
\end{thm}
\begin{proof}
Set $\Lambda:=\End_R(M)$.  By \ref{welldefined2}, $\Lambda$, $\End_R(\NU{N}(M))$ and $\End_R(\MU{N}(M))$ are all derived equivalent.  Hence (1) follows from \ref{closed under derived equivalences}(1), (2) follows from \ref{closed under derived equivalences}(2) and (4) follows from \ref{MMAs dim3 closed}(2).\\
(3) Since $M$ is CT, by definition $M\in\CM R$.  But $N$ is a generator, so the $a$ and $b$ in the exchange sequences (\ref{K0}) and (\ref{K1}) are surjective.  Consequently both $\MU{N}(M)$ and $\NU{N}(M)$ are CM $R$-modules, so the result follows from (2) and \ref{CT}.\\
\end{proof}

One further corollary to \ref{stillmodifying} is the following application to syzygies and cosyzygies. Usually syzygies and cosyzygies are only defined up to free summands, so let us first settle some notation.  Suppose that $R$ is a normal $d$-sCY ring and $M$ is a modifying generator.  Since $M$ and $M^*$ are finitely generated we can consider exact sequences
\begin{align}
&0\to K_0\to P_0\to M\to 0\label{syz1}\\
&0\to K_1\to P_1^*\to M^*\to 0 \label{cosyz1}
\end{align}
where $P_0,P_1\in\add R$.   We define $\Omega M:=R\oplus K_0=\MU{R}(M)$ and $\Omega^{-1} M:=R\oplus K_1^*=\NU{R}(M)$.  Inductively we define $\Omega^iM$ for all $i\in\mathbb{Z}$.

Our next result shows that modifying modules often come in infinite families, and that in particular NCCRs often come in infinite families:

\begin{cor}\label{syzcosyz result}
Suppose that $R$ is a normal $d$-sCY ring and $M\in\refl R$ is a modifying generator.  Then\\
\t{(1)}  $\End_R(\Omega^iM)$ are derived equivalent for all $i\in\mathbb{Z}$.\\
\t{(2)} $\Omega^iM\in\CM R$ is a modifying generator for all $i\in\mathbb{Z}$.
\end{cor}
\begin{proof}
The assertions are immediate from \ref{welldefined2} and \ref{stillmodifying}.
\end{proof}

\subsection{Mutations and Derived Equivalences in Dimension $3$}  In the special case $d=3$, we can extend some of the above results, since we have more control over the tilting modules produced from the procedure of mutation.  Recall from the introduction that given $0\neq N\in \add M$ we define $[N]$ to be the two-sided ideal of $\Lambda:=\End_R(M)$ consisting of morphisms $M\to M$ which factor through a member of $\add N$. 

The factors $\Lambda_N:=\Lambda/[N]$ are, in some sense, replacements for simple modules in the infinite global dimension setting.  For example, we have the following necessary condition for a module to be MM.

\begin{prop}\label{almost characterize}
Suppose that $R$ is a normal 3-sCY ring, let $M$ be an MM $R$-module    and denote $\Lambda=\End_R(M)$.  Then $\pd_\Lambda\Lambda_N\leq 3$ for all  $N$ such that $0\neq N\in \add M$.
\end{prop}
\begin{proof}
The sequence (\ref{K0}) $0\to K_{0}\to N_0\to M$ gives 
\[
0\to \Hom_{R}(M,K_{0})\to \Hom_{R}(M,N_{0})\to \Lambda\to \Lambda_N\to 0
\]
where $\Hom_{R}(M,N_{0})$ and $\Lambda$ are projective $\Lambda$-modules.  But  $K_{0}$ is a modifying module by \ref{stillmodifying}, so by \ref{gMRapprox} we know that $\pd_{\Lambda} \Hom_{R}(M,K_{0})\leq 1$.  Hence certainly $\pd_\Lambda\Lambda_N\leq 3$.  
\end{proof}

\begin{remark}
The converse of \ref{almost characterize} is not true, i.e.\ there exists non-maximal modifying modules $M$ such that $\pd_\Lambda\Lambda_N\le 3$ for all $0\neq N\in \add M$.  An easy example is given by $M:=R\oplus (a,c^2)$ for $R:=\mathbb{C}[[a,b,c,d]]/(ab-c^4)$.  In this case the right $(\add R)$-approximation
\begin{eqnarray*}
0\to (a,c^2)\xrightarrow{(-\frac{c^2}{a}\ inc)}R\oplus R\xrightarrow{a\choose c^2}(a,c^2)\to 0
\end{eqnarray*}
shows that $\pd_\Lambda(\Lambda/[(a,c^2)])=2$, whilst the right $(\add (a,c^2))$-approximation
\begin{eqnarray*}
0\to R\xrightarrow{(-a\ c^2)}(a,c^2)\oplus (a,c^2)\xrightarrow{\left(\begin{smallmatrix}inc\\ \frac{c^2}{a}\end{smallmatrix}\right)}R
\end{eqnarray*}
shows that $\pd_\Lambda(\Lambda/[R])=2$.  Also $\Lambda/[M]=0$ and so trivially $\pd_\Lambda(\Lambda/[M])=0$.  Hence $\pd_\Lambda\Lambda_N\leq 3$ for all $0\neq N\in\add M$, however $\End_R(M\oplus(a,c))\in\CM R$ with $(a,c)\notin\add M$, so $M$ is not an MM module.
\end{remark}

Roughly speaking, mutation in dimension $d=3$ is controlled by the factor algebra $\Lambda_N$, in particular whether it is artinian or not. When it is artinian, the derived equivalence in  \ref{welldefined2} is given by a very explicit tilting module. 

\begin{thm}\label{Aifinite1}
Let $R$ be a normal 3-sCY ring with modifying module $M$.   Suppose that $0\neq N\in \add M$ and denote $\Lambda=\End_R(M)$.  If $\Lambda_N=\Lambda/[N]$ is artinian then \\
\t{(1)}  $T_{1}:=\Hom_R(M,\NU{N}(M))$ is a tilting $\Lambda$-module such that $\End_\Lambda(T_{1})\cong \End_R(\NU{N}(M))$. \\
\t{(2)} $T_{2}:=\Hom_R(M^{*},\MU{N}(M)^{*})$ is a tilting $\Lambda^{\rm op}$-module such that $\End_{\Lambda^{\rm op}}(T_{2})\cong \End_R(\MU{N}(M))^{\rm op}$. 
\end{thm}

\begin{remark}\label{flartinian}
In the setting of \ref{Aifinite1}, we have the following.\\
\t{(1)} $\Lambda_N$ is artinian if and only if $\add M_\p=\add N_\p$ for all $\p\in \Spec R$ with $\hgt \p=2$.\\
\t{(2)} If $R$ is finitely generated over a field $k$ then $\Lambda_N$ is artinian if and only if $\dim_k\Lambda_N<\infty$.  Thus if the reader is willing to work over $\C{}$, they may replace the condition $\Lambda_N$ is artinian by $\dim_\C{}\Lambda_N<\infty$ throughout.
\end{remark}
\noindent
{\it Proof of \ref{Aifinite1}.}
(1) Denote $\mathbb{G}:=\Hom_{R}(N,-)$ and $\Gamma:=\End_{R}(N)$. Applying $\Hom_{R}(M,-)$ to (\ref{K1D}) and $\Hom_{\Gamma}(\mathbb{G}M,-)$ to \eqref{K1Da} gives a commutative diagram
\[
{\SelectTips{cm}{10}
\xy
(0,0)*+{0}="0",
(20,0)*+{\Hom_{\Gamma}(\mathbb{G}M,\mathbb{G}M)}="1",
(52,0)*+{\Hom_{\Gamma}(\mathbb{G}M,\mathbb{G}N_{1})}="2",
(84,0)*+{\Hom_{\Gamma}(\mathbb{G}M,\mathbb{G}K_{1}^{*})}="3",
(116,0)*+{\Ext_{\Gamma}^{1}(\mathbb{G}M,\mathbb{G}M)}="4",
(0,-10)*+{0}="0a",
(20,-10)*+{\Hom_{R}(M,M)}="1a",
(52,-10)*+{\Hom_{R}(M,N_{1})}="2a",
(84,-10)*+{\Hom_{R}(M,K_{1}^*)}="3a",
(108,-10)*+{C}="4a",
(120,-10)*+{0}="5a",
\ar"0";"1"
\ar"1";"2"
\ar"2";"3"
\ar"3";"4"
\ar"0a";"1a"
\ar"1a";"2a"^{\cdot b^{*}}
\ar"2a";"3a"^{\cdot d^{*}}
\ar"3a";"4a"
\ar"4a";"5a"
\ar@{=}"1";"1a"
\ar@{=}"2";"2a"
\ar@{=}"3";"3a"
\endxy}
\]
where the vertical maps are isomorphisms by \ref{reflequiv}(4), hence $C\subseteq \Ext_{\Gamma}^{1}(\mathbb{G}M,\mathbb{G}M)$.  We first claim that $C=0$.  Since $\End_\Gamma(\mathbb{G}M)\cong \Lambda$ by reflexive equivalence \ref{reflequiv},  by \ref{reflandCM} $\fl \Ext_{\Gamma}^{1}(\mathbb{G}M,\mathbb{G}M)=0$.  On the other hand $\Hom_{R}(N,-)$ applied to (\ref{K1D}) is exact (by \ref{dualityofapprox}), so $C$ is annihilated by $[N]$ and consequently $C$ is a $\Lambda_{N}$-module. Since  $\Lambda_{N}$ is artinian so too is $C$, thus it has finite length.  Hence $C=0$ and so 
\begin{eqnarray}
0\to\Hom_{R}(M,M)\to \Hom_{R}(M,N_{1})\to \Hom_{R}(M,K_{1}^{*})\to 0\label{LAB2}
\end{eqnarray}
is exact. Thus the tilting module $V\oplus Q$ in the proof of \ref{welldefined2}(1) is simply $\Hom_{R}(M,K_{1}^{*})\oplus \Hom_{R}(M,N)=T_1$. The remaining statements are contained in \ref{welldefined2}(1).\\
(2) Similarly to the above one can show that applying $\Hom_{R}(M^*,-)$ to \eqref{K0D}  gives an exact sequence
\begin{eqnarray}
0\to\Hom_{R}(M^{*},M^{*})\to \Hom_{R}(M^{*},N_{0}^{*})\to \Hom_{R}(M^{*},K_{0}^{*})\to 0\label{LAB1}
\end{eqnarray}
and so the tilting module $V\oplus Q$ in the proof of \ref{welldefined2}(2) is simply $\Hom_{R}(M^{*},K_{0}^{*})\oplus \Hom_{R}(M^{*},N^{*})=\Hom_{R}(M^{*},\MU{N}(M)^{*})$.
\qed

\begin{remark}
Note that the statement in \ref{Aifinite1} is quite subtle.  
There are examples where $\Hom_R(M,\MU{N}(M))$
(respectively, $\Hom_R(M^*,\NU{N}(M)^*)$) is \emph{not} a
tilting $\End_R(M)$-module (respectively,
$\End_R(M)^{\op}$-module).  Note however that these are always tilting modules if $M$ is an MM module, by combining \ref{Homtilting}(2) and \ref{stillmodifying}(4).
\end{remark}

If $\Lambda_N$ is artinian, the module $M$ changes under mutation:

\begin{prop}\label{Aifinite2}
Let $R$ be a normal 3-sCY ring with modifying module $M$.   Suppose $0\neq N\in \add M$, denote $\Lambda=\End_R(M)$ and define $\Lambda_N:=\Lambda/[N]$.  If $\Lambda_N$ is artinian then \\
\t{(1)} If $\add N\neq\add M$ then $\add \MU{N}(M)\neq \add M$.\\
\t{(2)}  If $\add N\neq\add M$ then $\add\NU{N}(M)\neq \add M$.  
\end{prop}
\begin{proof}
(1) Since $\Lambda_N$ is artinian, the sequence (\ref{LAB1})
\[
0\to \Hom_R(M^{*},M^{*})\to \Hom_R(M^{*},N_0^{*})\to \Hom_R(M^{*},K_{0}^{*})\to 0 
\]
is exact.  If this splits then by reflexive equivalence (\ref{reflequiv}(4)) $M^{*}$ is a summand of $N_0^{*}$, contradicting $\add N\neq\add M$.  Thus the above cannot split so $\Hom_R(M^{*},K_{0}^{*})$ cannot be projective, hence certainly $K_{0}^{*}\notin\add M^{*}$ and so $K_{0}\notin\add M$.  This implies $\add \MU{N}(M)\neq \add M$.\\
\t{(2)} Similarly, the exact sequence (\ref{LAB2}) cannot split, so $K_{1}^{*}\notin\add M$.  
\end{proof}

\begin{remark} It is natural to ask under what circumstances the hypothesis $\Lambda_N$ is artinian in \ref{Aifinite1}, \ref{Aifinite2} holds.  In the situation of \ref{1dfibre} the answer seems to be related to the contractibility of the corresponding curves; we will come back to this question in future work.  
\end{remark}

One case where $\Lambda_N$ is always artinian is when $R$ has isolated singularities:

\begin{lemma}\label{depthoffactor}
Suppose $R$ is a normal 3-sCY ring.  Let $M$ be a modifying module with $0\neq N\in \add M$,  denote $\Lambda=\End_R(M)$ and set $\Lambda_N=\Lambda/[N]$. Then \\
\t{(1)} $\dim_{R} \Lambda_N\leq 1$.\\
\t{(2)} $\depth_{R_\m}{(\Lambda_N)}_{\m}\leq 1$ for all $\m\in\Max R$.\\
\t{(3)} If $R$ is an isolated singularity then $\Lambda_N$ is artinian.\\
\t{(4)} If $\pd_{\Lambda}\Lambda_{N}<\infty$ then $\id_{\Lambda_{N}}\Lambda_{N}\leq 1$.
\end{lemma}
\begin{proof}
(1)  We have $(\End_R(M)/[N])_\p=\End_{R_\p}(M_\p)/[N_\p]$ for all $\p\in \Spec R$. Since $R$ is normal, $\add M_\p=\add R_\p=\add N_\p$ for all $\p\in \Spec R$ with $\hgt \p=1$.
Thus we have $(\End_R(M)/[N])_\p=\End_{R_\p}(M_\p)/[N_\p]=0$ for all these primes, and so the assertion follows.\\
(2) is immediate from (1).\\
(3) If $R$ is isolated then by the argument in the proof of (1) we have $\dim_R\Lambda_N=0$ and so $\Lambda_N$ is supported only at a finite number of maximal ideals.  Hence $\Lambda_N$ has finite length and so $\Lambda_N$ is artinian. \\
(4) Notice that $\Lambda$ is 3-sCY by \ref{3CY-nonlocal}.
Hence the assertion follows from \cite[5.5(3)]{IR} for 3-CY algebras,
which is also valid for 3-sCY algebras under the assumption that
$\pd_\Lambda\Lambda_N<\infty$.
\end{proof}

We now show that mutation does not change the factor algebra $\Lambda_N$.  Suppose $M$ is modifying and $N$ is such that $0\neq N\in \add M$, and consider an exchange sequence (\ref{K0})
\[
0\to K_0\stackrel{c}\to N_0\stackrel{a}\to M.
\]
We know by definition that $a$ is a right $(\add N)$-approximation, and by (\ref{K0b}) that $c$ is a left $(\add N)$-approximation. 

Since $\Lambda_N$ is by definition $\End_R(M)$ factored out by the ideal of all morphisms $M\to M$ which factor through a module  in $\add N$, in light of the approximation property of the map $a$, this ideal is the just the ideal $I_a$ of all morphisms $M\to M$ which factor as $xa$ where $x$ is some morphism $M\to N_{0}$.  Thus $\Lambda_N=\End_R(M)/I_a$.  

On the other hand taking the choice $\MU{N}(M)=K_{0}\oplus N$ coming from the above exchange sequence, $\Lambda'_{N}$ is by definition $\End_{R}(\MU{N}(M))=\End_R(K_{0}\oplus N)$ factored out by the ideal of all morphisms $K_{0}\oplus N\to K_{0}\oplus N$ which factor through a module in $\add N$.  Clearly this is just $\End_R(K_{0})$ factored out by those morphisms which factor through $\add N$. In light of the approximation property of the map $c$,  $\Lambda^\prime_N=\End_R(K_0)/I_c$ where $I_c$ is the ideal of all morphisms $K_0\to K_0$ which factor as $cy$ where $y$ is some morphism $K_0\to N_0$.

\begin{thm}
Let $R$ be a normal $d$-sCY ring, and let $M$ be a modifying module with $0\neq N\in \add M$. With the notation and choice of exchange sequence as above, we have $\Lambda_N\cong \Lambda^\prime_N$ as $R$-algebras. In particular\\
\t{(1)} $\Lambda'_{N}$ is independent of the choice of exchange sequence, up to isomorphism.\\
\t{(2)} $\Lambda_N$ is artinian if and only if $\Lambda^\prime_N$ is artinian. 
\end{thm}
\begin{proof}
We construct a map $\alpha:\Lambda_N=\End_R(M)/I_a\to\End_R(K_0)/I_c=\Lambda^\prime_N$ as follows: given $f\in\End_R(M)$ we have
\[
{\SelectTips{cm}{10}
\xy
(0,0)*+{0}="0",
(10,0)*+{K_{0}}="1",
(25,0)*+{N_{0}}="2",
(40,0)*+{M}="3",
(0,-12)*+{0}="0a",
(10,-12)*+{K_{0}}="1a",
(25,-12)*+{N_{0}}="2a",
(40,-12)*+{M}="3a",
\ar"0";"1"
\ar"1";"2"^c
\ar"2";"3"^a
\ar"0a";"1a"
\ar"1a";"2a"^c
\ar"2a";"3a"^a
\ar@{.>}"1";"1a"^{\exists\,h_{f}}
\ar@{.>}"2";"2a"^{\exists\,g_f}
\ar"3";"3a"^{f}
\endxy}
\]
where the $g_f$ exists (non-uniquely) since $a$ is an approximation.  Define $\alpha$ by $\alpha(f+I_a)=h_{f}+I_c$. We will show that $\alpha:\Lambda_N\to\Lambda^\prime_N$ is a well-defined map, which is independent of the choice of $g_f$.
Take $f'\in\Lambda$ satisfying $f-f'\in I_a$. We have a commutative diagram
\[
{\SelectTips{cm}{10}
\xy
(0,0)*+{0}="0",
(10,0)*+{K_{0}}="1",
(25,0)*+{N_{0}}="2",
(40,0)*+{M}="3",
(0,-12)*+{0}="0a",
(10,-12)*+{K_{0}}="1a",
(25,-12)*+{N_{0}}="2a",
(40,-12)*+{M}="3a",
\ar"0";"1"
\ar"1";"2"^c
\ar"2";"3"^a
\ar"0a";"1a"
\ar"1a";"2a"^c
\ar"2a";"3a"^a
\ar@{.>}"1";"1a"^{\exists\,h_{f'}}
\ar@{.>}"2";"2a"^{\exists g_{f'}}
\ar"3";"3a"^{f'}
\endxy}
\]
There exists $x:M\to N_{0}$ such that $xa=f-f'$.  Thus $(g_f-g_{f'}-ax)a=0$ so there exists $y:N_{0}\to K_{0}$ such that $yc=g_f-g_{f'}-ax$. This implies $cyc=c(g_f-g_{f'}-ax)=(h_f-h_{f'})c$, so since $c$ is a monomorphism we have $cy=h_f-h_{f'}$. Thus $h_f+I_c=h_{f'}+I_c$ holds, and we have the assertion.

It is easy to check that $\alpha$ is an $R$-algebra homomorphism since $\alpha$ is independent of the choice of $g_f$.

We now show that $\alpha$ is bijective by constructing the inverse map $\beta:\Lambda^\prime_N\to\Lambda_N$.  Let $t:K_0\to K_0$ be any morphism then on dualizing we have
\[
{\SelectTips{cm}{10}
\xy
(0,0)*+{0}="0",
(10,0)*+{M^{*}}="1",
(25,0)*+{N_{0}^{*}}="2",
(40,0)*+{K_{0}^{*}}="3",
(0,-12)*+{0}="0a",
(10,-12)*+{M^{*}}="1a",
(25,-12)*+{N_{0}^{*}}="2a",
(40,-12)*+{K_{0}^{*}}="3a",
\ar"0";"1"
\ar"1";"2"^{a^{*}}
\ar"2";"3"^{c^{*}}
\ar"0a";"1a"
\ar"1a";"2a"^{a^{*}}
\ar"2a";"3a"^{c^{*}}
\ar"3a";"3"^{t^{*}}
\ar@{.>}"2a";"2"^{\exists\, s}
\ar@{.>}"1a";"1"^{\exists\, r}
\endxy}
\]
where the rows are exact by (\ref{K0D}), $s$ exists (non-uniquely) by (\ref{K0Da}) and $r$ exists since $a^{*}$ is the kernel of $c^{*}$. Let $\beta(t+I_c):=r^*+I_a$. By the same argument as above, we have that $\beta:\Lambda^\prime_N\to\Lambda_N$ is a well-defined map.

Dualizing back gives a commutative diagram
\[
{\SelectTips{cm}{10}
\xy
(0,0)*+{0}="0",
(10,0)*+{K_{0}}="1",
(25,0)*+{N_{0}}="2",
(40,0)*+{M}="3",
(0,-12)*+{0}="0a",
(10,-12)*+{K_{0}}="1a",
(25,-12)*+{N_{0}}="2a",
(40,-12)*+{M}="3a",
\ar"0";"1"
\ar"1";"2"^c
\ar"2";"3"^a
\ar"0a";"1a"
\ar"1a";"2a"^c
\ar"2a";"3a"^a
\ar"1";"1a"^{t}
\ar@{.>}"2";"2a"^{s^{*}}
\ar@{.>}"3";"3a"^{r^{*}}
\endxy}
\]
which shows that $\beta$ is the inverse of $\alpha$.
\end{proof}

\subsection{Complete Local Case}\label{CLcase}

In this subsection we assume that $R$ is a \emph{complete local} normal Gorenstein $d$-dimensional ring, then since we have Krull--Schmidt decompositions we can say more than in the previous section.  Note that with these assumptions $R$ is automatically $d$-sCY by \ref{3.2IR}.  For a modifying module $M$ we write 
\[
M=M_1\oplus\hdots\oplus M_n=\bigoplus_{i\in I}M_{i} 
\]
as its Krull--Schmidt decomposition into indecomposable submodules, where $I=\{ 1,\hdots, n\}$. Throughout we assume that $M$ is \emph{basic}, i.e.\ the $M_i$'s are mutually non-isomorphic. With the new assumption on $R$ we may take minimal approximations and so the setup in the previous section can be simplified: for $\emptyset\neq J\subseteq I$ set $M_{J}:=\bigoplus_{j\in J}M_{j}$ and $\frac{M}{M_{J}}:=\bigoplus_{i\in I\backslash J}M_{i}$. Then 
\begin{itemize}
\item[(a)] Denote $L_0\stackrel{a}\to M_J$ to be a right $(\add\frac{M}{M_J})$-approximation of $M_J$ which is \emph{right minimal}.  If $\frac{M}{M_J}$ contains $R$ as a summand then necessarily $a$ is surjective.  
\item[(b)] Similarly denote $L_1^{*}\stackrel{b}\to M_J^{*}$ to be a right $(\add\frac{M^{*}}{M_J^{*}})$-approximation of $M_J^{*}$  which is \emph{right minimal}.  Again if $\frac{M}{M_J}$ contains $R$ as a summand then $b$ is surjective.
\end{itemize}
Recall that a morphism $a:X\to Y$ is called \emph{right minimal}
if any $f\in \End_R(X)$ satisfying $a=fa$ is an automorphism.
In what follows we denote the kernels of the above right minimal approximations by
\[
\begin{array}{ccc}
0\to C_0\stackrel{c}\to L_0\stackrel{a}\to M_J &\mbox{and}&
0\to C_1\stackrel{d}\to L_1^{*}\stackrel{b}\to M_{J}^{*}.
\end{array}
\]
This recovers the mutations from the previous subsection:
\begin{lemma}
With notation as above, $\MU{\frac{M}{M_{J}}}(M)=\frac{M}{M_{J}}\oplus C_{0}$ and $\NU{\frac{M}{M_{J}}}(M)=\frac{M}{M_{J}}\oplus C_{1}^{*}$.
\end{lemma}
\begin{proof}
There is an exact sequence 
\[
0\to C_{0}\stackrel{{}_{ (c\, 0)}}{\to} L_{0}\oplus\tfrac{M}{M_{J}}\stackrel{a\, 0\choose 0\, 1}{\to} M_{J}\oplus\tfrac{M}{M_{J}}=M\to 0
\]
with a right $(\add\frac{M}{M_J})$-approximation
$\left(\begin{smallmatrix}a&0\\ 0&1\end{smallmatrix}\right)$.
Thus the assertion follows.
\end{proof}

Since we have minimal approximations from now on we \emph{define} our mutations in terms of them. We thus define $\MU{J}(M):=C_{0}\oplus\frac{M}{M_{J}}$ and $\NU{J}(M):=C_{1}^{*}\oplus\frac{M}{M_{J}}$.   When $J=\{ i \}$ we often write $\NU{i}$ and $\MU{i}$ instead of $\NU{\{ i\}}$ and $\MU{\{ i\}}$ respectively.  Note that using this new definition of mutation involving minimal approximations,  $\MU{J}$ and $\NU{J}$ are now inverse operations up to isomorphism, not just additive closure.  This strengthens \ref{inverseoperations}.

We now investigate, in dimension three, the mutation of an MM module at an indecomposable summand $M_i$. Let $e_i$ denote the idempotent in $\Lambda:=\End_R(M)$ corresponding to the summand $M_i$, then the theory depends on whether or not $\Lambda_i:=\Lambda/\Lambda(1-e_i)\Lambda$ is artinian.

\begin{thm}\label{Aiartinian}
Suppose $R$ is complete local normal $3$-sCY and let  $M$ be an MM module with indecomposable summand $M_i$.  Denote $\Lambda=\End_R(M)$, let $e_i$ be the idempotent corresponding to $M_i$ and denote $\Lambda_i=\Lambda/\Lambda(1-e_i)\Lambda$.  If $\Lambda_i$ is artinian, then $\MU{i}(M)=\NU{i}(M)$ and this is not equal to $M$.
\end{thm}
\begin{proof}
We know that $M$, $\MU{i}(M)$ and $\NU{i}(M)$ are all MM modules by \ref{stillmodifying}, thus by \ref{gRgMRrigid}(1) it follows that $\Hom_R(M,\MU{i}(M))$ and $\Hom_R(M,\NU{i}(M))$ are both tilting $\End_R(M)$-modules.  But since $\MU{i}(M)\neq M$ and $\NU{i}(M)\neq M$ by \ref{Aifinite2},  neither of these tilting modules equal $\Hom_R(M,M)$.  Further by construction, as $\End_R(M)$-modules  $\Hom_R(M,\MU{i}(M))$ and $\Hom_R(M,\NU{i}(M))$ share all summands except possibly one, thus by a Riedtmann--Schofield type theorem  \cite[4.2]{IR} \cite[1.3]{RS}, they must coincide, i.e.\ $\Hom_R(M,\MU{i}(M))\cong\Hom_R(M,\NU{i}(M))$.  By reflexive equivalence \ref{reflequiv}(4) we deduce that $\MU{i}(M)\cong\NU{i}(M)$.
\end{proof}

The case when $\Lambda_i$ is not artinian is very different:

\begin{thm}\label{Ai_infinite2}
Suppose $R$ is complete local normal $3$-sCY and let  $M$ be a modifying module with indecomposable summand $M_i$.  Denote $\Lambda=\End_R(M)$, let $e_i$ be the idempotent corresponding to $M_i$ and denote $\Lambda_i=\Lambda/\Lambda(1-e_i)\Lambda$.  If $\Lambda_i$ is not artinian, then\\
\t{(1)} If $\pd_\Lambda\Lambda_i<\infty$, then $\MU{i}(M)=\NU{i}(M)=M$.\\
\t{(2)} If $M$ is an MM module, then always $\MU{i}(M)=\NU{i}(M)=M$.
\end{thm}
\begin{proof}
(1) It is always true that $\depth_R \Lambda_i\leq\dim_R\Lambda_i\leq \id_{\Lambda_i}\Lambda_i$ by \cite[3.5]{GN}  (see \cite[2.1]{IR}).  Since $\pd_\Lambda\Lambda_i<\infty$, by \ref{depthoffactor}(4) we know that $\id_{\Lambda_i}\Lambda_i\le 1$. Since $\Lambda_i$ is local and $\id_{\Lambda_i}\Lambda_i\le 1$, $\depth_R \Lambda_i=\id_{\Lambda_i}\Lambda_i$ by Ramras \cite[2.15]{Ram}.  If $\dim_R \Lambda_i=0$ then $\Lambda_i$ has finite length, contradicting the assumption that $\Lambda_i$ is not artinian.  Thus $\depth_R \Lambda_i= \dim_R \Lambda_i=\id_{\Lambda_i}\Lambda_i=1$.  In particular $\Lambda_i$ is a CM $R$-module of dimension 1.

Now $\Lambda$ is a Gorenstein $R$-order by \ref{2.14a} and \ref{3CY-nonlocal} so by Auslander--Buchsbaum \ref{ABlocal}, since $\pd_\Lambda\Lambda_i<\infty$ necessarily $\pd_\Lambda\Lambda_i=3-\depth_R\Lambda_{i}=2$.  Thus we have a minimal projective resolution
\begin{eqnarray}
0\to P_2\to P_1\stackrel{f}{\to}\Lambda e_{i}\to \Lambda_i\to 0.\label{minprojres}
\end{eqnarray}
where $f$ is a minimal right $(\add\Lambda(1-e_i))$-approximation
since it is a projective cover of $\Lambda(1-e_i)\Lambda e_{i}$.
By \cite[3.4(5)]{IR} we have
\[
\Ext^2_\Lambda(\Lambda_i,\Lambda)\cong \Ext^2_R(\Lambda_i,R)
\]
and this is a projective $\Lambda_i^{\rm op}$-module by \cite[1.1(3)]{GN}.  It is a free $\Lambda_i^{\rm op}$-module since $\Lambda_i$ is a local ring. Since $\Lambda_i$ is a CM $R$-module of dimension 1, we have $\Ext^2_R(\Ext^2_R(\Lambda_i,R),R)\cong\Lambda_i$ as $\Lambda_i$-modules. Thus the rank has to be one and  we have $\Ext^2_R(\Lambda_i,R)\cong\Lambda_i$ as $\Lambda_i^{\rm op}$-modules.  Applying $\Hom_\Lambda(-,\Lambda)$ to (\ref{minprojres}) gives an exact sequence
\[
\Hom_\Lambda(P_1,\Lambda) \to \Hom_\Lambda(P_2,\Lambda) \to \Lambda_i \to 0
\]
which gives a minimal projective presentation of the $\Lambda^{\rm op}$-module
$\Lambda_i$. Thus we have $\Hom_\Lambda(P_2,\Lambda)\cong e_i\Lambda$ and $P_2\cong \Lambda e_i$.

Under the equivalence $\Hom_R(M,-):\add M\to \proj\Lambda$, the sequence (\ref{minprojres}) corresponds to a complex
\[
0\to M_i\stackrel{h}{\to} L_0\stackrel{g}{\to} M_i
\]
with $g$ a minimal right $(\add \frac{M}{M_i})$-approximation.
Since the induced morphism $M_i\to \Ker g$ is sent to an isomorphism
under the reflexive equivalence $\Hom_R(M,-):\refl R\to \refl\Lambda$ (\ref{reflequiv}(4)), it is an isomorphism and so $h=\ker g$.
Consequently we have $\MU{i}(M)=\frac{M}{M_i}\oplus M_i= M$.
This implies that $\NU{i}(M)=M$ by \ref{inverseoperations}.\\
(2) This follows from (1) since $\pd_\Lambda\Lambda_i<\infty$ by \ref{almost characterize}.
\end{proof}

\begin{remark}
The above theorem needs the assumption that $M_i$ is indecomposable.  If we assume that $|J|\ge 2$ and $\Lambda_J$ is still not artinian, then both examples with $\NU{J}(M)\neq M$ and those with $\NU{J}(M)=M$ exist.  See for example \cite[\S5]{IW6} for more details.
\end{remark}

In dimension three when the base $R$ is complete local, we have the following summary, which completely characterizes mutation at an indecomposable summand.
\begin{summary}\label{main}
Suppose $R$ is complete normal $3$-sCY with MM module $M$.  Denote $\Lambda=\End_R(M)$, let $M_i$ be an indecomposable summand of $M$ and consider $\Lambda_i:=\Lambda/\Lambda(1-e_i)\Lambda$ where  $e_i$ is the idempotent in $\Lambda$ corresponding to $M_i$.  Then\\
\t{(1)}  If $\Lambda_i$ is not artinian then $\MU{i}(M)=M=\NU{i}(M)$.\\
\t{(2)}  If $\Lambda_i$ is artinian then $\MU{i}(M)=\NU{i}(M)$ and this is not equal to $M$.\\
In either case denote $\mu_{i}:=\MU{i}=\NU{i}$ then it is also true that\\
\t{(3)} $\mu_{i}\mu_{i}(M)=M$.\\
\t{(4)} $\mu_{i}(M)$ is a MM module.\\
\t{(5)} $\End_R(M)$ and $\End_R(\mu_{i}(M))$ are derived equivalent, via the tilting $\End_{R}(M)$-module $\Hom_{R}(M,\mu_{i}(M))$.
\end{summary}
\begin{proof}
(1) is \ref{Ai_infinite2} and (2) is \ref{Aiartinian}. The remainder is trivially true in the case when $\Lambda_i$ is not artinian (by \ref{Ai_infinite2}), thus we may assume that $\Lambda_i$ is artinian.  Now $\mu_{i}\mu_{i}(M)=\MU{i}(\NU{i}M)=M$ by  \ref{uptoadd}, proving (3).  (4) is contained in \ref{stillmodifying} and (5) is \ref{Aifinite1}(1).
\end{proof}

\begin{example}
Consider the subgroup $G=\frac{1}{2}(1,1,0)\oplus\frac{1}{2}(0,1,1)$ of $\SL(3,k)$ and let $R=k[[x,y,z]]^G$.  We know by \ref{skewgp} that $M=k[[x,y,z]]$ is a CT $R$-module, and in this example it decomposes into 4 summands $R\oplus M_{1}\oplus M_{2}\oplus M_{3}$ with respect to the characters of $G$.  The quiver of $\End_R(M)$ is the McKay quiver
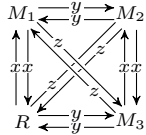
\[
\begin{tikzpicture} 
\node[name=s,regular polygon, regular polygon sides=4, minimum size=2cm] at (0,0) {}; 
\node (1) at (s.corner 1)  {$\scriptstyle M_{2}$};
\node (2) at (s.corner 2)  {$\scriptstyle M_{1}$};
\node (3) at (s.corner 3)  {$\scriptstyle R$};
\node (4) at (s.corner 4)  {$\scriptstyle M_{3}$};
\draw[->,black] (4)+(110:6.5pt) -- node[gap]{$\scriptstyle x$}  ($(1)+(-110:6.5pt)$);
\draw[<-,black] (4)+(70:6.5pt) -- node[gap]{$\scriptstyle x$}  ($(1)+(-70:6.5pt)$);
\draw[->,black] (3)+(110:6.5pt) -- node[gap]{$\scriptstyle x$}  ($(2)+(-110:6.5pt)$);
\draw[<-,black] (3)+(70:6.5pt) -- node[gap]{$\scriptstyle x$}  ($(2)+(-70:6.5pt)$);
\draw[<-,black] (2)+(-20:6.5pt) -- node[gap]{$\scriptstyle y$}  ($(1)+(-160:6.5pt)$);
\draw[->,black] (2)+(20:6.5pt) -- node[gap]{$\scriptstyle y$}  ($(1)+(-200:6.5pt)$);
\draw[<-,black] (3)+(-20:6.5pt) -- node[gap]{$\scriptstyle y$}  ($(4)+(-160:6.5pt)$);
\draw[->,black]  (3)+(20:6.5pt) -- node[gap]{$\scriptstyle y$}  ($(4)+(-200:6.5pt)$);
\draw[<-,black]  (3)+(30:6.5pt) -- node[inner sep=0.75pt,fill=white,pos=0.54] {} node[inner sep=0.75pt,fill=white,pos=0.46] {}node[inner sep=0.5pt,fill=white,pos=0.75]  {$\scriptstyle z$}  ($(1)+(-120:6.5pt)$);
\draw[]  (3)+(60:6.5pt) -- node[inner sep=0.75pt,fill=white,pos=0.54] {} node[inner sep=0.75pt,fill=white,pos=0.46] {} node[inner sep=0.5pt,fill=white,pos=0.25]  {$\scriptstyle z$} ($(1)+(-150:6.5pt)$);
\draw[<-,black]  (4)+(120:6.5pt) -- node[inner sep=0.5pt,fill=white,pos=0.75]  {$\scriptstyle z$}  ($(2)+(-30:6.5pt)$);
\draw[->,black]  (4)+(150:6.5pt) -- node[inner sep=0.5pt,fill=white,pos=0.25]  {$\scriptstyle z$}  ($(2)+(-60:6.5pt)$);
\end{tikzpicture} 
\]
and so to mutate at $M_2$ it is clear that the relevant approximation is
\[
R\oplus M_1\oplus M_3\stackrel{\left( \begin{smallmatrix} z\\ y\\x \end{smallmatrix} \right)}\to M_2\to 0
\]
Thus the mutation at vertex $M_2$ changes $M=R\oplus M_{1}\oplus M_{2}\oplus M_{3}$ into $R\oplus M_{1}\oplus K_{2}\oplus M_{3}$ where $K_2$ is the kernel of the above map which (by counting ranks) has rank 2.   On the level of quivers of the endomorphism rings, this induces the mutation
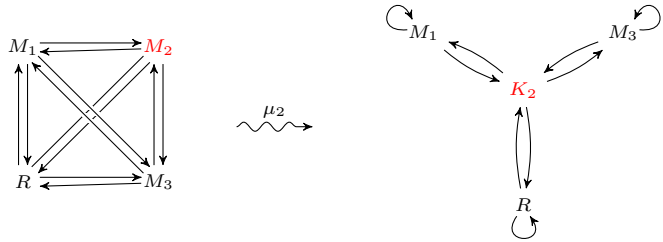
\[
\begin{array}{ccc}
\begin{array}{c}
\begin{tikzpicture} 
\node[name=s,regular polygon, regular polygon sides=4, minimum size=2.5cm] at (0,0) {}; 
\node (1) at (s.corner 1)  {\color{red}{$\scriptstyle M_{2}$}};
\node (2) at (s.corner 2)  {$\scriptstyle M_{1}$};
\node (3) at (s.corner 3)  {$\scriptstyle R$};
\node (4) at (s.corner 4)  {$\scriptstyle M_{3}$};
\draw[->,black]  (4)+(105:6.5pt) --  ($(1)+(-105:6.5pt)$);
\draw[<-,black]  (4)+(75:6.5pt) --     ($(1)+(-75:6.5pt)$);
\draw[->,black]  (3)+(105:6.5pt) --   ($(2)+(-105:6.5pt)$);
\draw[<-,black]  (3)+(75:6.5pt) --   ($(2)+(-75:6.5pt)$);
\draw[<-,black]  (2)+(-15:6.5pt) -- ($(1)+(-175:6.5pt)$);
\draw[->,black]  (2)+(15:6.5pt) --   ($(1)+(-195:6.5pt)$);
\draw[<-,black]  (3)+(-15:6.5pt) --  ($(4)+(-175:6.5pt)$);
\draw[->,black]  (3)+(15:6.5pt) --   ($(4)+(-195:6.5pt)$);
\draw[<-,black]  (3)+(30:6.5pt) -- node[inner sep=0.75pt,fill=white,pos=0.53] {} node[inner sep=0.75pt,fill=white,pos=0.47] {}  ($(1)+(-120:6.5pt)$);
\draw[]  (3)+(60:6.5pt) -- node[inner sep=0.75pt,fill=white,pos=0.53] {} node[inner sep=0.75pt,fill=white,pos=0.47] {}  ($(1)+(-150:6.5pt)$);
\draw[<-,black]  (4)+(120:6.5pt) --  ($(2)+(-30:6.5pt)$);
\draw[->,black]  (4)+(150:6.5pt) --  ($(2)+(-60:6.5pt)$);
\end{tikzpicture} \end{array}
&
\begin{array}{c}
\begin{tikzpicture} 
\draw [->,decorate, 
decoration={snake,amplitude=.6mm,segment length=3mm,post length=1mm}] 
(0,0) -- node[above] {$\scriptstyle \mu_{2}$} (1,0); 
\end{tikzpicture}
\end{array}
&
\begin{array}{c}
\begin{tikzpicture}[transform shape, rotate=0]
\node[rotate=60,name=s,regular polygon, regular polygon sides=3, minimum size=3cm] at (0,0) {}; 
\node (R) at (s.corner 2)  {$\scriptstyle R$};
\node (Ra) at ($(s.corner 2)+(-90:1pt)$)  {};
\node (1) at (s.corner 1)  {$\scriptstyle M_{1}$};
\node (1a) at ($(s.corner 1)+(-210:3pt)$) {};
\node (2) at (s.corner 3)  {$\scriptstyle M_{3}$};
\node (2a) at ($(s.corner 3)+(-330:3pt)$) {};
\node (M) at (s.center) {$\color{red}{\scriptstyle K_{2}}$};
\draw[->]  (Ra) edge [in=-55,out=-125,loop,looseness=7] (Ra);
\draw[<-]  (1a) edge [in=-175,out=-245,loop,looseness=9] (1a);
\draw[->]  (2a) edge [in=-295,out=-5,loop,looseness=9] (2a);
\draw[->]  (R) edge [in=-100,out=100,looseness=1] (M);
\draw[->]  (M) edge [in=80,out=-80,looseness=1] (R);
\draw[->]  (1) edge [in=160,out=-40,looseness=1] (M);
\draw[->]  (M) edge [in=-20,out=140,looseness=1] (1);
\draw[->]  (2) edge [in=40,out=-160,looseness=1] (M);
\draw[->]  (M) edge [in=-140,out=20,looseness=1] (2);
\end{tikzpicture}  \end{array}
\end{array}
\]
Due to the relations in the algebra $\End_R(\mu_{2}(M))$ (which we suppress), the mutation at $R$, $M_1$ and $M_3$ in the new quiver are trivial, thus in $\End_R(\mu_{2}(M))$ the only vertex we can non-trivially mutate at is $K_2$, which gives us back our original $M$.  By the symmetry of the situation we obtain the beginning of the mutation graph:
\[
\begin{array}{c}
{\xy
\POS (0,0)*+{{}_{R\oplus M_1\oplus M_2\oplus M_3}}="2", (-30,0)*+{{}_{R\oplus K_1\oplus M_2\oplus M_3}}="1", (30,0)*+{{}_{R\oplus M_1\oplus K_2\oplus M_3}}="3", (0,12)*+{{}_{R\oplus M_1\oplus M_2\oplus K_3}}="0"
\POS"1"\ar@{.}^{\mu_{1}}"2"
\POS"3"\ar@{.}_{\mu_{2}}"2"
\POS"0"\ar@{.}^{\mu_{3}}"2"
\endxy} \end{array}
\]
We remark that mutating at any of the decomposable modules $M_1\oplus M_2$, $M_1\oplus M_3$, $M_2\oplus M_3$ or $M_1\oplus M_2\oplus M_3$ gives a trivial mutation.  Note that the mutation $\MU{M/R}(M)$ at the vertex $R$ is not a CM $R$-module, and so we suppress the details.
\end{example}

\noindent
{\bf Acknowledgements.}  We would like to thank Michel Van den Bergh, Vanya Cheltsov, Constantin Shramov, Ryo Takahashi and Yuji Yoshino for stimulating discussions and valuable suggestions.  We also thank the anonymous referee for carefully reading the paper, and offering many useful insights and suggested improvements.

\end{document}